\documentclass{amsart}

\usepackage[all]{xy}
\usepackage{amsmath}
\usepackage{float}
\usepackage{hyperref}
\usepackage{amsfonts,graphics,amsthm,amsfonts,amscd,latexsym}
\usepackage{epsfig}
\usepackage{flafter}
\usepackage{mathtools}
\usepackage{comment}
\usepackage[hypcap=false]{caption}
\usepackage{stmaryrd}

\usepackage{tabularx}
\usepackage{longtable}

\usepackage{pst-node}
\usepackage{auto-pst-pdf}
\usepackage{tikz-cd}
\hypersetup{
    colorlinks=true,    
    linkcolor=blue,          
    citecolor=blue,      
    filecolor=blue,      
    urlcolor=blue           
}

\usepackage{pgfplots}
\pgfplotsset{compat=1.15}
\usepackage{mathrsfs}
\usepackage{tikz}
\usetikzlibrary{graphs,positioning,arrows,shapes.misc,decorations.pathmorphing}

\tikzset{
    >=stealth,
    every picture/.style={thick},
    graphs/every graph/.style={empty nodes},
}

\tikzstyle{vertex}=[
    draw,
    circle,
    fill=black,
    inner sep=1pt,
    minimum width=5pt,
]
\usepackage[position=top]{subfig}
\usepackage{amssymb}
\usepackage{color}

\calclayout

\usetikzlibrary{decorations.pathmorphing}
\tikzstyle{printersafe}=[decoration={snake,amplitude=0pt}]

\newcommand{\rank}{\operatorname{rank}}

\newcommand{\supp}{\operatorname{supp}}

\newcommand{\pp}{\mathbb{P}}

\newcommand{\qq}{\mathbb{Q}}
\newcommand{\zz}{\mathbb{Z}}

\newcommand{\rr}{\mathbb{R}}
\newcommand{\cc}{\mathbb{C}}

\usepackage{tikz}
\usetikzlibrary{matrix,arrows,decorations.pathmorphing}
\usetikzlibrary{arrows}

\definecolor{uuuuuu}{rgb}{0.26666666666666666,0.26666666666666666,0.26666666666666666}

  \newtheorem{introthm}{Theorem}

  \newtheorem{theorem}{Theorem}[section]
  \newtheorem{lemma}[theorem]{Lemma}
  \newtheorem{proposition}[theorem]{Proposition}
  \newtheorem{corollary}[theorem]{Corollary}

  \newtheorem{notation}[theorem]{Notation}
  \newtheorem{definition}[theorem]{Definition}
  \newtheorem{example}[theorem]{Example}

  \newtheorem{question}[theorem]{Question}

\newtheorem{remark}[theorem]{Remark}

\theoremstyle{remark}

\numberwithin{equation}{section}



\begin{document}

\title[Fundamental group of low-dimensional lc singularities]{Fundamental groups of low-dimensional lc singularities}

\author[F.Figueroa]{Fernando Figueroa}
\address{Department of Mathematics, Princeton University, Fine Hall, Washington Road, Princeton, NJ 08544-1000, USA
}
\email{fzamora@princeton.edu}

\author[J.~Moraga]{Joaqu\'in Moraga}
\address{UCLA Mathematics Department, Box 951555, Los Angeles, CA 90095-1555, USA
}
\email{jmoraga@math.ucla.edu}

\maketitle

\begin{abstract}
In this article, we study the fundamental groups of low-dimensional log canonical singularities, i.e.,
log canonical singularities of dimension at most $4$.
In dimension $2$, we show that the fundamental group of an lc singularity is a finite extension
of a solvable group of length at most $2$.
In dimension $3$, we show that every surface group appears as the fundamental group of a $3$-fold log canonical singularity.
In contrast, we show that for 
$r\geq 2$ the free group $F_r$ is not the fundamental group of a $3$-dimensional lc singularity.
In dimension $4$, we show that the fundamental group
of any $3$-manifold smoothly embedded in $\rr^4$ is the fundamental group of an lc singularity.
In particular, every free group is the fundamental group of a log canonical singularity of dimension $4$.
In order to prove the existence results, we introduce and study a special kind of polyhedral complexes: the smooth polyhedral complexes. 
We prove that the fundamental group 
of a smooth polyhedral complex of dimension $n$ appears as the fundamental group of a log canonical singularity of dimension $n+1$.
Given a $3$-manifold $M$ smoothly embedded in $\rr^4$,
we show the existence of a smooth polyhedral complex 
of dimension $3$ that is homotopic to $M$.
To do so, we start from a complex homotopic to $M$ and perform combinatorial modifications that mimic the resolution of singularities in algebraic geometry.
\end{abstract}

\setcounter{tocdepth}{1}
\tableofcontents

\section{Introduction}

The study
of singularities
is fundamental in algebraic geometry.
An approach that dates back
to the foundations
of algebraic geometry 
is to study 
the local topological structure of a singularity.
Over the complex numbers, it is known that the topology
of a sufficiently small punctured neighborhood of the algebraic singularity stabilizes (see, e.g.,~\cite{Mil68,Dur83}).
Thus, we can talk about the {\em local fundamental group} $\pi_1^{\rm loc}(X;x)$
of a singularity $(X;x)$.
By abuse of language, this is often called
the {\em fundamental group of the singularity}.
This fundamental group can allow us
to understand geometric information of the singularity.
For instance, in~\cite{Mum61}, Mumford proved that the smoothness of a normal surface singularity
is characterized by the triviality of its local fundamental group.
In \cite{Mil68}, Milnor proved that this does not hold for $3$-fold hypersurface singularities.
Furthermore, in~\cite{Gro68}, Grothendieck showed that the local fundamental group of a
local complete intersection singularity of dimension at least $3$ is trivial.

Note that the local fundamental group of an algebraic singularity 
is a finitely presented group.
Indeed, an algebraic singularity carries
the structure of a CW complex.
On the other hand, Koll\'ar and Kapovich 
showed that any finitely presented group
appears as the fundamental group
of a normal isolated $3$-fold singularity~\cite{KK14}.
This statement is not true for surface singularities due to the work of Mumford~\cite{Mum61}.
Koll\'ar and Kapovich also described the fundamental groups
of rational singularities in~\cite{KK10}
and of Cohen-Macaulay singularities in~\cite{Kol13b}.
The fundamental groups of these singularities
are closely related to
$\qq$-super-perfect groups.
Both classes of singularities: rational and Cohen-Macaulay have been a central topic in algebraic geometry for over fifty years.

Starting in the early 90's, with the development of birational geometry
and the minimal model program, 
the singularities of the MMP
attracted a lot of attention.
The singularities of the MMP
are defined by an invariant
called {\em minimal log discrepancy}.
Whenever this invariant is positive, we say that the singularity
is {\em log terminal} (also called {\em Kawamata log terminal} or simply {\em klt}).
If this invariant is non-negative, we say that the 
singularity is {\em log canonical}.
Log terminal singularities are the local analog of Fano varieties
while log canonical singularities are the local analog of Calabi--Yau varieties.
In this article, we study the local fundamental group
of log canonical singularities of dimension at most $4$.

\subsection{Log terminal singularities} 
Before turning to the main topic of this article, we recall what is known about the fundamental groups of klt singularities.
The following theorem gives a characterization of the fundamental groups of klt singularities (see~\cite[Theorem 7]{Mor21}).

\begin{introthm}\label{introthm:fund-group-klt}
Let $n\geq 2$ and $r$ be two positive integers. 
There exists a positive integer $c(n)$ only depending on $n$ satisfying the following.
Let $(X;x)$ be an $n$-dimensional klt singularity of regularity $r$.
Then, there exists a short exact sequence
\[
1\rightarrow A \rightarrow \pi_1^{\rm loc}(X;x)
\rightarrow N\rightarrow 1,
\]
where $A$ is a finite abelian group of rank at most $r+1$ and $N$ is a finite group of order at most $c(n)$.
\end{introthm}

The integer $r$, that is contained in $\{0,\dots,n-1\}$, measures the combinatorial complexity of the resolution of $(X;x)$ (see Definition~\ref{def:reg}).
The previous statement was obtained due to the work of many mathematicians~\cite{Xu14,TX17, Bra21,Mor20b,Mor20c,BFMS20}.
In a few words, the previous theorem says that the fundamental group of a log terminal singularity
behaves as the fundamental group of an orbifold singularity.

Log canonical singularities are somehow a limiting case of log terminal singularities.
Thus, a naive expectation is that the fundamental groups of lc singularities behave similarly to the fundamental group
of klt singularities.
However, the fundamental groups of log canonical singularities are still far from being understood.
Below, we summarize our results regarding lc singularities of dimension at most $4$.

\subsection{Two-dimensional log canonical singularities} In the case of dimension $2$, we can use the techniques developed by Mumford to understand the local fundamental groups.
These techniques depend on a resolution of the normal surface singularity.
These resolutions have been characterized by the work of Alexeev~\cite{Ale93}. 
Unlike klt singularities, the fundamental group of an lc singularity can be infinite starting in dimension $2$. See Example~\ref{ex:cone-over-smooth-CY}.

In the case of dimension two, we work with pairs $(X,B;x)$ and study a cover of $(X;x)$ that may ramify along $B$.
This leads to the notion of {\em regional fundamental group} denoted by $\pi_1^{\rm reg}(X,B;x)$ (see Definition~\ref{def:reg-fun}).
The fundamental group does depend on the boundary.
For instance, we have that 
\[
\pi_1\left(\mathbb{A}^2,\frac{1}{2}L_1+\frac{1}{2}L_2;(0,0) \right)\simeq \zz_2^2,
\]
where $L_1$ and $L_2$ are two transversal lines through $(0,0)$.
Our first result is an upper bound for the number of generators and relations
of the regional fundamental group of an lc surface singularity.

\begin{introthm}\label{introthm:surf}
Let $(X,B;x)$ be a log canonical surface singularity.
Then, $\pi_1^{\rm reg}(X,B;x)$ 
admits a presentation with at most 4 generators
and 
at most 7 relations.
Furthermore, if $\pi_1^{\rm reg}(X,B;x)$ admits a minimal presentation with 4 generators and 7 relations, then 
$X$ is toric at $x$ and $B$ 
has 4 components with coefficient $\frac{1}{2}$
through the singularity $x\in X$.
\end{introthm}

The previous result gives a bound on the number of generators and relations.
However, even groups with two generators
can be quite complicated.
Indeed, every finite simple group
rank at most $2$.
The following result gives a structural theorem regarding fundamental groups of surface lc singularities.

\begin{introthm}\label{introthm:solvable}
Let $(X,B;x)$ be a log canonical surface singularity.
Then, we have a short exact sequence
\[ 
1 \rightarrow N \rightarrow \pi_1^{\rm reg}(X,B;x) \rightarrow G \rightarrow 1 
\]
where $N$ is a solvable group of length at most $2$ and $G$ is a finite group of order at most $6$.
\end{introthm}

solvable groups
of length at most $2$ are somewhat analogous
to finite abelian groups of rank at most $2$.
Thus, the regional fundamental group of lc surface singularities still behaves like
the klt counterpart.
In Table~\ref{table1}, we
describe the possible isomorphism classes
of the regional fundamental groups of lc surface singularities.
For each isomorphism class, we detail the minimal resolution of $(X;x)$ and the strict transform of $B$ on the minimal resolution that leads to that group.

\subsection{Log Calabi--Yau surfaces} 
As a side product, we study the regional fundamental group of log Calabi--Yau surfaces.
In this case, we do not obtain a structural theorem, but we can bound the number of free generators of the abelianization.
This is a first step towards obtaining a version of Theorem~\ref{introthm:solvable} for the regional fundamental group of log Calabi--Yau surfaces.

\begin{introthm}\label{introthm:2-dim-lcy}
Let $(X,B)$ be a projective log Calabi-Yau pair of dimension 2.
Then, we have that 
\[
\rank(\pi_1^{\rm reg}(X,B)_{\qq}^{\rm ab})\leq 4.
\]
\end{introthm}

By Example~\ref{ex:possible-ranks}, all the possible ranks can happen in the previous theorem.
We refer the reader to Example~\ref{ex:cone-over-smooth-CY} for some conjectural statements about the regional fundamental group of log Calabi--Yau pairs.
Using some of the ideas in the proof of
Theorem~\ref{introthm:2-dim-lcy}, we will
prove the following theorem regarding the \'etale universal cover and universal cover of open Calabi--Yau surfaces~\footnote{An open Calabi--Yau surface is the smooth locus of a projective klt surface $X$ with $K_X\sim_\qq 0$.}.
The following statement is related to the work of Zhang (see, e.g.,~\cite[Proposition 4.1]{SZ01}).

\begin{introthm}\label{introthm:universal-open-k3}
Let $X$ be an open Calabi--Yau surface.
Then, one of the following statements holds:
\begin{enumerate}
    \item[(i)] The universal cover of $X$ is the complement of
    $\Lambda\mathcal{S}$ in $\cc^2$, 
    where $\Lambda$ is lattice of rank $4$ and $\mathcal{S}$ is a finite set of closed points, or 
    \item[(ii)] the \'etale universal cover of $X$ is the complement of finitely many points on the smooth locus of a K3 surface $X$ with $K_X\sim 0$.
\end{enumerate}
\end{introthm}

In the previous statement, 
the \'etale universal cover
is the cover 
associated to the pro-finite completion 
of the fundamental group.
It is not clear what should be expected for the \'etale universal covers of open Calabi--Yau $3$-folds.

\subsection{Three dimensional log canonical singularities} In dimension three, 
the fundamental groups of log canonical singularities can be much more complicated. 
In~\cite{Kol12}, Koll\'ar showed 
that for a surface group $G$
there exists a $3$-fold isolated lc singularity $(X_G;x)$ 
whose local fundamental group is a finite cyclic extension of $G$.
In particular, the local fundamental group
$\pi_1^{\rm loc}(X_G;x)$ is not a solvable group.
In this direction, we prove that surface groups
are indeed local fundamental groups
of $3$-dimensional isolated lc singularities.

\begin{introthm}\label{introthm:surf-group-in-3}
Let $S$ be a connected $2$-dimensional manifold without boundary. 
Then, there exists an isolated $3$-fold log canonical singularity $(X;x)$ for which 
$\pi_1^{\rm loc}(X;x)\simeq \pi_1(S)$.
\end{introthm}

In Example~\ref{ex:spherical-dc}, we show that many finite abelian groups of rank at most $3$ appear 
as the fundamental group of $3$-dimensional lc singularities.
As a negative result, 
we will prove that free groups
with at least $3$ generators
are not fundamental groups of isolated lc $3$-fold singularities.

\begin{introthm}\label{introthm:not-free-in-3}
No isolated $3$-fold log canonical singularity $(X;x)$ satisfy that $\pi^{\rm loc}(X;x)\simeq F_r$ with $r\geq 2$.
\end{introthm}

In particular, not every finitely presented group is the fundamental group
of an isolated lc singularity of dimension $3$.
The previous statement is closely related
to the fact that 
the fundamental group of a connected $2$-dimensional manifold without boundary is not free. In order to prove the previous statement,
we will consider the natural surjective homomorphism
$\pi_1^{\rm loc}(X;x)\rightarrow \pi_1(\mathcal{D}(X;x))$ and prove that most of the elements in the kernel are torsion.
Thus, if $\pi_1^{\rm loc}(X;x)$ was free, then it would force the fundamental group of the manifold $\mathcal{D}(X;x)$, that has dimension at most $2$, to be free. This can only happen for $S^1$ and a point.
In the first case, we use the Magnus-Karras-Solitar Theorem to deduce that $\zz$ must be a one-relator group of rank at least $2$, leading to a contradiction.
In the second case, we again use the Magnus-Karras-Solitar Theorem to conclude that there exists a smooth Calabi--Yau surface whose fundamental group is a one-relator group.
This will also lead to a contradiction.
Thus, no isolated $3$-fold lc singularity has free fundamental group with at least $3$ generators.

Theorem~\ref{introthm:surf-group-in-3} and Theorem~\ref{introthm:not-free-in-3} give new examples and constraints for the fundamental groups of lc $3$-fold singularities.
However, these do not give a complete description of the fundamental groups in dimension $3$, as we do have in dimension $2$.
In the case of $3$-dimensional lc singularities of coregularity $0$ (see Definition~\ref{def:reg}), we expect the fundamental group to be a finite cyclic extension of a surface group.
Although, it is not clear whether any such an extension can appear (see Question~\ref{quest:3-dim-extension}).

\subsection{Four-dimensional log canonical singularities} In dimension $4$,
we will show that every fundamental group of a $3$-dimensional manifold embedded in $\rr^4$ appears
as the fundamental group of an isolated lc singularity. In particular, every free group appears
as the fundamental group of an isolated $4$-dimensional lc singularity

\begin{introthm}\label{introthm:free-in-4}
Let $M$ be a connected $3$-dimensional manifold without boundary, smoothly embedded in $\rr^4$.
There exists a $4$-dimensional isolated lc singularity $(X;x)$ for which $\pi_1^{\rm loc}(X;x)\simeq \pi_1(M)$.
In particular, every free group appears as the fundamental group of an isolated $4$-dimensional lc singularity.
\end{introthm}

The last part of the previous theorem follows by considering $3$-manifold 
$M_r:=\#^r(S^2\times S^1)$ that admits a smooth embedding in $\rr^4$.
The proof of the previous theorem makes use of Koll\'ar's strategy developed in~\cite{Kol12}.
In this paper, the author constructs an $(n+1)$-dimensional lc singularity $(X;x)$ starting 
from an $n$-dimensional snc Calabi--Yau variety $T$ (possibly with many irreducible components). 
Then, they prove that the local fundamental group $\pi_1^{\rm loc}(X;x)$ will naturally surject onto $\pi_1(T)$.
Thus, in order to prove Theorem~\ref{introthm:free-in-4}, we will 
first need to prove the following statement.

\begin{introthm}\label{introthm:snc-CY-dim-3}
Let $M$ be a connected $3$-dimensional manifold without boundary, smoothly embedded in $\rr^4$.
There exists a $3$-dimensional snc Calabi--Yau variety $T$ for which
$\pi_1(T)\simeq \pi_1(M)$.
\end{introthm}

To construct the previous Calabi--Yau variety, we will use ideas emanating from toric geometry.
Indeed, our Calabi--Yau snc variety will satisfy that each irreducible component $T_i$ is a projective toric variety.
In the following subsection, we explain in more detail the idea of this construction.

\subsection{Smooth polyhedral complexes}
By means of toric geometry,
a convex polytope is associated to a projective toric variety.
Thus, in order to construct a Calabi--Yau variety $T$ for which every irreducible component is toric, we need to consider complexes whose elements are convex polyhedra
and morphisms are linear isomorphisms.
In Section~\ref{sec:pol-complex}, we associate a Calabi--Yau variety $T$
to a polyhedral complex $\mathcal{P}$ satisfying some mild conditions.
If we want the Calabi--Yau variety $T$ to be snc, then we need each polyhedron in the complex to be smooth of dimension $n$
and the nerve at each polyhedron
of the complex
to be a simplex.
In simple words, we need each vertex $v$ of the complex $\mathcal{P}$ to be contained in exactly $n+1$ maximal polyhedra.
In this direction, we prove the following theorem.

\begin{introthm}\label{introthm:from-sm-poly-comp-to-lc-sing}
Let $\mathcal{P}$ be a smooth polyhedral complex of dimension $n$. There exists an $(n+1)$-dimensional log canonical singularity $(X;x)$ for which 
$\pi_1^{\rm loc}(X;x)\simeq \pi_1(\mathcal{P})$. Furthermore, for $n \leq 4$ , the singularity $(X;x)$ is isolated. 
\end{introthm} 

The previous theorem gives us a tool to construct lc singularities with a prescribed fundamental group.
However, the category of smooth polyhedral complexes is not easy to deal with.
Naively, one can consider a smooth manifold $M$ with a triangulation $\mathcal{T}$ 
and consider the Poincar\'e dual of this triangulation. This would be a rough first approximation of a smooth polyhedral complex. 
However, the elements of the Poincar\'e dual are combinatorially convex polytopes,
so they may not be actual convex polytopes.
To remedy this issue, we will start with
the Freudenthal decomposition of $\rr^4$
and prove that its Poincar\'e dual is indeed made of convex polytopes (see Proposition~\ref{lem:dual-decomp}).
Then, the same statement will hold for subcomplexes of the Freudenthal complex as well.
There is a second difficulty that still stands: 
the Poincar\'e duals, even if they are convex polytopes, they may not be smooth.
To fix this issue we mimic the strategy of resolutions of singularities.
We define a {\em blow-up} of a polyhedral complex at a stratum. 
This construction replaces the stratum with a top-dimensional polyhedron that represents the tangent directions of the complex at the stratum (see Definition~\ref{def:blow-up-poly}). However, the blow-up is not always well-defined (see Remark~\ref{rem:blow-up-existence}).
Thus, in order to apply this strategy, we will need to deal with a case-by-case analysis. In the case of dimension $3$, we can prove the following theorem.

\begin{introthm}\label{introthm:3-manifold}
Let $M$ be a $3$-manifold
that admits a smooth embedding in $\rr^4$.
There exists a $3$-dimensional smooth polyhedral complex $\mathcal{P}_M$ that is homotopic to $M$.
Furthermore, we can choose 
each of the $3$-dimensional 
polytopes in $\mathcal{P}_M$
to be one of the following: 
\begin{itemize}
    \item quadrilateral prism,
    \item associahedron,
    \item a partial edge truncation  of a partial vertex truncation of a hexagonal prism, or
    \item a partial edge truncation of a partial vertex truncation of a permutahedron.
\end{itemize}
\end{introthm} 

In the previous theorem, truncating a face $F$ of a polyhedron $P$, means to 
replace $P$ with $P\cap H^+$ where
$H^+$ is a half-space that intersects $F$ trivially and contains all vertices of $P$ that are not contained in $F$. A {\em partial edge truncation} of $P$ is the polyhedron obtained by truncating a subset of the edges of the polyhedron $P$. Similarly, a {\em partial vertex truncation} of $P$ is the polyhedron obtained by truncating a subset of the vertices of the polyhedron $P$.

By \cite[Theorem 1]{CT88}, any closed orientable $3$-dimensional manifold admits a simple decomposition into $5$ types of polyhedra. This is a smaller family of polyhedra than the one required for our construction in Theorem~\ref{introthm:3-manifold}. This should be expected as we are imposing very strict conditions on the polyhedral complex. In forthcoming work,
we will aim to understand the resolution process of polyhedral complexes in and describe the polyhedra that are needed to approximate smooth manifolds with smooth polyhedral complexes.

The paper is organized as follows. 
In Section~\ref{subsec:prelim},
we introduce preliminaries about the singularities of the MMP
and local fundamental groups.
In Section~\ref{sec:dim2}, we study fundamental groups in complex dimension $2$ of lc singularities 
and log Calabi--Yau pairs.
In Section~\ref{sec:pol-complex}, we explain how to construct an snc CY variety from a smooth polyhedral complex, and how to construct an lc singularity from an snc CY variety.
In both cases, the constructions preserve the fundamental group.
In Section~\ref{sec:3-fold-sing},
we show that every surface group appears as the fundamental group of an lc $3$-folds singularity.
Moreover, we
show that $F_r$, with $r\geq 2$, is not the fundamental group of an isolated lc $3$-fold singularity.
In Section~\ref{sec:4-dim-lc-sing}, 
we show that every free group appears as the fundamental group of an lc $4$-dimensional singularity.
Finally, in Section~\ref{sec:e-and-q}, we give some examples and propose some questions for further research.

\subsection*{Acknowledgement}
The authors would like to thank 
Burt Totaro,
June Huh,
J\'anos Koll\'ar, 
Mirko Mauri, and 
De-Qi Zhang
for many useful comments.
Part of this work was carried out during a visit of FF to the University of Washington.
The authors wish to thank the University of Washington 
as well as Gaku Liu for his support, hospitality, and insight on polyhedral complexes.

\section{Preliminaries}
\label{subsec:prelim}

We work over the field of complex numbers $\cc$.
Given a hyperplane $H$ in $\qq^n$ defined by the equation $\sum_{i=1}^n a_ix_i=c$, 
we write $H^+:=\{(x_1,\dots,x_n)\mid \sum_{i=1}^n a_ix_i\geq c\}$
and $H^-:=\{(x_1,\dots,x_n)\mid \sum_{i=1}^n a_ix_i\leq c\}$
for the two half-spaces.
The {\em rank} of a group $G$, denoted by ${\rm rank}(G)$ is the least number of generators of $G$. 
Let $G$ be a group and $g_1,\dots,g_k\in G$ be elements.
We write $\langle g_1,\dots,g_k\rangle_n$ for the normal subgroup generated by $g_1,\dots,g_k$,
that is, the subgroup generated by all conjugates of $g_1,\dots,g_k$
in $G$. We say that this group is {\em normally generated} by the elements $g_1,\dots,g_k$.
Let $X$ be an algebraic variety and $E$ be an effective divisor. We say that a divisor $D$ on $X$ is {\em fully supported} on $E$ if $\supp(D)=\supp(E)$ holds.

\subsection{Log canonical singularities}

In this subsection, we recall the definition
of log canonical singularities. 

\begin{definition} 
{\em 
Let $X$ be a normal quasi-projective variety and $B$ be an effective $\qq$ divisor on X. A couple $(X,B)$ is said to be a {\em pair} if $K_X+B$ is $\qq$-Cartier. A pair $(X,B)$ is {\em log smooth} if $X$ is smooth and $B$ is simple normal crossing.
}
\end{definition} 

\begin{definition}
{\em 
Given a projective birational morphism $f:Y\rightarrow X$ with $Y$ normal and a pair $(X,B)$, we can define $B_Y$ the log pull-back of $B$ by the formulas 
\[
K_Y+B_Y = f^{*}(K_X+B) \quad
\text{ and } \quad  f_{*}(B_Y)=B.
\]
For a prime divisor $E\subset Y$ the {\em log discrepancy} of $(X,B)$ at E is defined to be: \[
a_{E}(X,B)=1-{\rm coeff}_E(B_Y).
\]

A pair $(X,B)$ is said to be:

\begin{enumerate}
    \item {\em log canonical}, abbreviated lc, if for every projective birational morphism $f:Y \rightarrow X$ and every prime divisor $E \subset Y$,
    we have that $a_E(X,B)\geq 0$.
    \item {\em log terminal}, also called {\em Kawamta log terminal} and abbreviated klt, if for every projective birational morphism $f:Y \rightarrow X$ and every prime divisor $E \subset Y$,
    we have that $a_E(X,B)> 0$
\end{enumerate}
}
\end{definition}

\begin{definition}
{\em 
For an lc pair $(X,B)$, an irreducible variety $Z\subset X$ is said to be a {\em log canonical center} (lc center for short) if there exists a birational morphism $f:Y\rightarrow X$ and a divisor $E\subset Y$, such that $f(E)=Z$ and $a_E(X,B)=0$.
In the previous case, the divisorial valuation $E$ is said to be a {\em log canonical place}.

For any pair $(X,B)$ there exists a largest open subset $X^{snc}\subset X$, such that
the pair $(X^{snc},B|_{X^{snc}})$ is log smooth. This locus is called the 
{\em simple normal crossing locus}
or {\em snc} locus.

A log canonical pair $(X,B)$ is {\em divisorial log terminal}, abbreviated dlt,  if all the lc centers intersect $X^{snc}$ and are given by strata of $\lfloor B\rfloor$.
}
\end{definition}

\begin{definition}
{\em 
Let $(X,B)$ be a log pair. 
Let $g\colon Y\rightarrow X$ be a 
projective birational morphism.
Let $B_Y=E+g^{-1}_*B$,
where $E$ is the reduced exceptional of $g$.
We say that $g$ is a {\em dlt modification} of $(X,B)$ if $(Y,B_Y)$ is dlt and
$K_Y+B_Y$ is g-nef.
In this case, we may also say that $(Y,B_Y)$ is a {\em dlt modification} of $(X,B)$.
}
\end{definition}

The following theorem is known as existence of dlt modifications. It is proved in~\cite[Theorem 3.1]{KK10}.

\begin{lemma}
\label{lem:dlt-mod}
Let $(X,B)$ be a log canonical pair.
Then, $(X,B)$ admits a $\qq$-factorial dlt modification $(Y,B_Y)$.
\end{lemma}

\subsection{Local and regional fundamental group} 

In this subsection, 
we recall the definition
of the local and the regional fundamental group.

\begin{notation}
{\em
All the algebraic varieties and analytic spaces considered in this paper are path-connected.
Let $X$ be an algebraic variety and $E_1,\dots, E_r$ be prime divisors on $X$.
For each $E_i$, we choose a circle $\ell_i\simeq S^1$ that is a general fiber of the circle bundle induced by the normal bundle of $E_i$ on $X$.
We choose a base point $x_0\in X$ 
and a path $r\colon [0,1]\rightarrow X$
for which $r(0)=x_0$ and $r(t_i)\in \ell_i$ for each $i$.
We write $r_i$ for the restriction of $r$ to $[0,t_i]$
By the {\em loop around the divisor} $E_i$, we mean the loop 
$r_i^{-1}\ell_ir_i$ with starting point $x_0$. 
These induce elements in $\pi_1(X;x_0)$.
}
\end{notation}

\begin{definition} 
{\em 
Let $B=\sum_{i=1}^k b_iB_i$
be a boundary divisor, 
$B_i$ its prime components,
and $b_i$ the corresponding coefficients.
We will write 
\[
B=B_s+B''.
\]
Where $B_s=\sum (1-\frac{1}{m_i})B_i$, with $m_i \in \zz_{>0}$, such that the coefficients $b_i$ of $B_i$ in $B$, satisfy $1-\frac{1}{m_i} \leq b_i< 1-\frac{1}{m_i+1}$.
The divisor $B_s$ is called the {\em standard approximation} of $X$ and any $B$ such that $B=B_s$ is said to have {\em standard coefficients}.

For a normal singularity $x \in X$,
we can embed $(X;x)$ in a smooth ambient space $\mathbb{A}^n$. 
The {\em link of $x\in X$} 
is the complement of $x$ in 
the intersection of a small euclidean ball of $x$ in $\mathbb{A}^n$ with $X$.
It is denoted by ${\rm Link}(x)$.
The {\em local fundamental group}
of $x\in X$ is defined as:
\[
\pi_1^{\rm loc}(X;x) := 
\pi_1({\rm Link}(x)).
\] 
The previous definition 
does not depend on the choice
of the smooth ambient space.
If $(X,B;x)$ is a singularity of pairs,
then we simply set $\pi_1^{\rm loc}(X,B;x):=\pi_1^{\rm loc}(X;x)$.
}
\end{definition}

\begin{definition}\label{def:reg-fun}
{\em 
Let $(X,B)$ be a log pair
and $x\in X$ be a closed point.
For each prime component $B_i$ of
$B_s$ passing through $x$, we denote by
$n_i$ the positive integer for which
\[
{\rm coeff}_{B_i}(B_s) = 1-\frac{1}{n_i}.
\]
We denote by $\gamma_i$ a loop
in the normal circle bundle of 
\[
B_i\setminus \bigcup_{j\neq i} B_j \subset 
X\setminus \supp(B_s).
\]
For simplicity,
this loop is called a {\em loop around $B_i$}.
For any open subset $U$, 
we define the group 
\[
\pi_1(U^{\rm sm},B|_U):=
\pi_1(Y\setminus \supp(B_s))/\langle \gamma_i^{n_i} \rangle_n.
\]
Here, $U^{\rm sm}$ denotes the smooth locus of $U$ and the subscript $n$ denotes the smallest
normal subgroup generated by such elements.
The {\em regional fundamental group}
of $(X,B)$ at $x$
is defined to be the inverse limit
of $\pi_1(U^{\rm sm},B|_U)$, where $U$ runs among all analytic neighborhoods of $x$ in $X$.
}
\end{definition}

\subsection{Coregularity} In this subsection, we recall the definition of coregularity. 
The coregularity is an invariant that measures the difference between the dimension of the dual complexes and the ambient variety.
It is related to log canonical thresholds~\cite{FMP22} and the theory of complements~\cite{FMM22,FFMP22}.
We refer the reader to~\cite{Mor22} for a survey on coregularity.

\begin{definition}\label{def:reg}
{\em 
Let $(X;x)$ be a log canonical singularity.
The {\em regularity} of $(X;x)$ is defined to be
\[
{\rm reg}(X;x)=\max\{\dim \mathcal{D}(Y,B_Y) \mid 
\text{ $(X,B;x)$ is log canonical
and $(Y,B_Y)$ is a dlt modification of $(X,B;x)$}
\}.
\]
In the case that $(X;x)$ is strictly log canonical with $\{x\}$ a log canonical center, 
then, in the previous definition 
we are forced to take $B=0$
and $(Y,B_Y)$ a dlt modification of $(X;x)$.

Let $(X;x)$ be a log canonical singularity.
The {\em coregularity} of $(X;x)$ is defined to be
\[
\dim X - {\rm reg}(X;x)-1.
\]
In particular, coregularity $0$ means that we can find a $0$-dimensional stratum in $\mathcal{D}(B_Y)$ where 
$(Y,B_Y)$ is the dlt modification of $(X,B)$.
}
\end{definition}

\section{Fundamental groups in dimension 2}
\label{sec:dim2}
In this section, we study fundamental groups
of log canonical surface singularities
and of Calabi--Yau surfaces.

\subsection{Surface log canonical singularities}
\label{subsec:surf-sing}
In this subsection, we describe
the regional fundamental group
of a log canonical surface singularity $(X,B;x)$.
The following theorem
is the first result of this subsection.
It gives a bound on the number of
generators and relations
of the regional fundamental group
of a log canonical surface singularity.
Throughout this section, we will get a complete classification
of the possible fundamental groups.

The main tools are the existence of dlt modifications (see Lemma~\ref{lem:dlt-mod}) and Mumford's description of the fundamental group of a normal surface singularity~\cite{Mum61}.
The following lemma follows from~\cite[Claim in Page 10]{Mum61}.

\begin{lemma}\label{lem:Mumford}
Let $X$ be a normal surface.
Let $E_1,\dots,E_n$ be curves on $X$
so that $E_1\simeq \pp^1$.
Assume each $E_i$, with $i\geq 2$,
intersects $E_1$ at $p_i$, so that
$(X,E)$ is log smooth,
where $E:=E_1+\dots+E_n$.
Define $U$ to be a small analytic
neighborhood of $E_1$ in $X$
intersected with the complement of the support of $E$.
Then, $\pi_1(U)$ is generated
by a loop $\alpha_i$ around each $E_i$ with relations
\[
\{ [\alpha_1\alpha_i]\}_{2\leq i\leq n} \quad \text{  and  }
\quad
\alpha_2\alpha_3 \dots \alpha_n \alpha_1^{E_1^2}.
\]
\end{lemma}

The following lemma gives an enhanced version of the main result of~\cite{Mum61} to the case of log pairs.
The idea is to treat the boundary as a codimension one orbifold singularity. 

\begin{lemma}\label{lem:Mumford-pairs}
Let $(X,B;x)$ be a singularity
of a $2$-dimensional log pair.
Let $f\colon (Y,B_Y)\rightarrow (X,B)$ be a log resolution
whose exceptional locus
is a tree of rational curves.
Then, the regional fundamental group
$\pi_1^{\rm reg}(X,B;x)$ is generated by:
\begin{enumerate}
    \item[(i)] a loop $\alpha_i$ around every prime exceptional divisor $E_i$ in $Y$, and 
    \item[(ii)] a loop $\gamma_i$ around every prime divisor in $f^{-1}_*(B_s)$,
\end{enumerate}
with the following relations:
\begin{enumerate}
    \item loops around intersecting divisors commute,
    \item 
    the relation
    $\alpha_1\dots \alpha_m\gamma_1 \dots \gamma_{m'}\alpha_i^{E_i^2}$, 
    for every divisor $E_i$ 
    intersecting $E_1,\dots,E_m$
    and $B_1,\dots,B_{m'}$ in
    $f^{-1}_*(B_s)$, and
    \item 
    the relation $\gamma_i^m$,
    for every prime divisor $B_i$
    of $f^{-1}_{*}(B_s)$ with coefficient $1-\frac{1}{m}$,
\end{enumerate}
 \end{lemma}

\begin{proof}
Note that as the singularity is isolated, the local and regional fundamental groups of $(X,B,x)$ coincide. Observe that a neighborhood of $x \in X$ is isomorphic to a neighborhood of the exceptional locus, 
and the coefficients in $B_s$ are the same as those in $f^{-1}_{*}(B_s)$. 
We obtain that 
\begin{equation}\label{eq:isom-groups}
\pi_1^{\rm reg}(X,B;x)=\pi_1^{\rm reg}(U,f^{-1}_{*}(B_s)|_{U})=\pi_1(U\backslash B_s)/\langle {\gamma_i^{n_i}}\rangle. 
\end{equation} 
Here, $U$ is a small analytic neighborhood of the exceptional locus in $Y$ with the exceptional locus removed.
To compute the rightmost  fundamental group in~\eqref{eq:isom-groups},
we apply the Seifert-Van Kampen Theorem to an open covering of  $U':=U\backslash B_s$.

For each exceptional divisor $E_i$,
we define $U_i$ to be $U'$ intersected with a small analytic neighborhood of $E_i$. 
Hence, we have that 
\[
U'=\bigcup_{i} U_i.
\]
As the exceptional divisors form a tree, we can apply Seifert-Van Kampen Theorem for pairs to obtain the free product of all the $\pi_1(U_i)$ modulo the amalgamation of the subgroups $\pi_1(U_i\cap U_j)$.
By Lemma~\ref{lem:Mumford}
the group $\pi_1(U_i)$ is generated by:
\begin{itemize}
    \item a loop around $E_i$, and 
    \item loops around each prime divisor (exceptional or in $f^{-1}_*B_s$) intersecting $E_i$, 
\end{itemize}
with relations 
\[
\{[\alpha_i\alpha_j]\}_j,
\quad 
\{[\alpha_i\gamma_j]\}_j, 
\quad 
\text{ and }
\quad 
\alpha_1\dots \alpha_m\gamma_1 \dots \gamma_{m'}\alpha_i^{E_i^2}.
\] 
If $E_i$ and $E_j$ intersect, then $U_i \cap U_j$ is homotopic to $(\mathbb{D}^*)^2$. So, it has fundamental group generated by loops around $E_i$ and $E_j$ that commute. Therefore, the amalgamation does not introduce any further relations between the loops around the divisors. Hence,  $\pi_1(U')$ is generated by loops around the divisors
with the relations
\[
[\alpha_i\alpha_j]_{i,j}, 
\quad 
[\alpha_i\gamma_j]_{i,j}, 
\quad  
\text{ and }
\quad 
\alpha_1\dots \alpha_m\gamma_1 \dots \gamma_{m'}\alpha_i^{E_i^2}.
\] 
Finally, we get the third relation
$\gamma_i^{m_i}$ by taking the quotient in the definition of the regional fundamental group of a pair.
\end{proof}

\begin{corollary}\label{cor:smooth-chain}
Let $(X,B;x)$ be a singularity of a 2-dimensional log pair, with $B_s=0$. Let $f:(Y,B_Y) \rightarrow (X,B)$ be a log resolution whose exceptional locus is a chain of $m \geq 2$ rational curves. Then, the regional fundamental group $\pi_1^{\rm reg}(X,B;x)$ 
is a finite cyclic group.
\end{corollary}

\begin{proof}
Call the exceptional divisors $E_1, \ldots, E_m$.
By Lemma~\ref{lem:Mumford-pairs}, we have that
\begin{align*}
\pi_1^{\rm reg}(X,B;x) \simeq \langle x_1,\ldots,x_m  \mid   x_2x_1^{-r_1},x_1x_3x_2^{-r_2},\ldots,x_{m-1}x_m^{-r_m} \rangle, 
\end{align*}
where $x_i$ is the loop around $E_i$ and $-r_i=E_i^2$. We can write
    \[\frac{b_{i}}{b_{i-1}}=[r_{i-1},r_{i-2},\ldots,r_1]\]
for the Hirzebruch-Jung continued fractions, we obtain $x_i=x_1^{b_i}$ inductively. Therefore, the following sequence of isomorphisms holds:
\begin{align*}
\pi_1^{\rm reg}(X,B;x)\simeq  \langle x_1\mid  x_1^{b_{m-1}}(x_1^{b_m})^{-r_m}\rangle \simeq \langle x_1 \mid  x_1^{r_mb_m-b_{m-1}}\rangle \simeq \zz /(r_mb_m-b_{m-1}) \zz.
\end{align*}
This implies that the regional 
fundamental group
is a finite cyclic group.
\end{proof}

\begin{lemma}\label{lem:hirzebruch-jung-group}
Consider the group
\begin{align*}
    G &:= \langle \alpha_1,\ldots,\alpha_t,x_1,\ldots,x_m \mid \alpha x_{2}x_1^{-m_1},\{x_{i-1}x_{i+1}x_{i}^{-m_i}\}_{1 <i < m},J \rangle, 
\end{align*}
where $J$ is a set of relations 
and $\alpha$ is some element generated by the $\alpha_i$'s.
Then 
\begin{align*}
   G& \simeq \langle \alpha_1,\ldots,\alpha_t,x_1 \mid J'\rangle, 
\end{align*}
where $J'$ are the relations in $J$
with $x_i$ replaced by $\alpha^{a_i}x_1^{b_i}$
for $1<i\leq m$.
\end{lemma}

\begin{proof}
We can write
   \begin{align*}
   \frac{b_{i+1}}{b_i} &= [m_i,m_{i-1},\ldots,m_1],\text{ and } \\
   \frac{a_{i+1}}{a_i} &= [m_i,m_{i-1},\ldots,m_2],
   \end{align*}
for the Hirzebruch-Jung continued fraction. Inductively, we can obtain 
$x_i=\alpha^{a_i}x_1^{b_i}$. Therefore, we obtain the required presentation of the group.
\end{proof}

In the following lemma, we will classify the possible exceptional divisors
of dlt modifications 
of log canonical surfaces singularities. 
We also prove a statement about the singular locus of the dlt modification. 
The first part of the following lemma is well-known to the experts (see, e.g.,~\cite{Ale93}).

\begin{lemma}\label{lem:dual-graph}
Let $(X,B;x)$ be a log canonical surface singularity.
Let $\phi:(Z,B_Z)\rightarrow (X,B)$ be a dlt modification.
Let $E_1,\dots,E_m$ be the exceptional prime divisors. Define $E \text{ to be } E_1+\dots+E_{m}$, $\Delta_i$ to be the different for
the adjunction of $K_Z+B_Z$ to $E_i$,
and $\hat{B}$ to be the strict transform of $B$.
Then, $E$ is one of the following:
\begin{enumerate}
    \item[(i)] a chain of rational curves,
    \item[(ii)] a cycle of rational curves, or
    \item[(iii)] an elliptic curve.
\end{enumerate}
Furthermore, the following statements hold:
\begin{enumerate}
    \item if $E$ is a chain of rational curves, then $(Z,B_Z)$ is log smooth in a neighborhood of $E_2+\dots+ E_{m-1}$,
    \item if $E$ is an elliptic curve or a cycle of rational curves, then $B_Z=E$ and 
    $(Z,B_Z)$ is log smooth
    in a neighborhood of $E$,
    \item if $m=1$ and $E$ is a rational curve, then $E_1 \cdot \hat{B}+\deg\Delta_1=2$, and
    \item if $m\geq 2$ and $E$ is a chain of rational curves, then 
    \[
    E_1 \cdot \hat{B}+\deg\Delta_1=E_m \cdot \hat{B}+\deg\Delta_m=1.
    \] 
\end{enumerate}
\end{lemma}

\begin{proof}
By the definition of $\phi$, we have that $K_Z+B_Z=\phi^*(K_X+B)$ and $(Z,B_Z)$ is a dlt pair.
Hence, $K_Z+B_Z\sim_{\qq,X} 0$.

By adjunction, we get that 
\[
0 \sim_{\qq} (K_Z+B_Z) |_{E_i}=K_{E_i}+\Delta_i+(B_Z-E_i)|_{E_i}.
\]
Therefore, $\deg{K_{E_i}}\leq 0$, so all exceptional divisors are rational curves or elliptic curves. 

If $\deg{K_{E_i}}=0$, then $\deg{(B_Z-E_i)|_{E_i}}=0$, so the vertex corresponding to $E_i$ in the dual graph of $E$ has no edges.
Since $E$ is the fiber of $\phi$ over $x$, $E$ is connected.
So, its dual graph is connected as well.
Hence, $E=E_1$ is the only exceptional divisor.
In this case, we have that  $B_Z=E_1+\hat{B}$.
Thus, by adjunction, we have that
\[
0=\deg{(K_Z+B_Z)|_{E_1}}=\deg{\Delta_1}+\deg{(B_Z-E_1)|_{E_1}}.
\]
Therefore, $E_1\cdot \hat{B}+\deg\Delta_1=0$. Consequently, in a neighborhood of $E$ the pair $(Z,B_Z)$ is log smooth and $B_Z=E$.

We are only left with the case in which $\deg{K_{E_i}}=-2$ for each $E_i$.
In this case,
we get that $\deg(B_Z-E_i)|_{E_i}\leq 2$, hence any $E_i$ intersects at most two other divisors in $E$. As $E$ is connected, it has to be a chain or a cycle of rational curves.

If $E$ is a cycle of rational curves, then by adjunction to $E_k$, we have that 
\[
0=\deg{ K_Z+B_Z |_{E_k}}=-2+\deg{\Delta_k}+2+\deg{(B_Z-E_k-E_{k-1}-E_{k+1})|_{E_k}} \geq 0.
\]
So, there are no singularities
along $E_k$.
Furthermore, $E_k$ only intersects $B_Z$ at $E_{i-1}$ and $E_{i+1}$,
and this intersection is transversal. 

If we have a chain of rational curves, then for any $E_k$ with $k \not\in \{ 1,m\}$, we have 
\[
0= \deg{ K_Z+B_Z |_{E_k}}=-2+\deg{\Delta_k}+2+\deg{(B_Z-E_k-E_{k-1}-E_{k+1})|_{E_k}} \geq 0.
\]
So, there are no singularities along $E_k$.
Furthermore, $E_k$ only intersects $B_Z$ at $E_{i-1}$ and $E_{i+1}$,
and this intersection is transversal.

If $m=1$, then $B_Z=E_1+\hat{B}$. By adjunction to $E_1$, we get that  
\[
0=\deg{(K_Z+B_Z)|_{E_1}}=-2+\deg{\Delta_1}+\deg{(B_Z-E_1)|_{E_1}}.
\] Hence, $E_1\cdot \hat{B}+\deg\Delta_1=2$.

If $m \geq 2$, then we
have the following sequence of equalities:
\[
0=\deg{(K_Z+B_Z)|_{E_1}}=-2+\deg{\Delta_1}+\deg{(B_Z-E_1)|_{E_1}}=-2+\deg{\Delta_1}+1+\deg{(B_Z-E_1-E_2)|_{E_1}}.
\]
We deduce that $E_1 \cdot \hat{B}+\deg\Delta_1=1$. We can proceed analogously for $E_m$.
This finishes the proof.
\end{proof}

\begin{proposition}\label{prop:caso-eliptic}
Let $(X,B;x)$ be a log canonical surface singularity.
Let $\phi:(Z,B_Z)\rightarrow (X,B)$ be a dlt modification whose exceptional locus $E$ is an elliptic curve. Then,
the following isomorphisms hold
\[
\pi_1^{\rm loc}(X,B;x) \simeq \pi_1^{\rm reg}(X,B;x) \simeq  \pi_1^{\rm reg}(Z,B_Z).
\]
Furthermore, $\pi_1^{\rm reg}(X,B;x)$
is of the form $\zz \rtimes \zz^2$.
\end{proposition}

\begin{proof}
By Lemma~\ref{lem:dual-graph},
we know that $E$ is a smooth elliptic curve and $B_Z$ contains no other curve.
Therefore, $\pi_1^{\rm reg}(Z,B_Z)$ is the fundamental group of an $S^1$-bundle over the elliptic curve. So, it fits in the exact sequence 
\[
1 \rightarrow \pi_1(S^1)\simeq \mathbb{Z} \rightarrow \pi_1^{\rm reg}(Z,B_Z) \rightarrow \pi_1(\mathbb{T}^2)\simeq \mathbb{Z}^2 \rightarrow 1.
\]
Hence, it has a presentation with 3 generators.
\end{proof}

\begin{proposition}\label{prop:caso-cycle-rational}
Let $(X,B;x)$ be a log canonical surface singularity.
Let $\phi:(Z,B_Z)\rightarrow (X,B)$ be a dlt modification whose exceptional locus $E$ is a cycle of rational curves. Then, we have that 
\[ \pi_1^{\rm loc}(X,B;x)
\simeq \pi_1^{\rm reg}(X,B;x)\simeq \pi_1^{\rm reg}(Z,B_Z) .
\]
Furthermore, $\pi_1^{\rm reg}(X,B;x)$
is of the form  $\zz ^2 \rtimes \zz$.
\end{proposition}

\begin{proof}
By Lemma~\ref{lem:dual-graph},
we know that $E$ is a chain of smooth rational curves and $B_Z$ contains no other curve.
Hence, by \cite[Lemma 2.3]{NW03},  $\pi_1^{\rm loc}(X,B;x)$ is of the form  $\zz ^2 \rtimes \zz$. In particular, it fits in the short exact sequence
$$1 \rightarrow \zz^2 \rightarrow \pi_1^{\rm loc}(X,B;x) \rightarrow \zz \rightarrow 1$$

\end{proof}

\begin{definition}{\em 
The coefficients of the different 
in Lemma~\ref{lem:dual-graph}
have the form $1-\frac{1}{n}$.
This follows from the adjunction formula~\cite{Kol13,Sho93}.
We will write $((n_1,n_2,\ldots,n_m))$
for the un-ordered set of fractions of the form $1-\frac{1}{n_i}$, 
allowing repetitions.
For a different $\Delta$, we will write
\[
b(\Delta):=((n_1,\dots,n_k)),
\]
for the corresponding set of non-trivial coefficients. 
We call $b(\Delta)$ the {\em basket of singularities} of the different
as it represents codimension two singularities of the ambient space.
In the case that the different is trivial, we write
$b(\Delta)=((\emptyset))$. More generally, for any divisor $D$ with standard coefficients, we will write:
\[
b(D):=((n_1,\ldots,n_k)),
\]
for the corresponding set of non-trivial coefficients. In this case, we will write $n_i=\infty$ if the coefficient is $1$.
}

\end{definition}

\begin{proposition}\label{prop:B_s=0}
Let $(X,B;x)$ be a log canonical surface singularity.
Let $\phi:(Z,B_Z)\rightarrow (X,B)$ be a dlt modification whose exceptional locus $E$ is a chain of rational curves and $B_s=0$. Then,
we have that 
\[
\pi_1^{\rm reg}(X,B;x)\simeq \pi_1^{\rm loc}(X,B;x),
\]
has a presentation with at most 3 generators and at most 3 relations.
\end{proposition}

\begin{proof}
Let $E=E_1+\dots+E_m$ be the chain of rational curves.
By Lemma~\ref{lem:dual-graph} (3),
we may assume that 
\begin{equation}\label{eq:upper-bound}
E_1\cdot \hat{B} + \deg\Delta_1=2.
\end{equation}
Here,
$\Delta_1$ is the different for the adjunction
to the exceptional divisor $E_1$.
We take a log resolution $f:Y \rightarrow Z$ of $(Z,B_Z)$, where each singular point in $E$ with $\deg\Delta_1=1-\frac{1}{n}$ has preimage a chain of rational curves with self-intersections $-m_1,\ldots,-m_l$, and $\frac{n}{q}=[m_1,\ldots,m_l]$. 
We will split the proof in two steps
depending on the number
of rational curves in the dlt modification.\\

\textit{Step 1:} We prove the statement
in the case that $E=E_1$ is a single rational curve.\\

As $E_1\cdot \hat{B}\geq 0$, 
we have that $\deg\Delta_1\leq 2$.
Recall that the coefficients of the different are of the form
$1-\frac{1}{n}$ for some positive integer $n$.
Hence, the condition
$\deg\Delta_1\leq 2$ implies that 

\begin{equation}\label{eq:subse}
b(\Delta_1) \in 
\left\{
  \begin{aligned}
((2,2,2,2)),((2,3,6)), ((2,4,4)),((3,3,3)), ((2,3,5))\\ ((2,3,4)), ((2,3,3)),
((2,2,n)),  ((n_1,n_2)),((n_2)),((\emptyset))
  \end{aligned}
\right\}
\end{equation}

In each of these cases, we  apply Lemma~\ref{lem:Mumford-pairs} to $\phi \circ f : Z \rightarrow X$:\\ 

\textit{Case 1.1:} We have $b(\Delta_1)=((2,2,2,2))$. 
Then, the following isomorphisms hold:
\begin{align*} 
\pi_1^{\rm reg}(X,B;x) &\simeq \langle a,b,c,d,x \mid xa^{-2},xb^{-2},xc^{-2},xd^{-2},abcdx^{-m}  \rangle \\& \simeq 
\langle a,b,c,d \mid a^{2}b^{-2},a^{2}c^{-2}, a^{2}d^{-2}, a^{1-2m}bcd\rangle \\& \simeq  \langle a,b,c \mid a^{2}b^{-2},a^{2}c^{-2}, a^{2}(a^{2m-1}b^{-1}c^{-1})^{-2}\rangle. 
\end{align*}

Then, in this case, we have at most 3 generators and 3 relations.\\

\textit{Case 1.2:} We have $b(\Delta_1)=((n,n',n''))$. Then, the following isomorphisms hold:
\begin{align*}
\pi_1^{\rm reg}(X,B;x)& \simeq
\left\langle 
  \begin{aligned}   
  a_1,\ldots,a_{l},b_1,\ldots,b_{l'}\\c_1,\ldots,c_{l''},x 
  \end{aligned}
\middle|
 \begin{aligned} &xa_{l-1}a_{l}^{-n},xb_{l'-1}b_{l'}^{-n'},xc_{l''-1}c_{l''}^{-n''},a_{l}b_{l'}c_{l''}x^{-m}\\& a_2a_1^{-q_1},\ldots,  a_{l}a_{l-2}a_{l-1}^{-q_{l-1}},b_2b_1^{-q'_1},\ldots,b_{l'}b_{l'-2}b_{l'-1}^{-q'_{l'-1}}\\&c_2c_1^{-q''_1},\ldots,c_{l''}c_{l''-2}c_{l''-1}^{-q''_{l''-1}}
  \end{aligned}
\right\rangle\\
& \simeq \langle a_1,b_1,c_1,x \mid xa_1^{-n},xb_1^{-n'},xc_1^{-n''},a_1^{t}b_1^{t'}c_{1}^{t''}x^{-m} \rangle \\&\simeq \langle a_1,b_1,c_1 \mid a_1^{n}b_1^{-n'},a_1^{n}c_1^{-n''},a_1^{t-mn}b_1^{t'}c_1^{t''} \rangle.
\end{align*} 
Here, the second isomorphism is a consequence of iterated applications of Lemma~\ref{lem:hirzebruch-jung-group}.
Then, in this case, we have 3 generators and 3 relations.\\

\textit{Case 1.3:} We have $b(\Delta_1)=((n,n'))$, $b(\Delta_1)=((n))$ or $b(\Delta_1)=((\emptyset))$. The exceptional divisor of $\phi \circ f:Z \rightarrow X$ is a chain of rational curves. Therefore, by Corollary~\ref{cor:smooth-chain}, we have an isomorphism:
\begin{align*}
    \pi_1^{\rm reg}(X,B;x)  & \simeq \zz /n \zz. 
\end{align*}

This finishes the proof
in the case that there is a unique
exceptional rational curve.\\
 
\textit{Step 2:} 
In this step,
we prove the statement
of the proposition when 
there are at least $2$ rational curves
in the exceptional divisor of the dlt modification.\\

In this case, the only singularities or intersections with non-exceptional divisors in $B_Y$ happen at the two end curves.
As for each of these curves,
we have that $\deg{\Delta_i}=1$, they have either
\begin{itemize}
    \item no singularities along the curve, 
    \item one orbifold singularity
    of index $n$ along the curve, or
    \item two orbifold singularities
    of index $2$ along the curve.
\end{itemize}
If the exceptional divisor of $\phi \circ f$ is a chain of rational curves, by Corollary~\ref{cor:smooth-chain}, we have the following isomorphism:
\begin{align*}
    \pi_1^{\rm reg}(X,B;x)  & \simeq \zz /n \zz. 
\end{align*}

 Depending on the singularities of $B_Y$,
we have two remaining cases.\\

\textit{Case 2.1:} Exactly one  curve has two orbifold singularities of order $2$. In this case, the exceptional divisor of $\phi\circ f:Y \rightarrow X$ is the union of a chain of rational curves and two additional rational curves with self-intersection $-2$ intersecting one end curve.
    By Lemma~\ref{lem:Mumford-pairs}, we have the following isomorphism:
    
    \begin{align*}
    \pi_1^{\rm reg}(X,B;x)&\simeq 
    \left\langle 
    \begin{aligned}  
    a,b,x_1,\ldots, x_n,
    \end{aligned}
    \middle|
    \begin{aligned}
    &x_1a^{-2},x_1b^{-2},abx_{2}x_1^{-m_1},\{x_{i-1}x_{i+1}x_{i}^{-m_i}\}_{1<i<n}\\&x_{n-1}x_n^{-m_n},\{[x_i,x_{i+1}]\}_{\{1\leq i < n-1\}}
    \end{aligned}
    \right\rangle.    
    \end{align*}
    By Lemma~\ref{lem:hirzebruch-jung-group} the group is isomorphic to:
    \begin{align*}
    \pi_1^{\rm reg}(X,B;x) &\simeq \langle a,b,x_1 \mid x_1a^{-2},x_1b^{-2},(ab)^{a_n}x_1^{b_n}((ab)^{a_{n+1}}x_1^{b_{n+1}})^{-m} \rangle\\&\simeq \langle a,b \mid a^2b^{-2},(ab)^{a_n-ma_{n+1}}a^{2b_n-2mb_{n+1}} \rangle.
    \end{align*}
    
\textit{Case 2.2:} Both end curves have two orbifold singularities of order $2$.
    In this case, the exceptional divisor of $\phi\circ f:Y \rightarrow X$ is the union of a chain of rational curves and two additional rational curves with self-interserction $-2$ intersecting each one a different end curve. By Lemma~\ref{lem:Mumford-pairs}, we have the following isomorphism:
    \begin{align*}
     \pi_1^{\rm reg}(X,B;x)&\simeq 
     \left\langle 
  \begin{aligned}  
  a,b,x_1,\ldots, x_n,c,d
  \end{aligned}
\middle|
 \begin{aligned}
  & x_1a^{-2},x_1b^{-2},abx_{2}x_1^{-m_1},\{x_{i-1}x_{i+1}x_{i}^{-m_i}\}_{1<i<n}\\ &x_{n-1}cdx_n^{-m_n},x_nc^{-2},x_nd^{-2} , \{[x_i,x_{i+1}]\}_{\{1\leq i \leq n-1\}}.
  \end{aligned}
\right\rangle
    \end{align*}
    By Lemma~\ref{lem:hirzebruch-jung-group} the group is isomorphic to:
    \begin{align*}
    \pi_1^{\rm reg}(X,B;x)&\simeq\langle a,b,x_1,c,d \mid x_1a^{-2},x_1b^{-2},(ab)^{a_{n-1}}x_1^{b_{n-1}}((ab)^{a_n}x_1^{b_n})^{-m_n}cd,(ab)^{a_n}x_1^{b_n}c^{-2},(ab)^{a_n}x_1^{b_n}d^{-2} \rangle\\&\simeq\langle a,b,c \mid a^2b^{-2},(ab)^{a_n}a^{2b_n}c^{-2},c^{2}((ab)^{a_{n-1}}a^{2b_{n-1}}c^{-2m_n+1})^{-2} \rangle. 
    \end{align*}

In each of the previous cases, we have a presentation with at most $4$ generators and at most $7$ relations 
This finishes the proof of the proposition.
\end{proof}

\begin{proposition}\label{prop:caso-chain-rational}
Let $(X,B;x)$ be a log canonical surface singularity.
Let $\phi:(Z,B_Z)\rightarrow (X,B)$ be a dlt modification, whose exceptional locus $E$ is a chain of rational curves, then $\pi_1^{\rm reg}(Z,B_Z)\simeq\pi_1^{\rm reg}(X,B;x)\simeq\pi_1^{\rm loc}(X,B;x)$ has a presentation with at most 4 generators and 7 relations.
\end{proposition}

\begin{proof}
We take a log resolution $f:Y \rightarrow Z$, where each singular point in $E$ with $\deg\Delta=1-\frac{1}{n}$ has preimage a chain rational curves with self-intersections $-m_1,\ldots,-m_l$, and $\frac{n}{q}=[m_1,\ldots,m_l]$. 
We only have to deal with the case in which ${B}_s\neq 0$.
Otherwise, we are in the context of Proposition~\ref{prop:B_s=0}.

Here $\hat{B_{Y,s}}$ and $\hat{E}$ will denote the strict transforms of $B_{Y,s}$ and $E$ in $Z$. $\hat{B_{s}}$ will denote the strict transform of $B_s$ in $Y$.
For $x,y\in \pi_1^{\rm reg}(X,B;x)$ we will define:

\begin{equation}
\delta_{t,x}(y):=
     \begin{cases}
       \{xy^{-t}\}    & \ \text{ if $y$ corresponds to a loop around an exceptional divisor of $f$}\\
  \{y^t,[x,y]\}  & \ \text{ if $y$ corresponds to a loop around a curve in $\hat{B_{Y,s}}\backslash \hat{E}$ }\\
  
 \{[x,y]\} & \ \text{ if $y$ corresponds to a loop around a curve in $\hat{B_{Y,s}}\backslash \hat{E}$ and $t= \infty$} \\
 \emptyset  & \  \text{ if $y$ corresponds to a trivial loop}\\ 

     \end{cases}
\end{equation}

\textit{Step 1:} We prove the statement
in the case that $E=E_1$ is a single rational curve.\\

By Lemma~\ref{lem:dual-graph},
we have an upper bound for the degree \[\deg(\Delta_1+\hat{B_s}|_{E_1})\leq 2.\]
We will prove this step in three cases
depending on $b(\Delta_1+\hat{B_s}|_{E_1})$.
The case  $b(\Delta_1+\hat{B_s}|_{E_1})=((2,2,2,2))$, gives rise to two different cases, depending on how many non-exceptional curves we have in $B_{Y,s}$.\\

\textit{Case 1.1:}
    We assume that $b(\Delta_1+\hat{B_s}|_{E_1})=((2,2,2,2))$
    and we have one, two, or three curves in $\hat{B_s}$.
    By Lemma~\ref{lem:Mumford-pairs}, we have the following isomorphisms:
    \begin{align*}
    \pi_1^{\rm reg}(X,B;x)
    &\simeq\langle a,b,c,d,x \mid xa^{-2},\delta_{2,x}(b),\delta_{2,x}(c),d^2,[x,d],abcdx^{-m},\rangle\\
    &\simeq\langle a,b,c,d \mid \delta_{2,a^2}(b),\delta_{2,a^2}(c),d^2,[a^2,d], a^{1-2m}bcd\rangle\\
    & \simeq \langle a,b,c \mid \delta_{2,a^2}(b),\delta_{2,a^2}(c),(a^{1-2m}bc)^2\rangle.    
    \end{align*}   
    Each $\delta_{t,x}(y)$ is at most two relations, so we have a presentation with 3 generators and at most 5 relations.\\

\textit{Case 1.2:} We assume that $b(\Delta_1+\hat{B_s}|_{E_1})=((2,2,2,2))$ and we have 4 curves in $\hat{B_{s}}$. By Lemma~\ref{lem:Mumford-pairs}, we have the following isomorphisms:
    \begin{align*}
    \pi_1^{\rm reg}(X,B;x)&\simeq\langle a,b,c,d,x \mid a^{2},b^{2},c^{2},d^{2},abcdx^{-m}, [x,a],[x,b][x,c][x,d]     \rangle\\&\simeq\langle a,b,c,x \mid a^2,b^2,c^2,(abcx^{m})^{2},[a,x],[b,x],[c,x] \rangle. 
    \end{align*}
    
    This gives a presentation with 4 generators and 7 relations.\\

\textit{Case 1.3:}
We assume that 
$b(\Delta_1+\hat{B_s}|_{E_1})$
is not equal to $((2,2,2,2))$.
In these cases, $b(\Delta_1+\hat{B_s}|{E_1})$ is of one of the following forms: $((n_1,n_2,n_3)),((n_1,n_2)),((n)),((\emptyset))$. Here, $n_i$ can be $\infty$ only if it is a coefficient in $B$. By applying Lemma~\ref{lem:Mumford-pairs}, we get the following presentation:

    \begin{align*}
     \pi_1^{\rm reg}(X,B;x)&\simeq 
     \left\langle 
  \begin{aligned}  
  A,B,C,x
  \end{aligned}
\middle|
 \begin{aligned}
  &a_{l_1}b_{l_2}c_{l_3}x^{-m},\delta_{t_1,a_{2}}(a_{1}),\delta_{t_2,b_{2}}(b_{1}),\delta_{t_3,c_{2}}(c_{1})\\& 
  \{a_{i-1}a_{i+1}a_{i}^{-m_{1,i}}\}_{\{2\leq i\leq l_1\}},\{[a_i,a_{i+1}]\}_{\{1\leq i \leq l_1 \}},
  \{b_{i-1}b_{i+1}b_{i}^{-m_{2,i}}\}_{\{2\leq i\leq l_2\}}\\
  & \{[b_i,b_{i+1}]\}_{\{1\leq i \leq l_2 \}},
    \{c_{i-1}c_{i+1}c_{i}^{-m_{3,i}}\}_{\{2\leq i\leq l_3\}},\{[c_i,c_{i+1}]\}_{\{1\leq i \leq l_3\}}
  \end{aligned}
\right\rangle.
    \end{align*}

Where $A,B,C$ is $\{a_1,\ldots,a_{l_1}\},\{b_1, \ldots, b_{l_2}\}, \{c_1,\ldots,c_{l_3}\}$, and $a_{l_1+11}=b_{l_2+1}=c_{l_3+1}=x$ in the previous presentation. Here, each $l_i$ is $0$ if $n_i$ is not present in $b(\Delta_1+\hat{B_s}|_{E_1})$.
By applying Lemma~\ref{lem:hirzebruch-jung-group} to each chain $A,B,C$, there exists a presentation of the local fundamental group: 

 \begin{align*}
    \pi_1^{\rm reg}(X,B;x)
    &\simeq\langle a_1,b_1,c_1,x \mid \delta_{n_1,x}(a_1),\delta_{n_2,x}(b_1),\delta_{n_3,x}(c_1),a_1^{t_1}b_1^{t_2}c_1^{t_3}x^{-m}\rangle. 
    \end{align*}

As each $\delta_{t,x}(y)$ contains at most two relations, we have at most 4 generators and 7 relations.
The only possible way for this to have exactly 4 generators and 7 relations is for all $\delta_{t,x}(y)$ here to be 2 elements. Hence, $a_1,b_1,c_1$ have to be loops around a curve in $\hat{B_{Y,s}}\backslash \hat{E}$. Therefore, by Lemma~\ref{lem:Mumford-pairs} we actually have the isomorphism:

 \begin{align*}
    \pi_1^{\rm reg}(X,B;x)
    &\simeq\langle a,b,c,x \mid a^{n_1},b^{n_2},c^{n_3},abcx^{-m},[a,x],[b,x],[c,x]\rangle \\
    &\simeq \langle a,b,x \mid a^{n_1},b^{n_2},(b^{-1}a^{-1}x^{m})^{n_3},[a,x],[b,x]\rangle .
    \end{align*}

Therefore, in Case 1.3., there are no
minimal presentations with 4 generators
and 7 relations.
This finishes the proof
in the case that there is a unique
exceptional rational curve in the dlt modification.\\

\textit{Step 2:} In this step, we prove the statement of the proposition when there are at least $2$ rational curves in the exceptional divisor of the dlt modification. \\ 

By Lemma~\ref{lem:dual-graph}, the only singularities in the strict transform of $E$
happen at the end curves of the chain. Also,
Lemma~\ref{lem:dual-graph} implies that
the only intersections of $\hat{E}$
with $\hat{B_{Y,s}}\backslash \hat{E}$ happen at the end curves of the chain.\\

\textit{Case 2.1:}  We assume that
$b(\Delta_m+\hat{B_s}|_{E_n})=((2,2))$
and $b(\Delta_1+\hat{B_s}|_{E_1})\in \{((n_1)),((\emptyset))\}$.
In this case, by
Lemma~\ref{lem:Mumford-pairs}, we have a presentation:

    \begin{align*}
     \pi_1^{\rm reg}(X,B;x)&\simeq 
     \left\langle 
  \begin{aligned}  
  a,b,x_1,\dots,x_n\\c_1,\ldots,c_l
  \end{aligned}
\middle|
 \begin{aligned}
  &\delta_{2,x_1}(a),\delta_{2,x_1}(b),abx_2x_1^{-m_1}, \{x_{i-1}x_{i+1}x_{i}^{-m_{i}}\}_{1<i<n},\{[x_i,x_{i+1}]\}_{\{1\leq i \leq n\}}\\ & c_lx_{n-1}x_{n}^{-m_n}, \delta_{m_{1,1},c_{2}}(c_{1}), \{c_{i-1}c_{i+1}c_{i}^{-m_{1,i}}\}_{\{2\leq i\leq l\}},\{[c_i,c_{i+1}]\}_{\{1\leq i \leq l\}}
  \end{aligned}
\right\rangle.
    \end{align*}
Where, 
$c_{l+1}=x_n$ 
in the previous presentation.
Applying Lemma~\ref{lem:hirzebruch-jung-group} to the chains $x_1,\dots,x_n$ and $c_1,\ldots,c_l$, we get an isomorphism:
\begin{align*}
\pi_1^{\rm reg}(X,B;x)&\simeq\langle a,b,x_1,c_1 \mid \delta_{2,x_1}(a),\delta_{2,x_1}(b),c_1^{t'}(ab)^{t_1}x_1^{t_2},\delta_{r,(ab)^{t_3}x_1^{t_4}}(c_{1})\rangle.
\end{align*}

As each $\delta_{t,x}(y)$ contains at most two relations, we have at most 4 generators and 7 relations.
The only possible way for this group to have exactly 4 generators and 7 relations is for all $\delta_{t,x}(y)$ 
in the right-hand side to contain 2 elements. Hence, in this case, $a,b,c_1=c_l$ are loops around a curve in $\hat{B_{Y,s}}\backslash \hat{E}$. Therefore, by Lemma~\ref{lem:Mumford-pairs}, we have the isomorphism:

\begin{align*}
\pi_1^{\rm reg}(X,B;x)&\simeq\langle a,b,x_1,c \mid a^2,[a,x_1],b^2,[b,x_1],c^m,[c,(ab^{t_3})x_1^{t_4}],\,c_1(ab)^{t_1}x_1^{t_2}\rangle\\
&\simeq \langle a,b,x_1 \mid a^2,[a,x_1],b^2,[b,x_1],((ab)^{t_1}x_1^{t_2})^m\rangle.
\end{align*}
Thus, under the assumptions of Case 2.1, there are no minimal presentations with exactly 4 generators and 7 relations.\\

\textit{Case 2.2:} We assume that
$b(\Delta_m+\hat{B_s}|_{E_n})=b(\Delta_1+\hat{B_s}|_{E_1})=((2,2))$. Here, we can apply Lemma~\ref{lem:Mumford-pairs} to get a presentation:

    \begin{align*}
     \pi_1^{\rm reg}(X,B;x)&\simeq 
     \left\langle 
  \begin{aligned}  
  a,b,x_1,\dots,x_n,c,d
  \end{aligned}
\middle|
 \begin{aligned}
  &\delta_{2,x_1}(a),\delta_{2,x_1}(b),abx_2x_1^{-m_1}, \{x_{i-1}x_{i+1}x_{i}^{-m_{i}}\}_{1<i<n}\\ & \{[x_i,x_{i+1}]\}_{\{1\leq i \leq n\}}, cdx_{n-1}x_{n}^{-m_n},\delta_{2,x_{n}}(c),\delta_{2,x_{n}}(d)
  \end{aligned}
\right\rangle.
    \end{align*}

By Lemma~\ref{lem:hirzebruch-jung-group} applied to the chain $x_1\ldots x_n$, we get an isomorphism:

    \begin{align*}
     \pi_1^{\rm reg}(X,B;x)& \simeq
     \left\langle 
  \begin{aligned}  
  a,b,x_1,c,d
  \end{aligned}
\middle|
 \begin{aligned}
  &(ab)^{a_{n-1}}x_1^{b_{n-1}}((ab)^{a_n}x_1^{b_n})^{-m_n}cd, \delta_{2,x_1}(a)\\& \delta_{2,x_1}(b), \delta_{2,(ab)^{a_n}x_1^{b_n}}(c)\delta_{2,(ab)^{a_n}x_1^{b_n}}(d)
  \end{aligned}
\right\rangle.
    \end{align*}

This case turns into two different cases.\\

\textit{Case 2.2.1:} We have one, two, or three non-exceptional curves in $\hat{B_{Y,s}}$. Then, we have the following isomorphisms due to Lemma~\ref{lem:Mumford-pairs}:
    \begin{align*}
    \pi_1^{\rm reg}(X,B;x) &\simeq
\left\langle 
   \begin{aligned}
   a,b,c,d,x_1 
   \end{aligned}
   \middle|
   \begin{aligned}
   &(ab)^{a_{n-1}}x_1^{b_{n-1}}((ab)^{a_n}x_1^{b_n})^{-m_n}cd,x_1a^{-2}\\ &\delta_{2,x_1}(b),\delta_{2,(ab)^{a_n}x_1^{b_n}}(c),d^2,[(ab)^{a_n}x_1^{b_n},d]
   \end{aligned}
\right\rangle
    \\&\simeq\langle a,b,c,d \mid(ab)^{a_{n-1}}a^{2b_{n-1}}((ab)^{a_n}a^{2b_n})^{-m_n}cd, \delta_{2,a^2}(b),\delta_{2,(ab)^{a_n}a^{2b_n}}(c),d^2,[(ab)^{a_n}a^{2b_n},d] \rangle,
    \\&\simeq\langle a,b,c, \mid \delta_{2,a^2}(b),\delta_{2,(ab)^{a_n}a^{2b_n}}(c),((ab)^{a_{n-1}}a^{2b_{n-1}}((ab)^{a_n}a^{2b_n})^{-m_n}c)^2\rangle.
    \end{align*}
    As each $\delta_{t,x}(y)$ contains at most two relations, we have a presentation with 3 generators and at most 5 relations.\\

\textit{Case 2.2.2:} We have four non-exceptional curves in the strict transform $\hat{B_{Y,s}}$. Then,
    the following isomorphisms hold
    due to Lemma~\ref{lem:Mumford-pairs}.
    \begin{align*}
    \pi_1^{\rm reg}(X,B;x)&
    \simeq 
\left\langle 
  \begin{aligned}
  a,b,c,d,x_1
  \end{aligned}
\middle|
 \begin{aligned}
 &(ab)^{a_{n-1}}x_1^{b_{n-1}}((ab)^{a_n}x_1^{b_n})^{-m_n}cd,a^{2},b^{2},c^{2},d^{2}\\ &[a,x_1],[b,x_1], [c,(ab)^{a_n}x_1^{b_n}],[d,(ab)^{a_n}x_1^{b_n}] 
  \end{aligned}
\right\rangle
    \\&\simeq\langle a,b,c,x_1  \mid a^{2},b^{2},[a,x_1],[b,x_1],c^{2},((ab)^{m_na_n-a_{n-1}}x_1^{m_nb_n-b_{n-1}}c)^{2},[c,(ab)^{a_n}x_1^{b_n}] \rangle.
    \end{align*}
Thus, we get a minimal presentation with 4 generators and 7 relations.\\

\textit{Case 2.3:}
We assume that each $b(\Delta_1+\hat{B_s}|_{E_1})$ and $b(\Delta_m+\hat{B_s}|_{E_m})$ are of the form $((n_i))$ or $((\emptyset))$. By Lemma~\ref{lem:Mumford-pairs}, we get a presentation:

    \begin{align*}
     \pi_1^{\rm reg}(X,B;x)&\simeq 
     \left\langle 
  \begin{aligned}  
  a_1\ldots a_{l_1},&\\x_1,\dots,x_n,&\\c_1,\ldots,c_{l_2}&
  \end{aligned} 
\middle|
 \begin{aligned}
  & \delta_{m_{1,1},a_{2}}(a_{1}), \{a_{i-1}a_{i+1}a_{i}^{-m_{1,i}}\}_{\{2\leq i\leq l_1\}},\{[a_i,a_{i+1}]\}_{\{1\leq i \leq l_1\}}\\& a_{l_1}x_2x_1^{-m_1}, \{x_{i-1}x_{i+1}x_{i}^{-m_{i}}\}_{1<i<n},\{[x_i,x_{i+1}]\}_{\{1\leq i \leq n\}}, c_{l_2}x_{n-1}x_{n}^{-m_n}\\& \delta_{m_{2,1},c_{2}}(c_{1}), \{c_{i-1}c_{i+1}c_{i}^{-m_{2,i}}\}_{\{2\leq i\leq l_2\}},\{[c_i,c_{i+1}]\}_{\{1\leq i \leq l_2\}}
  \end{aligned}
\right\rangle.
    \end{align*}

By applying Lemma~\ref{lem:hirzebruch-jung-group} to the chains $(a_1\ldots a_{l_1}), (x_1,\dots,x_n)$,
and $(c_1,\ldots,c_{l_2}),$ we obtain an isomorphism:

    \begin{align*}
     \pi_1^{\rm reg}(X,B;x)&\simeq 
\langle 
  a_1,x_1,c_1
\mid 
 \delta_{n_1,x_1}(a_1),c_1^{t'}(a_1)^{t_1}x_1^{t_2},\delta_{n_2,a_1^{t_3}x_1^{t_4}}(c_{1})
 \rangle.
    \end{align*}

As each $\delta_{t,x}(y)$ is at most two relations, we have at most 3 generators and 5 relations.

In each of the previous cases,
we have a presentation
with at most $4$ generators
and at most $7$ relations.
This finishes the proof of the proposition.
\end{proof}

\begin{proof}[Proof of Theorem~\ref{introthm:surf}]
First, we prove the statement about the upper bound on the number of generators and relations.
We take a dlt modification $\phi: (Z,B_Z) \rightarrow (X,B)$. By Lemma~\ref{lem:dual-graph}, the exceptional divisor of $\phi$ is one of the following:

\begin{enumerate}
    \item[(i)] a chain of rational curves,
    \item[(ii)] a cycle of rational curves, or
    \item[(iii)] an elliptic curve.
\end{enumerate}
In the first case, by Proposition~\ref{prop:B_s=0} and Proposition \ref{prop:caso-chain-rational} $\pi_1^{\rm reg}(X,B;x)$ has a presentation with at most 4 generators and at most 7  relations. 
In the second case, the statement follows by Proposition~\ref{prop:caso-cycle-rational}.
Finally, in the third case, the statement follows
by Proposition~\ref{prop:caso-eliptic}.

Now, we turn to prove the second statement of the theorem.
Assume that $\pi_1^{\rm reg}(X,B;x)$ admits a minimal presentation
with $4$ generators and $7$ relations. 
By Proposition~\ref{prop:caso-eliptic}, Proposition~\ref{prop:caso-cycle-rational}, 
and Proposition~\ref{prop:B_s=0}, Proposition~\ref{prop:caso-chain-rational} we conclude that this can only happen in the two following cases: 
\begin{itemize}
    \item the dlt modification
    $(Y,B_Y)$ of $(X,B)$ extracts a unique rational curve and $B$ has 4 components with coefficient $\frac{1}{2}$, or
    \item the dlt modification $(Y,B_Y)$ of $(X,B)$ extracts a chain of rational curve, 
    $B$ has 4 components with coefficients $\frac{1}{2}$, and their strict transforms only intersect the first and last curve of the chain.
\end{itemize}
We denote by $\phi\colon Y\rightarrow X$ the dlt modification.
Let $E$ be the reduced exceptional divisor of the dlt modification.
Then, we have that 
\[
\phi^*(K_X+B)=K_Y+E+\hat{B},
\]
where $\hat{B}$ is the strict transform of $B$ on $Y$.
By construction, the variety $Y$ is $\qq$-factorial.
We run a $(K_Y+E)$-MMP over the base.
Observe that in any of these two cases, the endpoints of the chain are $(K_Y+E)$-negative curves,
as they intersect $\hat{B}$ positively.
Hence, this MMP will 
inductively contract the endpoints of the chain of curves. Thus, this minimal model program terminates on $X$.
We conclude that $X$ is $\qq$-factorial at $x$. Therefore, the local class group ${\rm Cl}(X_x)$ is torsion.
Note that the sum of 
the coefficients of the components of $B$ through $x$,
denoted by $|B|$ equals $2$.
Altogether, we conclude that
\[
\dim X + \rank_\qq{\rm Cl}(X_x) -|B|=0.
\]
By~\cite[Theorem 2]{MS21},
we conclude that $X$ must be formally toric around the point $x$.
\end{proof} 

Now, we turn to prove a lemma that will be used in the proof of Theorem~\ref{introthm:solvable}.

\begin{lemma}\label{lem:normal-subgroup}
Let $(X,B;x)$ be a log canonical surface singularity of coregularity zero. Let $\phi: (Z,B_Z)\rightarrow (X,B)$ a dlt modification. There exists a subgroup $H < \pi_1^{\rm reg}(Z,B_Z)$ that is an abelian normal subgroup of $\pi_1^{\rm reg}(X,B;x)$.
\end{lemma}

\begin{proof}
By Lemma~\ref{lem:dual-graph}, the exceptional divisor is either a cycle or a chain of rational curves:

If it is a cycle of rational curves, then $H\cong \zz ^2$ is a normal subgroup of $ \pi_1^{\rm reg}(Z,B_Z)$, by Proposition~\ref{prop:caso-cycle-rational}. Hence, the Lemma holds. 

If it is a chain of rational curves, let $H < \pi_1^{\rm reg}(Z,B_Z)$ be the subgroup generated by the loops around the exceptional divisor of $\phi$.
Then we call $x_1,\ldots x_m$ the loops around the divisors $E_1.\ldots,E_m$.
By Lemma~\ref{lem:hirzebruch-jung-group} and Lemma~\ref{lem:Mumford-pairs}, $H=\langle x_1,x_2\rangle = \langle x_{m-1},x_m\rangle$. $H$ is abelian, since adjacent $x_i'$s commute.
Now, we turn to prove that $H$ is a normal subgroup of $\pi_1^{\rm reg}(X,B;x)$.
By the proof of Proposition~\ref{prop:B_s=0} and Proposition~\ref{prop:caso-chain-rational}, 
the remaining generators $\gamma$
of $\pi_1^{\rm reg}(X,B;x)$ 
satisfy one of the following:
\begin{itemize}
    \item they commute with $x_1$, and $\gamma x_2\gamma^{-1}=x_1^2x_2^{-1}$, whenever $b(\Delta_1+\hat{B_s}|_{E_1})=((2,2))$ holds,
    \item they commute with $x_m$, and $\gamma x_{m-1}\gamma^{-1}=x_m^2x_{m-1}^{-1}$, whenever $b(\Delta_m+\hat{B_s}|_{E_m})=((2,2))$ holds,
    \item they are generated by $x_1,x_2$, whenever   $b(\Delta_1+\hat{B_s}|_{E_1})=((n))$ holds, or 
    \item they are generated by $x_{m-1},x_m$,
    whenever
    $b(\Delta_m+\hat{B_s}|_{E_m})=((n))$ holds. 
\end{itemize}
We conclude that $H$ is normal.
\end{proof}

\begin{proof}[Proof of Theorem~\ref{introthm:solvable}]
We start with the case where $(X,B,x)$ is a log canonical singularity of coregularity one.
We take a dlt modification $\phi: (Z,B_Z) \rightarrow (X,B)$. As we are in the coregularity one case, by Lemma~\ref{lem:dual-graph} the exceptional divisor of $\phi$ is a unique rational curve or an elliptic curve.\\ 

If the exceptional is an elliptic curve, then as in the proof of Proposition~\ref{prop:caso-eliptic}, we have:
\[
1 \rightarrow \pi_1(S^1)\simeq \mathbb{Z} \rightarrow \pi_1^{\rm reg}(Z,B_Z) \rightarrow \pi_1(\mathbb{T}^2)\simeq \mathbb{Z}^2 \rightarrow 1.
\]
Therefore, $\pi_1^{\rm reg}(Z,B_Z)$ is solvable of length 2.

Now, we assume that the exceptional of the dlt modification is a rational curve. 
First, assume that $B_s=0$.
By the proof of Proposition~\ref{prop:B_s=0}, we have one of the following isomorphisms:

\begin{itemize}
\item $\pi_1^{\rm reg}(X,B;x) \simeq \langle a,b,c,d,x \mid xa^{-2},xb^{-2},xc^{-2},xd^{-2},abcdx^{-m}  \rangle. $

In this case, we can define $N=\langle ab,bc,x \rangle$. Then, we have that
\[ 1\rightarrow N \rightarrow \pi_1^{\rm reg}(X,B;x) \rightarrow  \mathbb{Z}/2\mathbb{Z} \rightarrow 1,\]
and $N$ is nilpotent of length 2 as $N/\langle x \rangle \simeq \mathbb{Z}^2$.

\item $\pi_1^{\rm reg}(X,B;x) \simeq \langle a_1,b_1,c_1,x \mid xa_1^{-n},xb_1^{-n'},xc_1^{-n''},a_1^{t}b_1^{t'}c_{1}^{t''}x^{-m} \rangle,
$
where  $b(\Delta)=((n,n',n''))$ is as in equation~\ref{eq:subse}.
We have that $x$ is in the center of $\pi_1^{\rm reg}(X,B;x)$ and we can define

\[G':=\pi_1^{\rm reg}(X,B;x)/\langle x \rangle \simeq \langle a_1,b_1,c_1 \mid a^n,b^{n'},c^{n''}, abc \rangle. \]

Where $a=a_1^{t},b=b_1^{t'},c=c_1^{t''}.$
Now, we have to check four different cases. 

\begin{enumerate}
\item If $b(\Delta)=((3,3,3))$, then $G'$ has an abelian normal subgroup $N'=\langle ab^{-1},a^{-1}b \rangle$ such that the following exact sequence holds

\[1 \rightarrow N' \rightarrow \pi_1^{\rm reg}(X,B;x)/\langle x \rangle \rightarrow \mathbb{Z}/3\mathbb{Z} \rightarrow 1.\]

Therefore, we can define $N=\langle x,a_1^tb_1^{-t'},a_1^{-t}b_1^{t'} \rangle$ and we have

\[1 \rightarrow N \rightarrow \pi_1^{\rm reg}(X,B;x) \rightarrow \mathbb{Z}/3\mathbb{Z} \rightarrow 1.\]

As
$N/\langle x \rangle \simeq \mathbb{Z}^2$, the group $N$ is nilpotent of length 2.

\item If $b(\Delta)=((2,4,4))$, then $G'$ has an abelian normal subgroup $N'=\langle bc^{-1},b^{-1}c \rangle$ such that the following exact sequence holds

\[1 \rightarrow N' \rightarrow \pi_1^{\rm reg}(X,B;x)/\langle x \rangle \rightarrow \mathbb{Z}/4\mathbb{Z} \rightarrow 1.\]

Therefore, we can define $N=\langle x,b_1^{t'}c_1^{-t''},b_1^{-t'}c_1^{t''}\rangle$ and we have

\[1 \rightarrow N \rightarrow \pi_1^{\rm reg}(X,B;x) \rightarrow \mathbb{Z}/4\mathbb{Z} \rightarrow 1.\]

As
$N/\langle x \rangle \simeq \mathbb{Z}^2$, the group $N$ is nilpotent of length 2.

\item If $b(\Delta)=((2,3,6))$, then $G'$ has an abelian normal subgroup $N'=\langle bc^{4},c^{4}b \rangle$ such that the following exact sequence holds

\[1 \rightarrow N' \rightarrow \pi_1^{\rm reg}(X,B;x)/\langle x \rangle \rightarrow \mathbb{Z}/6\mathbb{Z} \rightarrow 1.\]

Therefore, we can define $N=\langle x,b_1^{t'}c_1^{4t''},c_1^{4t''}b_1^{t'} \rangle$ and we have

\[1 \rightarrow N \rightarrow \pi_1^{\rm reg}(X,B;x) \rightarrow \mathbb{Z}/6\mathbb{Z} \rightarrow 1.\]

As
$N/\langle x \rangle \simeq \mathbb{Z}^2$, the group $N$ is nilpotent of length 2.

\item If $b(\Delta)=((n,n')), ((n)), \text{ or } ((\emptyset))$, then we have that $G'\simeq \langle a,b \mid a^n,b^{n'},ab \rangle  \simeq \mathbb{Z}/m \mathbb{Z}.$ Therefore, $\pi_1^{\rm reg}(X,B;x)$ is nilpotent of length 2. 
\end{enumerate}

\item $\pi_1^{\rm reg}(X,B;x) \simeq \mathbb{Z} /n \mathbb{Z}.$
This is already an abelian group, hence the proposition is trivial in this case.
\end{itemize}

Now assume $B_s\neq 0$.
By the proof of Proposition~\ref{prop:caso-chain-rational}, we have one of the following isomorphisms

\begin{itemize}
\item $\pi_1^{\rm reg}(X,B;x) \simeq  \langle a,b,c,d,x \mid \delta_{2,x}(a),\delta_{2,x}(b),\delta_{2,x}(c),\delta_{2,x}(d),abcdx^{-m}\rangle.$

Here  we can define $N=\langle ab,bc,x \rangle$ and we have

\[ 1\rightarrow N \rightarrow \pi_1^{\rm reg}(X,B;x) \rightarrow \mathbb{Z}/2\mathbb{Z} \rightarrow 1.\]

The group $N$ is nilpotent of length 2, as $N/\langle x \rangle \simeq \mathbb{Z}^2$.

\item $\pi_1^{\rm reg}(X,B;x) \simeq   \langle a_1,b_1,c_1,x \mid \delta_{n_1,x}(a_1),\delta_{n_2,x}(b_1),\delta_{n_3,x}(c_1),a_1^{t_1}b_1^{t_2}c_1^{t_3}x^{-m}\rangle $.

We have that $x$ is in the center of $\pi_1^{\rm reg}(X,B;x)$ and we can define

\[G':=\pi_1^{\rm reg}(X,B;x)/\langle x \rangle \simeq \langle a,b,c \mid a^n,b^{n'},c^{n''}, abc \rangle. \]

Where $a=a_1^{t},b=b_1^{t'},c=c_1^{t''}.$ Hence, the subgroup $N$ can be obtained as when $B_s=0$.\\
\end{itemize}

Now, we prove the case where $(X,B,x)$ is a log canonical singularity of coregularity zero. We take a dlt modification $\phi: (Z,B_Z) \rightarrow (X,B)$. As we are in the coregularity zero case, by Lemma \ref{lem:dual-graph} the exceptional divisor is a chain of rational curves or a cycle of rational curves.

We use the notation of Lemma~\ref{lem:normal-subgroup} for the normal subgroup $H$.
If it is a cycle of rational curves, then by Proposition~\ref{prop:caso-cycle-rational}, we have the exact sequence $1 \rightarrow H \cong \zz^2 \rightarrow  \pi_1^{\rm reg}(X,B;x) \rightarrow \zz \rightarrow 1$, hence $ \pi_1^{\rm reg}(X,B;x)$ is solvable of length 2.

If it is a chain of rational curves. First, assume $B_s=0$. By the proof Proposition~\ref{prop:B_s=0}, we have one of the following isomorphisms:

\begin{itemize}
\item $\pi_1^{\rm reg}(X,B;x) \simeq \langle a,b \mid a^2b^{-2},(ab)^{a_n-ma_{n+1}}a^{2b_n-2mb_{n+1}} \rangle.$

Here we can define $N= \langle ab, a^2 \rangle=H$, and we get:
\[ 1\rightarrow N \rightarrow \pi_1^{\rm reg}(X,B;x) \rightarrow  \mathbb{Z}/2\mathbb{Z} \rightarrow 1.\]

The group $N$ is abelian.

\item $\pi_1^{\rm reg}(X,B;x) \simeq \langle a,b,c \mid a^2b^{-2},(ab)^{a_n}a^{2b_n}c^{-2},c^{2}((ab)^{a_{n-1}}a^{2b_{n-1}}((ab)^{a_n}a^{2b_n})^{-m_n}c)^{-2} \rangle.$

Here, we can define $N= \langle ac,ab,a^2 \rangle=\langle ac,H\rangle$ and we get:
\[ 1\rightarrow N \rightarrow \pi_1^{\rm reg}(X,B;x) \rightarrow  \mathbb{Z}/2\mathbb{Z} \rightarrow 1.\]

The group $N$ is solvable of length 2 as $N/H \simeq \mathbb{Z}$

\end{itemize}
    
Now assume that $B_s \neq 0$. By the proof of Proposition~\ref{prop:caso-chain-rational}, we have one of the following isomorphisms:

\begin{itemize}
\item $\pi_1^{\rm reg}(X,B;x) \simeq \langle a,b,x_1,c_1 \mid \delta_{2,x_1}(a),\delta_{2,x_1}(b),c_1^{t'}(ab)^{t_1}x_1^{t_2},\delta_{r,(ab)^{t_3}x_1^{t_4}}(c_{1})\rangle.$

We can define $N= \langle ab, ac, x_1 \rangle=\langle ac,H\rangle$ and we get:
\[ 1\rightarrow N \rightarrow \pi_1^{\rm reg}(X,B;x) \rightarrow  \mathbb{Z}/2\mathbb{Z} \rightarrow 1.\]
The group $N$ is solvable of length 2, as $N/H \simeq \mathbb{Z}$

\item $\pi_1^{\rm reg}(X,B;x) \simeq \langle a,b,c, \mid \delta_{2,a^2}(b),\delta_{2,(ab)^{a_n}a^{2b_n}}(c),((ab)^{a_{n-1}}a^{2b_{n-1}}((ab)^{a_n}a^{2b_n})^{-m_n}c)^2\rangle.$

We can define $N= \langle ab, ac, a^2 \rangle=\langle ac,H\rangle$ and we get:
\[ 1\rightarrow N \rightarrow \pi_1^{\rm reg}(X,B;x) \rightarrow  \mathbb{Z}/2\mathbb{Z} \rightarrow 1.\]
The group $N$ is solvable of length 2, as $N/H \simeq \mathbb{Z}$

\item $\pi_1^{\rm reg}(X,B;x) \simeq \langle a,b,c,x_1  \mid a^{2},b^{2},[a,x_1],[b,x_1],c^{2},((ab)^{m_na_n-a_{n-1}}x_1^{m_nb_n-b_{n-1}}c)^{2},[c,(ab)^{a_n}x_1^{b_n}] \rangle.$

We can define $N= \langle ab, ac, x_1 \rangle=\langle ac,H\rangle$ and we get:
\[ 1\rightarrow N \rightarrow \pi_1^{\rm reg}(X,B;x) \rightarrow  \mathbb{Z}/2\mathbb{Z} \rightarrow 1.\]
$N$ is solvable of length 2, as $N/H \simeq \mathbb{Z}$

\item $\pi_1^{\rm reg}(X,B;x) \simeq \langle a_1,x_1,c_1 \mid \delta_{n_1,x_1}(a_1),c_1^{t'}(a_1)^{t_1}x_1^{t_2},\delta_{n_2,a_1^{t_3}x_1^{t_4}}(c_{1}) \rangle.$

Here  $H=\pi_1^{\rm reg}(X,B;x)$, hence the group is abelian.

\end{itemize}
\end{proof}

\subsection{Log Calabi--Yau surfaces}
\label{subsec:CY-surf}
In this subsection,
we prove some results regarding the fundamental group 
of log Calabi--Yau surfaces.
To do so, we will start with some lemmata.
The first lemma allows us to understand the regional fundamental group 
under contractions of curves.

\begin{lemma}\label{lem:MMP-vs-reg-fun}
Let $(X,B)$ be a $2$-dimensional log Calabi--Yau pair.
Let $X\rightarrow Y$
be a divisorial contraction. 
Let $B_Y$ be the push-forward of $B$ to $Y$.
Then, there is a surjective homomorphism
\[
\pi_1^{\rm reg}(Y,B_Y)\rightarrow \pi_1^{\rm reg}(X,B).
\]
\end{lemma}

\begin{proof}
Let $C\subset X$ be the exceptional curve.
If the image of $C$ on $Y$ is a smooth point $y$, then we have
\[
\pi_1^{\rm reg}(Y,B_Y)\simeq 
\pi_1^{\rm reg}(Y\setminus\{y\},B_Y)
\simeq 
\pi_1^{\rm reg}(X\setminus C, 
B|_{X\setminus C}) 
\rightarrow 
\pi_1^{\rm reg}(X,B),
\]
where the last homomorphism is surjective. 
On the other hand, if $y$ lies in the smooth locus, then we have that
\[
\pi_1^{\rm reg}(Y,B_Y)\simeq 
\pi_1^{\rm reg}(X\setminus C,B|_{X\setminus C})
\rightarrow 
\pi_1^{\rm reg}(X,B),
\]
where the last homomorphism is surjective.
This finishes the proof.
\end{proof}

The second lemma allows us to understand the regional fundamental group of log pairs
admitting fibrations.
It can be regarded as a version for pairs of~\cite[Lemma 1.5.C]{Nor83}.

\begin{lemma}\label{lem:fun-reg-vs-cbf}
Let $(X,B)$ be a log Calabi--Yau
surface.
Let $\phi\colon X\rightarrow C$ be a Mori fiber space to a curve.
Let $B_s$ be the standard approximation of $B$.
For each point $p\in C$, we denote by $n_p$ the positive integer for which 
${\rm coeff}_{\phi^{-1}(p)}(B_s)=1-n_P^{-1}$.
For each $p\in C$, we let $m_p$ be the multiplicity of the fiber over $p$. 
Consider the boundary 
$B_C=\sum_{p\in C}1-(n_pm_p)^{-1}$.
Then, there is a short exact sequence
\[
\pi_1^{\rm reg}(F,B|_F)
\rightarrow 
\pi_1^{\rm reg}(X,B)
\rightarrow 
\pi_1^{\rm reg}(C,B_C)\rightarrow 1,
\]
where $F$ is a general fiber of $\phi$.
Furthermore, the pair $(C,B_C)$ is of log Calabi--Yau type, i.e., there exists
$\Gamma_C\geq B_C$ for which $(C,\Gamma_C)$ is log Calabi--Yau.
\end{lemma}

\begin{proof}
By~\cite[Lemma 1.5.C]{Nor83}, there is a short exact sequence 
\[
\pi_1(F\setminus B_s|_F)
\rightarrow 
\pi_1(X\setminus B_s)
\rightarrow 
\pi_1(C\setminus B_C) 
\rightarrow 1. 
\]
This induces a commutative diagram
\[
\xymatrix{
\pi_1(F\setminus B_s|_F) \ar[r]^-{p_1}
\ar[d]^-{\phi_F}
&
\pi_1(X\setminus B)\ar[r]^-{p_2}
\ar[d]^-{\phi_X} 
&
\pi_1(C\setminus B_C)
\ar[d]^-{\phi_C}\ar[r]
& 
1\\
\pi_1^{\rm reg}(F,B_F)\ar[r]^-{\psi_1} 
&
\pi_1^{\rm reg}(X,B)\ar[r]^-{\psi_2}
&
\pi_1^{\rm reg}(C,B_C)\ar[r] 
&
1
}
\]
where the top row is exact
and the vertical arrows are surjective. 
We show the exactness 
at $\pi_1^{\rm reg}(X,B)$
as the surjectivity of $\psi$ 
is analogous.
Let $\beta \in \pi_1^{\rm reg}(X,B)$ an element that maps to zero in $\pi_1^{\rm reg}(C,B_C)$. 
Let $\beta_0$ be a lifting to $\pi_1(X\setminus B)$. 
Let $\beta_1$ be the image of $\beta_0$ in $\pi_1(C\setminus B_C)$.
We conclude that 
$\beta_1=\prod_{i=1}^k n_i\gamma_{p_i}^{n_im_i} n_i^{-1}$
where $\gamma_{p_i}$ is a loop
around the point $p_i\in C$.
We denote by $\gamma_i$ a loop
around $\phi^{-1}(p)$.
A local computations 
shows that the element 
$\gamma_0=\prod_{i=1}^k n_i' \gamma_i^{n_i}{n_i'}^{-1}$ maps to $\beta_1$,
where the $n_i'$ are arbitrary liftings of the $n_i$'s. 
By the exactness of the top row, there exists an element $\beta_F$ in $p_1(F\setminus B_s|_F)$ 
such that $\beta_0=p_1(\beta_F)\gamma_0$.
By construction,
$\gamma_0$ is in the kernel of $\phi_X$.
We conclude that 
$\phi_X(\beta_0)=\phi_X(p_1(\beta_F))=
\psi_1(\phi_F(\beta_F))$.
We conclude that 
$\beta'_F:=\phi_F(\beta_F)$ is a lifting to $\pi_1^{\rm reg}(F,B_F)$ of $\beta$.
This finishes the proof of the exactness.

The fact that $(C,B_C)$ is a log Calabi--Yau type pair follows from 
the canonical bundle formula.
\end{proof}

Now, we turn to prove a theorem regarding the number of free generators of the abelianization
of fundamental groups of log Calabi--Yau surfaces.

\begin{proof}[Proof of Theorem~\ref{introthm:2-dim-lcy}]
First, note that if 
$(C,B_C)$ is of log Calabi--Yau type, then
$\rank(\pi_1^{\rm reg}(C,B_C)^{\rm ab}_\qq)\leq 2$.

Let $(X,B)$ be a log Calabi--Yau surface.
We run a $K_X$-MMP
that we denote by $X\rightarrow X'$.
Let $B'$ be the push-forward of $B$ to $X'$.
By Lemma~\ref{lem:MMP-vs-reg-fun}, 
we may replace $(X,B)$ with $(X',B')$
in order to prove the statement.
Hence, we may assume that one of the following statements holds:
\begin{itemize}
    \item[(i)] the variety $X$ is Calabi--Yau and $B=0$,
    \item[(ii)] the variety $X$ admits a Mori fiber space to a curve $X\rightarrow C$, or
    \item[(iii)] the variety $X$ is Fano of Picard rank one.
\end{itemize}

Assume that $X$ is Calabi--Yau.
Let $Y\rightarrow X$ be the index one cover of $K_X$, so $Y$ is Calabi--Yau
with canonical singularities
and $K_Y\sim 0$.
It suffices to prove the statement for $Y$.
The minimal resolution of $Y$ is either an abelian surface or a K3 surface.
If the minimal resolution of $Y$ is an abelian surface, 
then $Y$ is itself an abelian surface and the statement follows.
We may replace $X$ with $Y$ and assume that its minimal resolution is a K3 surface.
Assume that 
\[
\pi_1^{\rm reg}(X)\rightarrow \pi_1^{\rm reg}(X)^{\rm ab}\simeq \zz \oplus A,
\]
where $A$ is a finitely generated abelian group.
We let $N_k:=k\zz \oplus A \leqslant \pi_1^{\rm reg}(X)$
 a normal subgroup of index $k$.
 Then, we have a short exact sequence 
\[
1\rightarrow N'_k\rightarrow \pi_1^{\rm reg}(X)\rightarrow \zz/ k\zz\rightarrow 1.
\]
We denote $G_k:=\zz/k\zz$.
Let $Z_k\rightarrow X$ be the branched cover associated to $N'_k$.
Then, $Z_k$ admits the action of $G_k$.
Let $Z'_k$ be the $G_k$-equivariant minimal resolution of $Z_k$.
Let $X'_k$ be the quotient of $Z'_k$ by $G_k$.
We have a commutative diagram
\[
\xymatrix{
Z'_k \ar[r]\ar[d] & X'\ar[d]\\
Z\ar[r] & X.
}
\]
We may assume that $Z'_k$ is a smooth K3 surface.
Otherwise, $Z$ is an abelian surface
and the statement follows.
Note that $X'\rightarrow X$ only extracts divisors with log discrepancy at most $1$. 
Hence, the projective birational morphism $X'\rightarrow X$ is crepant.
This implies that $X'$ carries a nowhere vanishing $2$-form.
Hence, the automorphism group $G_k$ must act symplectically on the smooth K3 surface $Z'_k$.
This leads to a contradiction, as a finite group acting symplectically on a smooth K3 surface has order bounded above~\cite{Muk88}.

It suffices to prove the statement for $\pi_1(Z)$.
Let $Z'\rightarrow Z$ be a minimal resolution.
Then $Z'$ is a smooth Calabi--Yau variety
and $\pi_1(Z')\simeq \pi_1(Z)$.
Thus, the statement follows from the classification of smooth Calabi-Yau surfaces.

Assume that $X$ admits a Mori fiber space
$X\rightarrow C$ to a curve.
Then, the statement follows from Lemma~\ref{lem:fun-reg-vs-cbf}
and the first line of the proof.

Assume that $X$ is a Fano variety
of Picard rank one. 
By means of contradiction, assume that 
\[
\rho \colon \pi_1^{\rm reg}(X,B)\rightarrow 
\pi_1^{\rm reg}(X,B)^{\rm ab} 
\simeq \zz^m \oplus F,
\]
where $F$ is a finite group
and $m\geq 5$. 
Let $N_k=(k\zz)^m \oplus F \leqslant \pi_1^{\rm reg}(X,B)$
a normal subgroup of index $k^m$.
Let $N'_k:=\psi^{-1}(N_k)$.
Then, we have a short exact sequence
\[
1\rightarrow 
N'_k \rightarrow 
\pi_1^{\rm reg}(X,B) 
\rightarrow 
(\zz/k\zz)^m
\rightarrow 1.
\]
We denote $G_k:=(\zz/k\zz)^m$. 
Let $Y\rightarrow X$ be the branched
cover associated to $N'_k$.
Then $Y$ admits the action
of $G_k$ and the log pull-back of $(X,B)$ to $Y$ is
a $G_k$-equivariant log Calabi--Yau pair
$(Y,B_Y)$.
By assumption, we have 
that 
\[
\rank(\pi_1^{\rm reg}(Y,B_Y)_\qq^{\rm ab})
\geq m \geq 5.
\]
Then, by Lemma~\ref{lem:MMP-vs-reg-fun}, case (i), and case (ii) above, 
we conclude that
the $G_k$-equivariant MMP $Y\rightarrow Y'$ for $K_Y$
must terminate in a Fano surface of Picard rank one.
We obtain a Fano type surface
$Y'$ that admits the action
of $(\zz/k\zz)^m$ for $k$ arbitrarily large
and $m\geq 5$.
By~\cite[Lemma 2.24]{Mor21} there exists a
$G_k$-equivariant boundary $B_{Y'}$
for which $(Y',B_{Y'})$ is log Calabi--Yau
and $N(K_{Y'}+B_{Y'})\sim 0$,
where $N$ is independent of $k$.
This contradicts~\cite[Theorem 6]{Mor21}.
We conclude that $m\leq 2$ in case (iii).
This finishes the proof.
\end{proof}

To finish this subsection, we prove the following theorem regarding the  universal cover
of the smooth locus of normal K3 surfaces.
In the following proof, $\hat{\pi}(X)$ stands for the profinite completion of $\pi_1(X)$.
This theorem is not used in the rest of the paper,
but it is interesting on its own.

\begin{proof}[Proof of Theorem~\ref{introthm:universal-open-k3}]
Let $\overline{X}$ be a Calabi--Yau surface, i.e., a surface with klt singularities and $K_{\overline{X}}\sim_\qq 0$.
Let $X$ be the smooth locus of $\overline{X}$
and $\{x_1,\dots,x_k\}$
be the singular points of $\overline{X}$.
Let $p\colon \overline{Y}\rightarrow \overline{X}$ be the index one cover of $K_{\overline{X}}$.
Then, $\overline{Y}$ is a Calabi--Yau surface with $K_{\overline{Y}}$ and canonical singularities.
By~\cite[Theorem 1.5.(2)]{GKP16},
there exists a finite cover
$q\colon \overline{Z}\rightarrow \overline{Y}$,
possibly ramified over the singular points of $\overline{Y}$, 
such that the following isomorphisms hold:
\begin{equation} 
\label{eq:isom-fund-2}
\hat{\pi}_1(\overline{Z})
\simeq 
\hat{\pi}_1(\overline{Z}^{\rm reg})
\simeq 
\hat{\pi}_1(Z),
\end{equation}
where $Z:=\overline{Z}\setminus
\{
q^{-1}(p^{-1}(x_1)),\dots,q^{-1}(p^{-1}(x_k))
\}$.
Note that $\overline{Z}$
There are two cases, depending on the minimal resolution of $\overline{Y}$.
\begin{itemize}
    \item If the minimal resolution of $\overline{Z}$ is an abelian surface, then $\overline{Z}$ is itself an abelian surface.
    Let $r\colon \cc^2\rightarrow \overline{Z}$ be the universal cover of $\overline{Z}$.
    Let $W$ be the preimage of $Z$ on $\cc^2$ with respect to $r$. 
    By the isomorphism~\eqref{eq:isom-fund-2}, we conclude that $W$ is the universal cover of $X$.
    By construction, we have that $W$ is the complement of $\Lambda\mathcal{S}\subset \cc^2$, where $\Lambda$ is a lattice of rank $4$ and $\mathcal{S}$ is a finite collection of closed points in $\cc^2$. This implies case (i).
    \item If the minimal resolution of $\overline{Z}$ is a K3 surface, 
    then $\overline{Z}$ is simply connected.
    We conclude that $\pi(\overline{Z})$ is trivial and hence
    $\hat{\pi}_1(Z)$ is trivial.
    This implies that $Z$ is the \'etale universal cover of $X$.
    In particular, the \'etale universal cover of $X$ is the complement of finitely many points on a K3 surface with canonical singularities. 
    This leads to case (ii).
\end{itemize}
\end{proof}

\section{Smooth polyhedral complexes}\label{sec:pol-complex}

In this section, we study the relationship between smooth polyhedral complexes and log canonical singularities.
In this section, we start from a smooth polyhedral complex, 
we construct an snc Calabi--Yau variety, and then we construct a log canonical singularity.
We will proceed in such a way that the fundamental groups of these objects
are isomorphic.

\subsection{Polyhedral complexes}\label{subsec:pol-complex}
In this subsection, we introduce the concept of smooth polyhedral complexes, blow-ups of polyhedral complexes, and prove a couple of lemmas.

\begin{notation}
{\em 
Let $M$ be a finitely generated free abelian group. We set $N:={\rm Hom}(M;\zz)$, the dual space of $M$. We denote by $M_\qq := M\otimes_\zz \qq$ and
$N_\qq:=N\otimes_\zz \qq$ the associated $\qq$-vector spaces. We denote by $M_\rr:=M\otimes_\zz \rr$ and $N_\rr:=N\otimes_\zz \rr$ the associated $\rr$-vector spaces.
}
\end{notation}

\begin{definition}{\em 
A {\em polyhedron} in $M_\rr $ is the convex hull of finitely many points in $M_\qq \subset M_\rr$. Equivalently, it is the closed convex bounded subset of a real vector space $M_\rr$ that is defined by finitely many inequalities over the rational numbers.

A  {\em lattice polyhedron } is a polyhedron with vertices in the lattice $M \subset M_\qq$.
Let $P$ be a lattice polyhedron with a vertex $v$. For each edge $E$ of $P$ that contains $v$, let $w_E$ be the closest lattice point to $v$ in $E$.
Then $P$ is said to be {\em smooth at the vertex} $v$, if the vectors $w_E-v$
form a basis of the lattice $\zz^N$.
A {\em lattice polyhedron} $P$ is said to be {\em smooth} if it is smooth at every vertex.
For a polyhedron $P$ with vertex $v$, the {\em cone in $P$ with vertex $v$ }is defined as the cone generated by all the rays starting at $v$ and going through any point in $P$.
}
\end{definition}

\begin{definition}
{\em 
    Let $P$ be a polyhedron in $\rr ^n$ and $P'$ be the embedding of $P$ in $\rr ^{n+1}$, with coordinates $x_0,\dots,x_n$ by taking the first coordinate $x_0$ to be zero.
    Let $H_0:=\{(x_0,\dots,x_n)\mid x_0=0\}$.
    A {\em pyramid} over  $P$ is any polyhedron with the same cell structure as the convex hull of  $P'$ and a point $p$ not in $H_0$.
    A {\em bypyramid} over $P$ is any polyhedron with the same cell structure as the convex hull of $P'$ and points $p,q$, such that $p$ and $q$ are on different sides of the hyperplane $H_0$ and have projections to $H_0$ that lie inside of $P'$.
    Therefore, an $n$-dimensional simplex can be defined to be the pyramid over an $(n-1)$-dimensional simplex.
    }
\end{definition}

\begin{definition}{\em 
Given $M_\qq$ and $M'_\qq$ two $\qq$-vector spaces
with lattices $M$ and $M'$. We say that an affine $\qq$-linear map 
$f\colon M'_\qq\rightarrow M_\qq$ is a {\em lattice embedding} if $f$ induces an isomorphism of lattices
$M'\rightarrow f(M'_\qq)\cap M$. A {\em polyhedral complex} 
is a finite category $\mathcal{P}$ satisfying the following:
\begin{enumerate}
    \item the elements are lattice polyhedra, 
    \item the morphisms are lattice embeddings, 
    \item the morphisms are inclusions of polyhedra to faces, where we consider the entire polyhedron to be a face.
    \item for every $P\in \mathcal{P}$ its faces are in $\mathcal{P}$ and the inclusion of faces are morphisms in $\mathcal{P}$, and 
    \item for every $P_1,P_2\in \mathcal{P}$, there exists at most one morphism $f$ in $\mathcal{P}$ between $P_1$ and $P_2$.
\end{enumerate}
We denote by $|\mathcal{P}|$
the amalgamation of the polyhedral complex $\mathcal{P}$. We will sometimes write $\pi_1(\mathcal{P})$ to refer to $\pi_1(|\mathcal{P}|)$.

We say that $\mathcal{P}$ is an {\em $n$-dimensional polyhedral complex} if every maximal polyhedron has the same dimension $n$. In this case, the number $n$ is called the {\em dimension} of $\mathcal{P}$.
}
\end{definition}

We will usually say intersections and inclusions of polyhedra in $\mathcal{P}$ to refer to the intersections and inclusions that happen to the corresponding cells in the amalgamation $|\mathcal{P}|$.

\begin{definition}
{\em 
   Let $\mathcal{P}$ be an $n$-dimensional polyhedral complex. We define the {\em nerve} of a cell $P \in \mathcal{P}$ to be the complex of the maximal polyhedra containing $P$.
   In particular, the nerve of $P$ will have one vertex for each maximal polyhedron in $\mathcal{P}$ containing $P$. A set of vertices in the nerve of $P$ will be in the same $k$-dimensional face if the corresponding maximal polyhedra in $\mathcal{P}$ contain a common $(n-k)$-dimensional face. 
   
   We say a $k$-dimensional polyhedron $P$ is {\em combinatorially smooth} in an $n$-dimensional complex $\mathcal{P}$ if the nerve of $P$ is an $(n-k)$-dimensional simplex. We say that $\mathcal{P}$ is a {\em simple polyhedral complex} if it is $n$-dimensional and every polyhedron $P\in \mathcal{P}$ is combinatorially smooth in $\mathcal{P}$. We say that $\mathcal{P}$ is a {\em smooth polyhedral complex} if it is simple and its objects are smooth lattice polyhedra.
   In the case that the nerve is combinatorially equivalent to a polyhedron $Q$, we may simply say that the nerve is $Q$.
   }
\end{definition}

\begin{remark} 
{\em 
Hence, a polyhedral complex can fail to be smooth at $0$-dimensional cells by being combinatorially non-smooth at the cell or non-smooth at any polyhedral containing it. While for higher dimensional cells a polyhedral complex can only fail to be smooth by combinatorial non-smoothness at the cell or non-smoothness at lower dimensional cells.
}
\end{remark}

\begin{definition}
{\em Let $\mathcal{P}$ be a polyhedral complex, We define $m\mathcal{P}$ to be the polyhedral complex, defined by the following:
\begin{itemize}
    \item for each $Q\in \mathcal{P}$ with lattice $M$. The polyhedron $mQ$, defined by the vertices of $Q$ inside the lattice $\frac{1}{m}M$, is in $m\mathcal{P}$, and 
    \item for each morphism $f:Q\rightarrow R$ in $\mathcal{P}$, the morphism $f:mQ \rightarrow mR$ is in $m\mathcal{P}$.
\end{itemize}}
\end{definition}

\begin{remark}
{\em The polyhedral complex $m\mathcal{P}$ is smooth exactly when $\mathcal{P}$ is smooth. This is because we are simply changing the lattice generators from $\{e_i\}_i$ to $\{\frac{1}{m}e_i\}_i$. }
\end{remark}

\begin{definition}
{\em 
    Let $\mathcal{P}$ be an $n$-dimensional polyhedral complex. 
    Let $v\in \mathcal{P}$ be a vertex in the complex. 
    For each polyhedra $P_i \in \mathcal{P}$
    containing $p$, we can consider the cone $C_i$ spanned by $P_i$ with vertex $v$.
    Since the morphisms in $\mathcal{P}$
    are lattice inclusions, these cones form a complex of cones $\mathcal{C}(\mathcal{P},v)$ that we call the {\em neighborhood} of $\mathcal{P}$ at $v$.
    We say that the neighborhood of a vertex $v \in \mathcal{P}$ can be {\em embedded in $\qq^n$} if for each $C_i\in \mathcal{C}(\mathcal{P},v)$ there exists a linear map $C_i \rightarrow \qq^n$ such that the images of two cones $C_i'$, $C_j'$ in $\qq^n$ are contained in each other if and only if the corresponding cones are contained in each other in $|\mathcal{P}|$.
    }
\end{definition}

The following definition is motivated
by the concept of blow-up in algebraic geometry.

\begin{definition}
\label{def:blow-up-poly}
{\em 
Let $P$ be a polyhedron inside the $n$-dimensional polyhedral complex $\mathcal{P}$. We define a {\em blow-up of $\mathcal{P}$ at $P$} to be
a polyhedral complex $\mathcal{P}'$ with the following objects:
\begin{enumerate}
    \item for any polyhedron $Q$ disjoint with $P$, the polyhedron $Q$ is in $\mathcal{P}'$
    \item for any polyhedron $R$ intersecting $P$ but not contained in $P$, there exist polyhedra $R'$ and $R_P$ in $\mathcal{P}'$, satisfying the following: 
    \begin{enumerate}
        \item Let $H_R$ be a hyperplane $H_R$ separating $P \cap R \subset R$ and the vertices of $R \backslash P \cap R $. Let $H$ be the half-space defined by $H_R$ that does not contain $P$.
        The polyhedron $R'$ is the intersection of $R$ and $H$.
        \item $R_P$ is the intersection of $R$ and the hyperplane $H_R$.

    \end{enumerate}
    \item A polyhedron $P'$ whose faces are the polyhedron $R_P$ for each polyhedron $R$ intersecting $P$, not contained in $P$.
\end{enumerate}
Moreover, the morphisms are defined by the following rules:
\begin{enumerate}
    \item For any embedding $q:Q_1 \rightarrow Q_2$, where $Q_1$ is disjoint with P, the embedding $q$ is in $\mathcal{P'}$.
    \item  For any embedding $r:R_1 \rightarrow R_2$ between polyhedra that intersect $P$ but are not contained in $P$, the restrictions of the embedding $r:R_1 \rightarrow R_2$ to ${R_1}_P \rightarrow {R_2}_P$ and $R_1' \rightarrow R_2'$ are in $\mathcal{P}'$.
    \item For any $R$ intersecting $P$ but not contained in $P$, we have the morphism $i:R_P \rightarrow P'$.
\end{enumerate}
}
\end{definition}

\begin{remark}\label{rem:faces-blow-up}
{\em 
The maximal polyhedra of a blow-up are hence: 
\begin{enumerate}
    \item The polyhedron $Q$ for any maximal polyhedra $Q$ not intersecting $P$
    \item The polyhedron $R'$ for any maximal polyhedra $R$ intersecting $P$
    \item The polyhedron $P'$, which is combinatorially equivalent to the dual polyhedron of the nerve of $P$ in $\mathcal{P}$.
\end{enumerate}
}
\end{remark}

\begin{remark}\label{rem:blow-up-existence}
{\em 
A priori, it is not clear that a blow-up always exists, as we require the existence of the rational polyhedron $P'$.
}
\end{remark}

\begin{definition}
{\em 
    A lattice fan in $\zz^n$ is a {\em strongly polytopal fan} if the lattice generators of each ray are the vertices of a convex polytope.
    }
\end{definition}

\begin{lemma}\label{lem:blow-up-dim-3}
Let $\mathcal{P}$ be a 3-dimensional polyhedral complex.
Let $v\in \mathcal{P}$ be a vertex such that all polyhedra $P_i$ containing $v$ are smooth at $v$.
Assume that there is an embedding 
$\phi\colon \mathcal{C}(\mathcal{P},v)\rightarrow \qq^3$ as a strongly polytopal fan. 
Then, there exists a blow-up of $2\mathcal{P}$ at $v$.
\end{lemma}

\begin{proof}
Let $\mathcal{Q}:=\phi(\mathcal{C}(\mathcal{P},v))$ be the embedded fan of cones with vertex $v$ in $\qq^3$. 
Without loss of generality, let us assume that $\phi(v)$ is the origin. This subcomplex gives us a 3-dimensional complete fan around the origin, generated by rays $r_1, \ldots, r_i$, with lattice generators $e_1, \ldots, e_i$. 
Since all polyhedra are smooth at $v$, the cones are smooth. Hence, each point in the fan can be uniquely written as a linear combination of the generators in the cone that contains it. 
Therefore, taking $f_j:=(e_j,1) \in \qq^4$ for each generator, defines a piecewise linear map from $\qq^3$ to a convex cone $\sigma$ in $\qq^4$.
Any polyhedron in $\mathcal{P}$ containing $p$ is a subset of one of the cones with vertex $v$. Hence, they also are a face of $\sigma$.
Let $H$ be the hyperplane $\{2x_4=1\}$.
This hyperplane separates $(0,0,0,0)$ with all the other vertices of polyhedra in $\mathcal{Q}$.
We replace the lattice $\zz^4\subset \qq^4$ with $(\frac{1}{2}\zz)^4$ and correspondingly in all the polyhedra of the complex.
We perform the following replacements:
\begin{itemize}
    \item We replace any polyhedron containing $v$ in $\mathcal{Q}$ with its intersection with $H^+$, and
    \item we replace $v$ with the intersection of $H$ and the cone $\sigma$ in $\qq^4$.
\end{itemize}
The polyhedral complex $\mathcal{P}'$ obtained by performing these replacements
is a blow-up of $2\mathcal{P}$ at $v$.
\end{proof}

\begin{lemma}\label{lem:nerve-blowup}
    Let $\mathcal{P'}$ be a blow-up of $\mathcal{P}$ at $P$. Let $R\in \mathcal{P}$ be a polyhedron intersecting $P$, but not contained in $P$.
    If the nerve at $R\in \mathcal{P}$ is $Q$, then the nerve at $R_P\in \mathcal{P}'$ is a pyramid over $Q$.
\end{lemma}

\begin{proof}
    All maximal dimensional polyhedra $R_i$ that contain $R$ must also intersect $P$. Hence, in $\mathcal{P}'$ the maximal dimensional polyhedra that contain $R_P$ are the polyhedra $R_i'$ and the polyhedron $P'$. The polyhedra $R_i'$ and $R_j'$ intersect non-trivially whenever $R_i$ and $R_j$ do. The polyhedron $P'$ intersects non-trivially all the polyhedra $R_i$ by the definition of  blow-up. Hence, the nerve of $R_P$ is the pyramid over $Q$ with vertex corresponding to the polyhedron $P'$. 
\end{proof}

\subsection{Projective toric varieties and lattice isomorphisms.}\label{subsec:lattice-isomorphisms}
In this subsection, we prove a result regarding embeddings of projective toric varieties and lattice embeddings.

\begin{lemma}\label{lem:toric-translation}
Let $P$ be a full-dimensional lattice polyhedron in $M_{\qq}$. Let $X_P$ be the associated projective variety.
Let $k$ be a vector in $M_{\qq}$. Let $A$ be the ample line bundle on $X_P$ associated to $P$ and $A'$ be the ample line bundle on $X_{P}$ associated to $P+k$. Then, we have that $A \sim_{\qq} A'$. Furthermore, if $k$ is a lattice vector, then $A \sim A'$.
\end{lemma}

\begin{proof}
Let $t$ be such that $tk  \in M_{\zz}$.
The lattice polyhedron $P$ is given by the inequalities:
\[
\{ m \mid \langle m,u_F \rangle \geq -a_F\}.
\]
Here, $F$ are the facets of $P$ and $u_F\in N$ is the inward pointing normal vector of the face $F$.
For each face $F$ of $P$, we have an associated prime torus invariant divisor $D_F$ on $X_P$~\cite[Definition 2.3.14]{CLS11}.
By~\cite[Proposition 4.2.10]{CLS11}, we have that 
\[
A=\sum_{F}
a_F D_F,
\]
where the sum runs over all the faces $F$ of $P$.
Observe that 

\begin{align*}
P' &=\{ m+k \mid \langle m,u_F \rangle \geq -a_F\}=\{ m \mid \langle m-k,u_F \rangle  \geq -a_F\} \\&=
\{ m \mid \langle m,u_F \rangle \geq -a_F+\langle k, u_F\rangle  \}. 
\end{align*}

Hence, $P$ and $P'$ have the same associated normal fan.
In particular, they have the same associated projective toric variety $X_P$.
By~\cite[Proposition 4.2.10]{CLS11}, we have that
\[
A'=\sum (a_F- \langle k,u_F \rangle )D_F.
\]
We conclude that 
\[
A-A'= \sum \langle k,u_F \rangle D_F= \sum \frac{1}{t} \langle tk, u_F \rangle D_F = \frac{1}{t} div(\chi ^{k}).
\]
The last equality follows from~\cite[Proposition 4.1.2]{CLS11}.
We conclude that $A\sim_\qq A$ holds.
If $k$ is a lattice vector, then $t=1$ and $A-A'$ is principal on $X_P$.
\end{proof}

\begin{lemma}\label{lem:toric-linear-lattice}
Let $P$ and $P'$ be two full-dimensional lattice polyhedra in $M_{\qq}$ and $M_{\qq}'$, respectively. Let $X_P$ (resp. $X_{P'}$) be the associated projective toric variety with ample line bundle $A_P$ (resp. $A_{P'}$). Let $H: M_\qq \rightarrow M'_\qq$ be a linear lattice isomorphism
so that $H(P)=P'$. Then, we have an associated toric isomorphism $\phi:X_{P'}\rightarrow X_{P}$, for which $\phi^{*}(A_P)=A_{P'}$.
\end{lemma}

\begin{proof}
The lattice polyhedron $P$ is given by inequalities:
\[
\{ m \mid \langle m,u_F \rangle \geq -a_F\}.
\]
Then, the lattice polyhedron $P'$ is given by inequalities:
\[
\{ H(m) \mid \langle m,u_{F} \rangle \geq -a_F\}=\{ m \mid \langle m,H^{t}(u_{F}) \rangle \geq -a_F\}.
\]
Here, $F'$ are the facets of $P'$ and $u_{F'}=H^{t}(u_{F})\in N'$ is the inward pointing normal vector of the face $F'$.
By~\cite[Proposition 4.2.10]{CLS11}, we have that
\[
A_P=\sum a_F D_F 
\text{ and }
A_{P'}=\sum a_F D_{F'}.
\] 
The lattice isomorphism $H: M_\qq \rightarrow M_{\qq}'$ induces a lattice isomorphism 
$H^{t} : N_{\qq}' \rightarrow N_\qq$ on the dual lattices.
By~\cite[Theorem 3.3.4]{CLS11}, we have an associated equivariant isomorphism of projective toric varieties
$\phi\colon X_{P'}\rightarrow X_P$.
Using~\cite[Proposition 6.2.7]{CLS11}, we can compare 
\begin{align*}
\phi^*(A_P) & =
\phi^*\left( \sum_F a_FD_F\right) 
= 
\sum_{F'} \phi_{A_{P'}}(H^t(u_F)) D_{F'}\\
&=
\sum_{F'} \phi_{A_{P'}}(u_{F'})D_{F'}
 =
\sum_{F'} a_{F'} D_{F'}
=
A_{P'}.
\end{align*}
Here, $\phi_D$ is the support function associated to the Cartier divisor $D$ (see, e.g.~\cite[Theorem 4.2.12]{CLS11}).
This finishes the proof of the lemma.
\end{proof}

\begin{lemma}\label{lem:toric-codim-one-embedding}
Let $P$ be a full-dimensional lattice polyhedron in $M_{\qq}$. Let $X_P$ be the associated projective toric variety and $A_P$ be the associated ample line bundle. Let $F$ be a facet of $P$ and $D_F$ be the corresponding prime torus invariant divisor. Let $M_{F,\qq}$ be the smallest linear subspace of $M_{\qq}$ that contains $F$ and set $M_F:= M \cap M_{F,\qq}$. Let $A_D$ be the ample line bundle in $D$ associated to the lattice polytope $F$ in $M_{F,\qq}$. Then, we have $A_P|_{D} \sim A_D$.
\end{lemma}

\begin{proof}
Let $\Sigma_P$ be the dual fan of $P$. 
For each cone $\sigma \in \Sigma_P$, we have an associated affine toric variety $U_\sigma$ which is an open chart of $X_P$. $D_P$ is principal on each chart of this cover.
Let the local data of $D_P$ be $\{(U_{\sigma},\chi^{-m_{\sigma}}) \}_{\sigma \in \Sigma_P}$. By~\cite[Theorem 4.2.8]{CLS11}, we can take $m_\sigma$ to be the vertex of the polytope that corresponds to $\sigma$. Then, $\mathcal{O}_{X_P}(D_P)$ is the sheaf of sections of a rank 1 vector bundle $V_P \rightarrow X_{P}$ with transition functions $g_{\sigma \tau}=\chi ^{m_{\tau}-m{\sigma}}$.

Let the local data of $D_F$ be $\{(U_{\sigma '},\chi^{-m_{\sigma '}}) \}_{\sigma' \in \Sigma_F}$, where each $\sigma' \in \Sigma_F$ is the projection to $N_{F,\qq}$ of a $\sigma \in \Sigma_P$ that corresponds to a vertex in $F$. Again, by~\cite[Theorem 4.2.8]{CLS11}, we can take $m_{\sigma '}$ to be the vertex in $F$ that corresponds to $\sigma '$, hence also the vertex in $P$ that corresponds to $\sigma$. 
Thus, $\mathcal{O}_{X_F}(D_F)$ is the sheaf of sections of a rank 1 vector bundle $V_F \rightarrow X_{F}$ with transition functions 
$g_{\sigma ' \tau '}=\chi ^{m_{\tau '}-m{\sigma '}}$, where $m_{\tau '}$ and $m{_\sigma '}$ correspond to vertices of $F$.

Therefore, we have that $V_F$ is the restriction of $V_P$ to $F$, i.e. we have a commutative diagram:
\[
\begin{tikzcd}
V_F  \arrow{d} \arrow[r, hook] & V_P \arrow{d}{} \\
X_F \arrow[r, hook, "i"] &  X_P.
\end{tikzcd}
\]
We conclude that  $i^{*}\mathcal{O}_{X_P}(D_P)=\mathcal{O}_{X_F}(D_F)$.
\end{proof}

\begin{proposition}\label{prop:toric-linear}
Let $P$ be a full-dimensional lattice polyhedron in $M_\qq$.
Let $X_P$ be the associated projective toric variety with induced polarization $A_P$.
Let $F$ be a full-dimensional lattice polyhedron in $M'_\qq$.
Let $X_F$ be the associated projective toric variety with induced polarization $A_F$. 
Let $f\colon M'_\qq \rightarrow M_\qq$ be a lattice embedding for which $f(F)$ is a face of $P$. 
Then, we have an associated toric embedding $\phi\colon X_F\rightarrow X_P$ for which $\phi^*A_P \sim A_F$.
\end{proposition}

\begin{proof}
The lattice embedding $f$ can be decomposed as a translation on $M'_\qq$, a linear lattice isomorphism, 
and a sequence of linear lattice embeddings of codimension one. 
Then, the proposition follows from Lemma~\ref{lem:toric-translation}, Lemma~\ref{lem:toric-linear-lattice}, and Lemma~\ref{lem:toric-codim-one-embedding}.
\end{proof}

\subsection{From polyhedral complexes to snc toric varieties}
\label{subsec:from-poly-to-toric}

In this subsection, we construct projective snc toric Calabi-Yau varieties from smooth polyhedral complexes.

\begin{proposition}\label{prop:toric-variety-from-pol-complex}
    Let $\mathcal{P}$ be a smooth polyhedral complex of dimension $n$. Then, there exists an $n$-dimensional simple normal crossing projective toric Calabi-Yau variety $T$, with $\pi_1(T)\cong \pi_1(|\mathcal{P}|)$
\end{proposition}

\begin{proof}
Each polyhedra $F$ in $\mathcal{P}$ is contained in an affine space $M_{F,\qq} \simeq \qq^{N_F}$ so that 
the vertices of $F$ are contained in $\zz^{N_F}$.
By replacing $F$ with its affine span, we may assume $N_F=\dim F$, i.e., $F$ is a full-dimensional polyhedron in $M_{F,\qq}$.
Hence, for each polyhedron $P\in \mathcal{P}$, 
we have an associated toric projective variety $T_P$
and an ample Cartier divisor $A_P$.
Assume that $F$ is a face of $P$, then we have an induced 
lattice embedding $i_{F,P}\colon M_{F,\qq}\rightarrow M_{P,\qq}$.
By Proposition~\ref{prop:toric-linear},
this lattice embedding induces a toric embedding 
$i_{F,P}\colon T_F\hookrightarrow T_P$
for which 
\begin{equation}\label{eq:ample-line} 
i_{F,P}^* A_P \sim A_F.
\end{equation} 
Furthermore, if $G$ is a face of $F$ and $F$ is a face of $P$, then we have the following equality:
\begin{equation}\label{eq:comp}
i_{F,P}\circ i_{G,F} = 
i_{G,P}.
\end{equation} 
Let $P_1,\dots,P_k$ be the $n$-dimensional faces of $\mathcal{P}$.
We denote by $T_1,\dots,T_k$ the associated projective $n$-dimensional toric varieties.
Note that all the polyhedra $P_i$ are smooth,
so each $T_i$ is a smooth projective toric variety.
Then, we can glue the varieties
$T_i$ whenever the polyhedra $P_i$ have a common facet. 
To obtain a normal crossing scheme, we define the gluing by taking affine covers and glue affine locally by taking the fiber product of rings. 
This gluing is well-defined due to the compatibility condition~\eqref{eq:comp}.
We obtain a scheme $T$. 
Since $\mathcal{P}$ is a smooth polyhedral complex and the gluing of irreducible components is normal crossing, the scheme $T$ has snc singularities.
Note that each irreducible component $T_i$ comes with a line bundle $L_i\rightarrow T_i$.
Let $T_{i,j}:=T_i\cap T_j$ and $L_{i,j}\rightarrow T_{i,j}$ be the induced line bundle.
Due to~\eqref{eq:ample-line}, we have closed embeddings
$L_{i,j}\hookrightarrow L_i$ and $L_{i,j}\hookrightarrow L_j$ for each $i$ and $j$.
The push-out of closed embeddings exists in the category of schemes~\cite[Corollary 3.9]{Sch05}.
Hence, we obtain a scheme $L_T$ by gluing the $L_i$'s.
By the universal property of push-outs,
$L_T$ admits a morphism $L_T\rightarrow T$
that restricts to $L_i\rightarrow T_i$ on each component $T_i$.
In particular, $L_T$ is a line bundle over the snc variety $T$.
By the Nakai-Mosheizon criterion, this line bundle is ample.
Hence, $T$ is a projective snc variety.

We claim that $T$ is a Calabi-Yau variety, i.e., for each component $T_i$ of $T$,
we have that $K_T|_{T_i}\sim 0$.
Indeed, since $\mathcal{P}$ is simple, every face of dimension $n-1$ in $\mathcal{P}$ is contained in exactly two maximal polyhedra.
Hence, by adjunction formula, for each $T_i$, we have that 
\[
K_T|_{T_i} \sim K_{T_i}+B_{T_i} \sim 0.
\] 
Here, $B_{T_i}$ is the reduced toric boundary of $T_i$.
This proves that $T$ is a Calabi-Yau variety.

The moment maps $\phi_P \colon T_P\rightarrow P$ glue together to a map
$\phi_T\colon T\rightarrow |\mathcal{P}|$.
We claim that $\phi_T$ induces an isomorphism
\begin{equation}\label{eq:isom-fundamental}
{\phi_T}_* \colon \pi_1(T) \rightarrow \pi_1(|\mathcal{P}|).
\end{equation} 
Indeed, the restriction
of $\phi_T$ to 
\[
\bigcup_{\dim F=1} T_F \rightarrow 
|\mathcal{P}|_1
\] 
induces an isomorphism between fundamental groups.
Here, $|\mathcal{P}|_1$ is the $1$-skeleton of $|\mathcal{P}|$.
Indeed, both spaces have the same graph structure, when considering each $\pp^1$ in $\bigcup_{\dim F=1} T_F$ 
as a topological $S^2$.
Moreover, we have a commutative diagram
\[ \begin{tikzcd}
\pi_1\left( 
\bigcup_{\dim F =1} T_F 
\right) \ar[r]{} \arrow{r}{{\phi_T}_*} \ar[d] \arrow{d}{} &
\pi_1(|\mathcal{P}|_1)\ar[d]
\arrow{d}{} \\
\pi_1\left(\bigcup_{\dim F\leq 2} T_F\right)\ar[r] \arrow{r}{{\phi_T}_*} & 
\pi_1(|\mathcal{P}|_2).
\end{tikzcd}
\]
The kernel of the right vertical map
is the smallest normal subgroup 
of $\pi_1(|\mathcal{P}|_1)$ containing all the loops around $2$-skeleta. 
On the other hand, the kernel of the left vertical arrow is the smallest normal subgroup generated by the 
torus invariant boundary
of the toric surfaces $T_F$ 
with $\dim F=2$.
Indeed, the fundamental group
of a projective toric surface is trivial~\cite[Theorem 10.4.3]{CLS11}.
We conclude that 
the restriction of $\phi_T$ to 
\[
\bigcup_{\dim F \leq 2} T_F \rightarrow
|\mathcal{P}|_2
\] 
induces an isomorphism between fundamental groups.
Observe that gluing cells of dimension at least $3$ does not change the fundamental group of a CW-complex.
Analogously, gluing a toric manifold of dimension at least $3$ along the torus invariant boundary does not change
the fundamental group of an snc variety.
Hence, 
the isomorphism~\ref{eq:isom-fundamental} holds. We deduce that $T$ is an snc CY variety with fundamental group $\pi_1(\mathcal{P})$.
\end{proof}

\subsection{From snc toric varieties to lc singularities}
\label{subsec:from-toric-to-lc}

In this subsection, we construct log canonical singularities of dimension $n+1$ from polyhedral complexes of dimension $n$.

In order to construct log canonical singularities, we will use a result due to Koll\'ar~\cite{Kol12}, (see also~\cite[Theorem 35]{KK14}).
This proposition allows us to construct singularities admitting a prescribed partial resolution.
The following theorem is proved in~\cite[Theorem 8, Propositions 9 and 10]{Kol12}.

\begin{proposition}
\label{prop:prescribed-exceptional}
Let $T$ be an $n$-dimensional projective variety with simple normal crossing singularities and $n\geq 2$. Let $L$ be an ample line bundle on $T$. Then, for $m \gg 1$ there is a germ of a normal singularity
$(X;x)$ with a partial resolution:
\[ 
\xymatrix@R=4em@C=4em{
T  \ar[d] \ar@{^{(}->}[r] & Y \ar[d]^-{\phi} \\
x \ar@{^{(}->}[r] &  (X;x)
}
\]
satisfying the following conditions:
\begin{enumerate}
    \item $T$ is a Cartier divisor on $Y$,
    \item the normal bundle of $T$ in $Y$ is $K_T \otimes L^{-m}$,
    \item we have an isomorphism $\pi_1(Y)\simeq \pi_1(T)$,
    \item the kernel of $\pi_1^{\rm loc}(X;x) \twoheadrightarrow \pi_1(Y)$ is cyclic, central, and generated by any loop around an irreducible component of $T$,
    \item if $\dim T \leq 4 $, then $(x \in X )$ is an isolated singular point, and
    \item if $K_T \sim 0$, then $K_X$ is Cartier and $(x\in X)$ is lc.
\end{enumerate}
\end{proposition}

\begin{remark}\label{rem:very-ample-prescribed-exceptional} {\em
In Proposition~\ref{prop:prescribed-exceptional} we only need $L^m$ to be very ample. In our application of this proposition, we will have two ample line bundles $L_1$ and $L_2$. By \cite[Chapter II. Exercise 7.5]{Har77}, for $k \gg 1$, we have that $L_1 \otimes L_2^k$ is very ample. Thus, in the previous proposition, we may replace $L^m$ with $L_1 \otimes L_2^k$ and yield the same conclusion of the proposition, up to replacing
$L^{-m}$ with $L_1\otimes L_2^k$ in Proposition~\ref{prop:prescribed-exceptional}.$(2)$.
}
\end{remark}

\begin{definition}\label{def:blow-up-snc-reducible-var}
{\em
Let $T$ be an $n$-dimensional simple normal crossing variety. Let $F$ be an $(n-1)$-dimensional stratum.
Let $Z$ be a divisor in $F$ that intersects transversally all the strata of $T$.
We will describe blow-ups in each irreducible component of $T$ and then glue them together. In this definition, when we say the blow-up of a subvariety $V \subset W$, we mean the blow-up defined by the reduced scheme $V$ in $W$.
Let $E$ be an irreducible component, we will say $F_E$ and $Z_E$ for the intersections of $F$ and $Z$ with $E$, respectively. 

We first perform the blow-up of $Z_E$ in $E$, call it $p:E' \rightarrow E$. As $Z_E$ in $F_E$ has codimension at most $1$, we can identify $F_E$ with its strict transform in $E'$. Let $D_E$ be the exceptional divisor of $p:E' \rightarrow E$. Then, we perform the blow-up of $D_E \cap F_E$ in $E'$ and call it $p':E''\rightarrow E'$. Similarly, the locus of the blow-up has codimension at most $1$ in $F$. Thus, we can identify again $F_E$ with its strict transform in $E''$. Therefore, we still have the datum for gluing the blow-ups of the irreducible components in $T$. So, we can glue the irreducible components $E''$, to obtain a simple normal crossing variety $T''$. This will be called the {\em iterated blow-up of $Z$ in $T$}. Here we can also identify $Z_F$ in $T''$ as the intersection of the exceptional divisor with $F_E$ in $E''$. We call $D_1$ the union of all the strict transforms of $D_E$ for every irreducible component $E''$, and we call $D_2$ the union of the exceptional divisors of $p':E'' \rightarrow E'$ for every irreducible component $E''$. Then, $D_1$ and $D_2$ are Cartier Divisors on $T''$. }
\end{definition}

\begin{remark}\label{rem:blow-up-CY}
{\em 
Let $T$ be an $n$-dimensional simple normal crossing Calabi-Yau variety, such that all the codimension $c$ strata are contained in exactly $c+1$ irreducible components. Then, $T''$ the iterated blow-up defined in Definition~\ref{def:blow-up-snc-reducible-var} is also Calabi-Yau. 
Indeed, we only need to check the statement after the first blow-up, our snc variety $T'$ is also Calabi-Yau. This can be checked at each irreducible component $E_i$. Let us call $Z_i:=Z \cap E_i$ and $F_i:=F \cap E_i$. Let $p:E_i' \rightarrow E_i$ be the blow-up of $Z_i$ on $E_i$. Let $D$ be the exceptional divisor of $p:E' \rightarrow E$.

Let $c$ be the codimension of $Z_i$ in $E_i$. Then, $p^\ast(K_{E_i})=K_{E_i}-(c-1)D$. And for any other component $E_j$ that contains $Z_i$, we have that  $p^\ast(E_J)|_{E_i}=E_J|_{E_i}+D$. By our hypothesis, $F_i'$ is contained in $c$ irreducible components of $T'$. Thus, $Z_i'$ is contained in $c-1$ irreducible components of $T'$, other than $E_i$. Hence, adding our equalities for every irreducible component containing $Z_i$, we obtain that $K_{T'}|_{E_i}$ is trivial.
}
\end{remark}    

\begin{lemma}\label{lem:ample-blow-up}
Let $T$ be an snc variety with an irreducible component $E$. Let $F$ be a codimension one stratum contained in $E$. Moreover, let $Z$ be a smooth divisor in $F$ that intersects transversally all the strata of $T$. Let $T'' \rightarrow T$ be the iterated blow-up of $Z$ in $T$, as in Definition~\ref{def:blow-up-snc-reducible-var}.
Let $D_1$ and $D_2$ be the exceptional divisors over $T$, as in Definition~\ref{def:blow-up-snc-reducible-var}. Then, the divisors $-3D_2-2D_1$ and $-4D_2-3D_1$ are relatively ample over $T$.
\end{lemma}

\begin{proof}
It suffices to check that the restrictions of these divisors to every irreducible component are ample over $T$.
Let $E_i$ be an irreducible component of $T$. Let us call $Z_i:=Z \cap E_i$ and $F_i:=F \cap E_i$. We will call $F_i''$ the strict transform of $F_i$ in $E_i''$.
So, we have the blow-ups $p_2:E_i'' \rightarrow E_i'$ and $p_1:E_i' \rightarrow E_i$. By abuse of notation, the intersections of $D_1$ and $D_2$ with $E_i''$ are also called $D_1$ and $D_2$. We will define $Z_i''$ to be the intersection of the exceptional divisor with $F_i''$.

As $Z_i$ in $E_i$ has codimension $c\geq 2$. The fiber of the exceptional divisor of $p_1:E_i' \rightarrow E_i$ over $Z_i'$ are $\mathbb{P}^{c-1}$. When we perform the second blow-up $p_2:E_i'' \rightarrow E_i'$, in the fibers of $D_1$ over $Z_i'$ we are performing the blow-up at one smooth point. Hence, the fibers of $D_1$ over $Z_i''$ are the blow-ups of $\mathbb{P}^{c-1}$ at one point, and the fibers of $D_2$ over $Z_i''$ are $\mathbb{P}^{c-1}$.

Let $C_1$ be a curve in $D_1$ that intersects $D_2$ transversally in one point and $C_2$ a curve in $D_2$ that intersects $D_1$ transversally in one point.

We first prove that $C_1$ and $C_2$ generate the relative cone of curves of $E'' \rightarrow E$. Indeed, let $N$ be a relatively nef divisor that is positive on both curves. There must be a curve $C_2'\sim C_2$, such that $C_2'$ lies in $D_1 \cap D_2$ and is contracted on $E$. We have that $C_1$ and $C_2'$ generate the relative cone of curves of $E'$. Hence, $N$ is relatively ample on $D_1$. Furthermore, $N$ is relatively ample on $D_2$, since it intersects a curve positively and $D_2$ has relative Picard rank 1. Therefore, $N$ intersects positively any exceptional curve, hence it is relatively ample. Thus, there cannot be any other generator of the relative cone of curves.

As $C_1$ and $C_2$ generate the relative cone of curves, we only need to compute the intersection products with these curves. We have that $p_1^\ast(K_{E'})=K_{E''}-(c-1)D_2$. Hence, $K_{E''} \sim_{E_i'} (c-1)D_2$. By adjunction, we obtain the value of $cD_2 \cdot C_2=(K_{E''}+D_2) \cdot C_2=K_{D_2}\cdot C_2=-c$. Therefore, $D_2 \cdot C_2=-1$.
Similarly, we have that $(p_2\circ p_1)^\ast (K_E)=K_{E''}-(c-1)D_1-(2c-2)D_2$, hence $K_{E''}\sim_{E} (c-1)D_1+(2c-2)D_2$. By adjunction, $(cD_1+(2c-2)D_2)\cdot D_1=(K_{E''}+D_1)\cdot C_1=K_{D_1}\cdot C_1=-2$, we conclude that $D_1\cdot C_1=-2$.
This implies that $-3D_2-2D_1$ and $-4D_2-3D_1$ are relatively ample over $T$.
\end{proof}

\begin{proof}[Proof of Theorem~\ref{introthm:from-sm-poly-comp-to-lc-sing}]

By Proposition~\ref{prop:toric-variety-from-pol-complex}, we have a Toric Calabi-Yau snc variety $T$, with an ample line bundle $H$. We will perform iterated blow-ups with loci contained in disjoint irreducible components. Let $E$ be an irreducible component of $T$. Pick $F$ a codimension 1 stratum of $T$ contained in $E$. Let $Z$ be a smooth divisor of $F$ that intersects transversally all the strata.

We perform the iterated blow-up of $Z$ in $T$, as in Definition~\ref{def:blow-up-snc-reducible-var}. Let us call these maps
$p_1:T'' \rightarrow T'$ and $p_2: T' \rightarrow T$. Let $D_1$ be the exceptional divisor of the first blow-up and $D_2$ be the exceptional divisor of the second blow-up in $E'$. By Lemma~\ref{lem:ample-blow-up}, the divisors $p^\ast(mH)-3D_2-2D_1$  and $p^\ast(mH)-4D_2-3D_1$ are ample. We can perform this same procedure for $G$ an irreducible component disjoint from $E$, with exceptional divisors $G_1$ and $G_2$. Thus,
$p^\ast(mH)-3G_2-2G_1$ and $p^\ast(mH)-4G_2-3G_1$ are ample.

Now, consider the following line bundles
$L_1:=p^*(mH)-3D_2-2D_1-4G_2-3G_1$ and $L_2:=p^*(mH)-3D_2-2D_1-3G_2-2G_1$. By our construction, we have smooth rational curves $P_1$ and $P_2$ that do not intersect any irreducible component other than $E$ and $G$, respectively.
By Remark~\ref{rem:blow-up-CY}, we have that $T''$ is a simple normal crossing Calabi-Yau variety. For $k \gg 1$, the line bundle $L_1 \otimes L_2 ^k$ is very ample. Hence, as mentioned in Remark~\ref{rem:very-ample-prescribed-exceptional}, we can use Proposition~\ref{prop:prescribed-exceptional}. Then, there exists an
$(n+1)$-dimensional 
singularity $(X;x)$ and a partial resolution:
\[ 
\begin{tikzcd}
T  \arrow{d} \arrow[r, hook] & Y \arrow{d}{\phi} \\
x \arrow[r, hook] &  X
\end{tikzcd}
\]
satisfying the following
statements: 
\begin{enumerate}
    \item $T$ is a Cartier divisor on $Y$,
    \item the normal bundle of $T$ in $Y$ is $L_1 ^{-1} \otimes {L_2}^{-k}$,
    \item we have an isomorphism $\pi_1(Y)\simeq \pi_1(|\mathcal{P}|)$,
    \item the kernel of $\pi_1^{\rm loc}(X;x) \twoheadrightarrow \pi_1(Y)$ is cyclic, central, and generated by any loop around an irreducible component of $T$, and
    \item the singularity $(X;x)$ is log canonical.
\end{enumerate}

Since $T$ and $Y$ are smooth along $P_i$, the restriction of the normal bundle to $P_i$ are lens spaces $S_i$, with $|\pi_1(S_i)|=c_1(N) \cap P_i$.  The order of any loop around $T$ divides $|\pi_1(S_i)|$.

The restriction of $L_1$ to $P_1$ is $\mathcal{O}(2)$ and its restriction to $P_2$ is $\mathcal{O}(1)$.
Furthermore, the restriction of $L_2$ to $P_1$ is $\mathcal{O}(1)$ and its restriction to $P_2$ is $\mathcal{O}(1)$. Hence, the restrictions of the normal bundle $L_1 \otimes L_2^k$ to these curves are $\mathcal{O}(k)$ and $\mathcal{O}(k+1)$. Thus, the Lens spaces $S_1$ and $S_2$ have coprime orders, therefore all the loops around $T$ are trivial in the local fundamental group of $(X;x)$. 
Therefore, the kernel of $\pi_1^{\rm loc}(X;x) \twoheadrightarrow \pi_1(Y)$ is trivial, so we have the following isomorphisms:

$$\pi_1^{\rm loc}(X;x) \cong \pi_1(Y) \cong \pi_1(|\mathcal{P}|) $$
\end{proof}

\section{Threefold log
canonical singularities}
\label{sec:3-fold-sing}

In this section, 
we study the fundamental groups of log canonical $3$-fold singularities.
We show positive and negative results.
First, we prove that surface groups
appear among the fundamental groups
of lc $3$-fold singularities.
Then, we prove that no $F_r$ with $r\geq 2$ appears as the fundamental group of an lc $3$-fold singularity.

\subsection{Surface groups in dimension $3$}
\label{subsec:surface-groups-in-dim-3}
In this subsection, we show that surface groups appear as the fundamental group of  a log canonical $3$-fold singularity.

\begin{proof}[Proof of Theorem~\ref{introthm:surf-group-in-3}]
Let $S$ be a $2$-dimensional manifold and $G_S$ be its fundamental group. Let $\mathcal{T}$ be a triangulation of  $S$. We can consider $\mathcal{T}$ as a polyhedral complex where each triangle is given by
\[
T={\rm conv}\{(0,0),(0,3),(3,0)\}\subset \qq^2.
\]
We may assume that each vertex of $\mathcal{T}$ is contained in at most $7$ triangles.
Let $v$ be a vertex of degree $d$ at least $4$. 
Let $P_{d}$ be a smooth lattice polygon
with no lattice points in the relative interior of its edges.
Let $e_1,\dots,e_{d}$ be its edges.
Let $T_1,\dots,T_{d}$ be the triangles containing $v$.

We replace $T_i$ with $T'_i$, where:
\begin{enumerate}
    \item If $v=(0,0) \in T_i$, then we replace
    $T_i$ with $T_i\cap H_1^+$, where 
    $H_1=\{(x,y)\mid x+y\geq 1\}$. By doing so, we
    delete the vertex $(0,0)$ and add the vertices
    $(1,0)$ and $(0,1)$.
    \item If $v=(0,3)\in T_i$, then we replace $T_i$
    with $T_i\cap H_2^-$, where $H_2=\{(x,y)\mid y=2\}$. By doing so, we delete the vertex $(0,3)$ and add the vertices $(1,2)$ and $(0,2)$.
    \item If $v=(3,0)\in T_i$, then we replace $T_i$ with $T_i\cap H_3^-$, where $H_3=\{(x,y)\mid x=2\}$.
    By doing so, we delete the vertex $(3,0)$ and add
    the vertices $(2,0)$ and $(2,1)$.
\end{enumerate}

Let us call $f_i$ the edge of $T'_i$ joining the two new vertices.
For each $i\in \{1,\dots,d\}$ we can consider a lattice isomorphism that glues $e_i$ with $f_i$.
By doing so, we obtained a $2$-dimensional polyhedral complex $\mathcal{T}'$ homotopic to $T$, that does not contain the singular point $v$ of $\mathcal{T}$ and has no other singular point introduced. Since our construction works for any singular point of $T$, even after doing it for any other vertex. Performing this construction for all the singular vertices, yields a $2$-dimensional smooth polyhedral complex $\mathcal{P}_S$ that is homotopic to $S$. 
Then, by Theorem~\ref{introthm:from-sm-poly-comp-to-lc-sing}, we obtain an isolated $3$-dimensional lc singularity $(X_S;x)$ for which the isomorphism $\pi_1^{\rm loc}(X_s;x)\simeq G_S$ holds.
\end{proof} 

\subsection{Free groups and $3$-fold singularities}
\label{subsec:no-Fr-in-dim-3}
In this subsection, we show that no free $F_r$, with $r\geq 2$, appears as the fundamental group of a log canonical $3$-fold singularity.
To do so, we will first prove some lemmas.

\begin{lemma}\label{lem:fibrations}
Let $(E,B_E)$ be a log Calabi--Yau dlt surface for which $\mathcal{D}(E,B_E)\simeq S^1$.
There exists a sequence of blow-ups 
$E'\rightarrow E$ at $0$-dimensional strata of $B_E$
satisfying the following:
\begin{itemize}
    \item there exists a rational movable curve $C$ on $E'$ that is contained in the smooth locus of $E'$, and 
    \item the curve $C$ intersects $B_{E'}$ transversally at two points, where $(E',B_{E'})$ is the log pull-back of $(E,B_E)$.
\end{itemize}
\end{lemma}

\begin{proof}
We know that $(E,B_E)$ is log crepant equivalent to a log Calabi--Yau toric surface $(S,B_S)$ (see, e.g.,~\cite[Proposition 1.3]{GHS16}).
We may assume that $S$ is smooth.
By blowing-up strata of $B_E$ and $B_S$, we may assume that for each component $P$ of $B_E$ the center of $P$ on $S$ is a divisor, and vice-versa.
By further blowing-up strata of $B_E$ and $B_S$, we may assume that
$S$ admits a toric fibration $p$ to $\pp^1$.
Note that the general fiber of $p$ is a movable rational curve that intersects $B_S$ transversally at two points.
Let $C_1,\dots,C_k \in E$ be the exceptional curves of $E\dashrightarrow S$. 
Then, we have that $a_{C_i}(E,B_E)=1$ for each $C_i$.
In particular, $C_i$ may be extracted on $S$ by performing a blow-up at a closed point on $B_S$.
Let $S'\rightarrow S$ be the model where all the $C_i$'s are extracted.
Let $(S',B_{S'})$  be the log pull-back of $(S,B_S)$ to $S'$.
Let $C'$ be the strict transform of a general fiber of $p$ on $S'$.
Note that $C'$ is a movable rational curve that intersects $B_{S'}$ transversally at two points.
Furthermore, we have a morphism 
$q\colon S'\rightarrow E$.
Note that no exceptional curve of $q$ intersects $C'$.
Otherwise, we would have two components
of $B_{S'}$ that do not intersect on $S'$ but their images intersect on $E$.
This would contradict the fact that $\mathcal{D}(E,B_E)\simeq S^1$.
Thus, we deduce that the image $C_E$ of $C'$ on $E$ is a rational movable curve that intersects $B_E$ transversally at two points.
\end{proof}

\begin{lemma}\label{lem:nice-3-fold-dlt-mod}
Let $(X;x)$ be an isolated lc $3$-fold singularity.
There exists a $\qq$-factorial dlt modification $\phi\colon Y\rightarrow X$
and a divisor $E_Y$ fully supported 
on ${\rm Ex}(\phi)$ that is effective
and satisfy:
for every movable curve $C_i$ on an irreducible component $E_i$ of $E$,
we have that $-E_Y\cdot C_i>0$.
\end{lemma}

\begin{proof}
Let $\psi\colon X'\rightarrow X$ be a log resolution obtained by blowing-up centers of codimension at least $2$.
Then, there exists an effective divisor 
$E_{X'}$ supported on $E_0:={\rm Ex}(\psi)$ such that $-E_{X'}$ is ample over the base.
We run a $(K_{X'}+E_0)$-MMP over $X$ that terminates with a dlt modification $(Y,E)$.
Let $X''\rightarrow X'$ be a resolution for which there exists a morphism
$p\colon X''\rightarrow Y$.
Therefore, there exists an effective divisor $E_{X''}$ supported on ${\rm Ex}(X''\rightarrow X)$
for which $-E_{X''}$ is ample over $X$.
We set $E_Y$ to be the push-forward of $E_{X''}$ to $Y$.
Then, the statement follows from the negativity lemma.
Indeed, we can write
\[
p^*(-E_Y)=-E_{X''}+F, 
\]
where $F$ is an effective divisor
exceptional over $Y$.
Let $C_i$ be a movable curve on $E_i$.
Then, $C_i$ is not contained in $p(F)$
as $p(F)$ has codimension at least two in $Y$.
Let $C''_i$ be the strict transform of $C_i$ on $X''$.
The following equalities hold 
by the projection formula 
and the fact that $C''_i$ is not contained in the support of $F$
\[
-E_Y\cdot C_i = 
p^*(-E_Y)\cdot C_i=
(-E_{X''}+F)\cdot C_i\geq 
-E_{X''}\cdot C_i >0.
\]
This finishes the proof of the lemma.
\end{proof}

\begin{proof}[Proof of Theorem~\ref{introthm:not-free-in-3}]
Note that $F_3$ is a normal subgroup of $F_2$ of index $2$. 
Thus, we may assume that $r\geq 3$.

Let $(X;x)$ be a $3$-dimensional isolated lc singularity.
We assume that $\pi_1^{\rm loc}(X;x)\simeq F_r$ for some $r\geq 3$.
Let $(X';x')\rightarrow (X;x)$
be the index one cover of $K_X$.
Then $(X';x')$ is again an isolated log canonical singularity with $K_X'$ Cartier.
Note that we have an exact sequence
\[
1\rightarrow 
\pi_1^{\rm loc}(X';x')
\rightarrow 
\pi_1^{\rm loc}(X;x)
\rightarrow 
\zz/k\zz
\rightarrow 1,
\]
so the local fundamental group
of $(X';x')$ is free with at least $3$ generators.
We may replace $(X;x)$ with $(X';x')$ so we can assume that $K_X$ is Cartier.

If $(X;x)$ is klt, then its fundamental group is finite due to Theorem~\ref{introthm:fund-group-klt}.
Hence, we may assume that $(X;x)$ is log canonical, $K_X$ is Cartier, and $\{x\}$ is a log canonical center. 
In this case, the dual complex $\mathcal{D}(X;x)$ is a manifold of dimension at most $2$.

Let $(Y,E)$ be a $\qq$-factorial dlt modification of $(X;x)$.
We let $E_1,\dots,E_r$ be the irreducible components of $E$
and $\gamma_i$ be a loop around the divisor $E_i$.
We have a sequence of surjective homomorphisms and isomorphisms:
\[
\xymatrix{
\pi_1^{\rm loc}(X;x)\simeq 
\pi_1(Y\setminus E)\ar[r]^-{\psi_1} &
\pi_1(Y\setminus Y^{\rm sing}) \ar[r]^-{\psi_2} &
\pi_1(Y)\simeq \pi_1(E) \ar[r]^-{\psi_3}&
\pi_1(\mathcal{D}(E)) 
\simeq 
\pi_1(\mathcal{D}(X;x)).
}
\]
The kernel of $\psi_1$ is the smallest normal group containing $\gamma_i$ for $i\in \{1,\dots,r\}$.
The kernel of $\psi_2$ is generated by torsion elements as $Y$ has isolated klt singularities
and klt singularities have finite regional fundamental groups~\cite[Theorem 1]{Bra21}.
The kernel of $\psi_3$ is the smallest normal group containing the image of $\pi_1(E_i)\rightarrow \pi_1(E)$ for each irreducible component $E_i$. We will proceed in three different cases depending on the coregularity of $(X;x)$.\\

\textit{Case 1:} In this case, we assume that ${\rm coreg}(X;x)=2$.\\

In this case, $E$ only has one component.
Thus, the kernel of $\psi_1$ is 
normally generated by a single loop $\gamma_1$.
We conclude that the group
$\pi_1(Y\setminus Y^{\rm sing})$ is a one-relator group.
So, there is an isomorphism
\[
\pi_1(Y\setminus Y^{\rm sing})
\simeq 
\langle x_1,\dots,x_r \mid s\rangle.
\]
By the Magnus-Karras-Solitar theorem~\cite[Theorem 1]{KMS60},
every torsion element of
$\pi_1(Y\setminus Y^{\rm sing})$
has the form $nrn^{-1}$
where $n\in \pi_1(Y\setminus Y^{\rm sing})$
and $r^k=s$ for some integer $k$.
Since the kernel of $\psi_2$ is torsion,
we conclude that $\pi_1(E)$ is a one-relator group as well.
Observe that in this case, $E$ is a Calabi-Yau surface with canonical singularities.
Hence, its fundamental group is isomorphic to the fundamental group
of a smooth Calabi--Yau surface.
We use the classification of such surfaces, to show that their fundamental groups cannot be a one-relator group with at least $3$ generators.
Indeed, the fundamental group of a smooth Calabi--Yau surface is either:
\begin{itemize}
    \item trivial,
    \item a finite cyclic group, 
    \item the free abelian group $\zz^4$, or
    \item an extension of $\zz^4$ by a finite cyclic or a finite bi-cyclic group.
\end{itemize}
The first two cases have rank at most $2$ and are not free, so they are not one-relators with at least $3$ generators.
By~\cite{New68}, we know that 
an abelian finitely generated subgroup of a one-relator group
has rank at most $2$.
We conclude that the third and fourth cases
cannot be one-relator groups.
This leads to a contradiction.\\

\textit{Case 2:} In this case, we assume that ${\rm coreg}(X;x)=1$.\\

We show that the loops $\gamma_i$ and $\gamma_j$ commute.
We run a $(K_Y+E-E_i)$-MMP over $X$.
This minimal model program terminates 
with a good minimal model $\phi'\colon Y'\rightarrow X$.
Let $E'_i$ be the strict transform 
of $E_i$ on $Y'$.
In the minimal model $Y'$
the divisor $-E'_i$ is nef over $X$.
We claim that ${\phi'}^{-1}(x)=E'_i$ holds set-theoretically.
Indeed, if there is another divisor $E'_j$ in this fiber, then by connectedness we may assume $E'_i\cap E'_j\neq \emptyset$.
Thus, a general ample curve~\footnote{We say that $C$ is a {\em general ample curve} in an $n$-dimensional variety if it is the intersection of $(n-1)$ general effective ample divisors.}  in $E'_j$ will intersect $E'_i$ positively and hence will intersect $-E'_i$ negatively.
We conclude that, at some step of this minimal model program, we lead to a model $Y''$ on which the strict transform of $E_i$ and $E_j$ intersect along a codimension $2$ subvariety.
We write $\phi''\colon Y''\rightarrow X$ for the associated projective morphism.
We let $E''_i$ and $E''_j$ the strict transform of $E_i$ and $E_j$ on $Y''$, respectively.
We let $\gamma''_i$ and $\gamma''_j$ be the loops around $E''_i$ and $E''_j$.
Note that at the generic point of $Z$, the variety $Y''$ has a toric surface singularity.
In particular, the regional fundamental group $\pi_1^{\rm reg}(Y'',\eta_Z)$ is abelian.
We can move homotopically the loops
$\gamma''_i$ and $\gamma''_j$ so that the corresponding circles around $E''_i$ and $E''_i$ are in a small neighborhood of $Z$.
Hence, the loops $\gamma''_i$ and $\gamma''_j$ commute as loops in 
$\pi_1(Y''\setminus {\phi''}^{-1}(x))$.
There is a natural isomorphism
\[
\pi_1(Y\setminus E) \rightarrow 
\pi_1(Y''\setminus {\phi''}^{-1}(x)),
\]
that sends $\gamma_i$ (resp. $\gamma_j$) to $\gamma_i''$ (resp. $\gamma_j''$). 
We conclude that $\gamma_i$ and $\gamma_j$ commute.

Now, we turn to prove that every $\gamma_i$ is a torsion element of $\pi_1(Y\setminus E)$.
Observe that each surface $E_i$ admits a morphism $E_i\rightarrow C$ to an elliptic curve $C$ with general fiber a rational curve.
Let $C_i$ be a general fiber of $E_i\rightarrow C$.
Let $-m_i:=C_i\cdot E_i$.
Then, we have a relation of the form
$\gamma_i^{m_i}=\gamma_{i-1}\gamma_{i+1}$
in the fundamental group
$\pi_1(Y\setminus E)$.
Indeed, this relation holds in the restriction of the normal bundle to $C_i$.
The $r\times r$ matrix $M$ with entries
$M_{i,j}:=E_i\cdot C_j$ is invertible.
This fact and the commutativity of 
the loops $\gamma_i$ implies that there exists $k\gg 0$ such that $\gamma_i^k=1$ for each $i\in \{1,\dots,r\}$.
We are assuming that $\pi_1(Y\setminus E)$ is free. 
Hence, it has no non-trivial torsion elements.
Since the kernel of $\psi_1$ is normally generated by torsion elements, we conclude that $\psi_1$ is an isomorphism.
Since the kernel of $\psi_2$ is torsion, again we conclude that 
$\psi_2$ is an isomorphism.
Hence, we get that $\pi_1(E)\simeq F_r$ for some $r\geq 3$.

Finally, we turn to analyze the kernel of $\psi_3$ in this case.
Note that $\pi_1(\mathcal{D}(E))\simeq \zz$.
By Lemma~\ref{lem:fun-reg-vs-cbf},
we conclude that $\pi_1(E_i)\simeq \pi_1(C)\simeq \zz^2$ for each $i$.
Since $C$ has a lifting to $E_i$ via the isomorphism $C\rightarrow E_i\cap E_{i-1}$, we conclude that the image of every homomorphism
$\pi_1(E_i)\rightarrow \pi_1(E)$ consists of the same two commuting loops in $\pi_1(E)$.
By the previous paragraph, we have that
$\pi_1(E)\simeq F_r$ so these two commuting elements have the form 
$g^s$ and $g^k$ for some $g\in F_r$.
In particular, we obtain a presentation 
$\zz \simeq \langle x_1,x_2,\dots,x_r\mid s\rangle$ with a single relation $s$ and $r\geq 3$. 
This leads to a contradiction.\\

\textit{Case 3:} In this case, we assume that ${\rm coreg}(X;x)=0$.
In particular, $\mathcal{D}(X;x)$ is a closed $2$-manifold.\\

By Lemma~\ref{lem:nice-3-fold-dlt-mod}, we may assume that there exists $E_Y$ supported on $E$ for which $-E_Y$ intersects positively every curve that is movable on some $E_i$.
By Lemma~\ref{lem:fibrations}, we may find a dlt modification $(Y',E')$ satisfying the following.
For each $E_i$, its strict transform $E'_i$ on $E_i$ contains a movable rational $C'_i$ curve contained on its smooth locus that intersects $(E'-E_i')|_{E_i'}$ transversally at two points.
Furthermore, the dlt modification $(Y',E')$ is obtained by consecutively blowing-up $0$-dimensional strata of $(Y,E)$.
We show that for every component $E_i'$ of $E'$ there exists a movable smooth rational curve $C_i'$ satisfying the relation
$\gamma_i^{m_i}=\prod_{j\neq i} \gamma_j^{k_{j,i}}$
holds in $\pi_1(Y'\setminus E')$,
where $-m_i=E'_i\cdot C'_i$
and $k_{j,i}=E'_j\cdot C_i'$ for each $i$.
If $E_i'$ is the strict transform of a component $E_i$ of $E$, then this follows from Lemma~\ref{lem:fibrations}.
Indeed, this relation holds in the normal bundle of $E$ restricted to $C_i'$.
On the other hand, if $E_i'$ is exceptional over $Y$, then 
it is first extracted on the blow-up of a smooth point.
After such a blow-up, the strict transform of $E_i'$ is isomorphic to $\pp^2$, so it suffices to take the strict transform of a general line in this projective space.
As $\psi\colon Y'\rightarrow Y$ is obtained by blowing-up smooth points, there exists an effective divisor
$F'$ supported on ${\rm Ex}(\psi)$
for which $-F'$ is ample over $Y$.
We let $E_{Y'}=\psi^*(E_Y)+\epsilon F$ for $\epsilon$ small enough.
Then, $-E_{Y'}$ intersects positively every curve that is movable on a component of $E_i'$.
We replace $(Y,E)$ with $(Y',E')$
and $E_Y$ with $E_{Y'}$.

By the previous paragraph,
we have a $\qq$-factorial dlt modification $\pi\colon (Y,E)\rightarrow X$
and we know that the following conditions hold:
\begin{itemize}
    \item[(i)] There exists an effective divisor $E_Y$ supported on ${\rm Ex}(\pi)$ such that $-E_Y$ intersects positively every curve that is movable on a component of $E$, 
    \item[(ii)] the loops $\gamma_i$ around the components $E_i$ commute, and 
    \item[(iii)] for each component $E_i$ of $E$, there exists a movable rational curve $C_i$ on $E_i$ such that the relation 
    \begin{equation}\label{eq:relations} 
    \gamma_i^{m_i}=\prod_{j\neq i}\gamma_j^{k_{j,i}}
    \end{equation} 
    holds in $\pi_1(Y\setminus E)$, where $m_i=E_i\cdot C_i$ and 
    $k_{j,i}=E_j\cdot C_i$.
\end{itemize}
The first and third statements are proved in the previous paragraph.
The proof of the second statement
follows from the previous case.\\

\textit{Claim:}
There is no non-trivial divisor $G$
with support contained in ${\rm Ex}(\pi)$
such that $G\cdot C_i=0$ for each $i$.

\begin{proof}[Proof of the Claim]
By means of contradiction,
assume there exists such a divisor $G$.
First, we assume that $G$ has a positive coefficient, i.e.,
${\rm coeff}_{E_i}(G)>0$ for some $E_i$.
Consider the positive real number
\[
\lambda_0:=\max\{ \lambda \mid 
{\rm coeff}_{E_i}(-E_Y + \lambda G)\leq 0
\text{ for each $i\in \{1,\dots,r\}$}
\}.
\]
By construction, we know
that $-E_Y+\lambda_0 G$ supported on ${\rm Ex}(\pi)$ but it is not fully supported
on ${\rm Ex}(\pi)$.
Without loss of generality, we may assume
that ${\rm coeff}_{E_1}(-E_Y+\lambda_0 G)=0$.
Recall that $C_1\cdot E_j\geq 0$ for each $j\geq 2$.
Thus, we have that 
\[
0\geq (-E_Y+\lambda_0 G)\cdot C_1 = 
-E_Y \cdot C_1 >0,
\]
where the equality follows from the fact $G\cdot C_1=0$.
This leads to a contradiction.
We conclude that $G$ has no positive coefficients.

Now, we assume that $G\leq 0$ is non-trivial.
We consider the positive real number
\[
\lambda_0:=\max\{\lambda\mid
{\rm coeff}_{E_i}(-E_Y-\lambda G)\leq 0
\text{ for each $i\in \{1,\dots,r\}$}\}.
\]
By construction, we know that $-E_Y-\lambda_0 G$ is not fully supported on ${\rm Ex}(\pi)$. Without loss of generality, 
we assume 
${\rm coeff}_{E_1}(-E_Y+\lambda_0 G)=0$.
Then, the inequalities
\[
0\geq (-E_Y-\lambda_0 G)\cdot C_1 =
-E_Y\cdot C_1 >0,
\]
hold. This leads to a contradiction.
We conclude that the only divisor with support contained in ${\rm Ex}(\pi)$ that intersects each curve $C_i$ trivially is the trivial divisor.
\end{proof}

The claim implies that the $r\times r$ matrix $M$ with entries $M_{i,j}:=E_i\cdot C_j$ is invertible.
The invertibility of the matrix $M$,
the relations~\eqref{eq:relations},
and the fact that the loops $\gamma_i$ commute in $\pi_1(Y\setminus E)$ imply
that there exists $k\in \zz_{>0}$
such that $\gamma_i^k=1$ for each $i$.

We conclude that the kernel of $\psi_1$ is normally generated by torsion elements.
Since $F_r$ is torsion-free, 
we conclude that $\psi_1$ is an isomorphism.
Analogously, $\psi_2$ is an isomorphism.
Finally, each component $E_i$ of $E$ is simply connected.
Indeed, the minimal resolution of $E_i$
is a smooth rationally connected projective surface.
This implies that $\psi_3$ is also an isomorphism.
Thus, we have that
\[
F_r\simeq \pi_1^{\rm loc}(X;x)\simeq \pi_1(\mathcal{D}(X;x)),
\]
where $r\geq 3$ and $\mathcal{D}(X;x)$ is a closed $2$-dimensional manifold. 
This leads to a contradiction.
\end{proof}

\section{Fourfold log canonical singularities}
\label{sec:4-dim-lc-sing}
In this section, 
we study the fundamental groups
of log canonical $4$-fold singularities.
In order to produce interesting examples of fundamental groups of $4$-dimensional lc singularities, 
we need to exhibit $3$-dimensional smooth polyhedral complexes
with interesting fundamental groups.
To do so, we will consider subcomplexes of the Freudenthal decomposition of $\rr^4$.

\subsection{The Freudenthal decomposition}
\label{subsec:F-decomp}
In this subsection, we introduce the Freudenthal decomposition of $\rr^4$ and prove a property of its dual.

\begin{definition}
{\em 
For a cube $C$ with coordinates $x_j \in [i_j,i_j+1]$ for $j \in \{1,2,3,4\}$, the { \em Freudenthal decomposition} of $C$ is defined in the following way.
For each path $P$ in the edges of $C$ that:\begin{itemize}
    \item starts in $(i_1,i_2,i_3,i_4)$,
    \item ends in  $(i_1+1,i_2+1,i_3+1,i_4+1)$, and
    \item is non-decreasing in each coordinate
\end{itemize}
we take the simplex with the vertices that are in the path $P$ and the cube $C$. The decomposition of $\rr^4$ obtained by taking the Freudenthal decomposition of each unit cube is called the {\em Freudenthal decomposition of $\rr^4$}
or the {\em F-decomposition of $\rr^4$}.
}
\end{definition}

\begin{definition}
{\em 
   For each 4-dimensional simplex $\Delta$ in the $F$-decomposition of $\rr^4$ define $B_\Delta$ to be its barycenter.
    For each simplex $\Delta$ in the F-decomposition,  we define its {\em dual polyhedron} $\Delta^\vee$ to be the convex hull of the vertices $B_{\Delta_i}$, for each $\Delta \subseteq \Delta_i$.
}
\end{definition}

\begin{lemma}
\label{lem:dual-decomp}
Let $\Delta $ be a $k$-dimensional simplex of the $F$-decomposition of $\rr^4$. The dual $\Delta ^\vee$ is a $(4-k)$-dimensional polyhedron whose vertices are exactly the $B_{\Delta_i}$ for each $\Delta \subset \Delta_i$. Moreover, the dual polyhedra tile $\rr^4$.
\end{lemma}

\begin{proof}
   For a three or four-dimensional $\Delta$, the dual $\Delta ^ \vee$ is the convex hull of two or one vertices, respectively. In those cases, there is nothing to prove.
   
   For a 2-dimensional $\Delta$, we need to prove that the points $B_{\Delta_i}$ are all contained in two different hyperplanes and are not collinear.
   Since the Freudenthal decomposition is invariant under translation, permutation of coordinates, and symmetry with respect to the origin, we only need to check for $\Delta$ being   one of the following triangles:
   \begin{enumerate}
    \item Triangle with vertices $(0,0,0,0)$, $(1,0,0,0)$, and $(1,1,0,0)$.
    Let $L_1$ be the edge joining $(0,0,0,0)$ and $(1,0,0,0)$.
    We enumerate the possible $4$-dimensional simplices $S$ containing this triangle and compute the coordinates of their barycenters. 
    \begin{enumerate}
        \item Assume $L_1$ is the first line in the path that defines the simplex $S$. Then, the simplex $S$ has vertices with coordinates $(0,0,0,0)$, $(1,0,0,0)$, $(1,1,0,0)$, $(1,1,1,0)$, and $(1,1,1,1)$, up to a permutation of the coordinates $x_3$ and $x_4$.
        \item Assume $L_1$ is the second line in the path that defines the simplex $S$. Then, the simplex $S$ has vertices with coordinates $(0,0,0,-1)$, $(0,0,0,0)$, $(1,0,0,0)$, $(1,1,0,0)$, and $(1,1,1,0)$, up to a permutation of the coordinates $x_3$ and $x_4$.
        \item Assume $L$ is the third line in the path that defines the simplex $S$. Then, the simplex $S$ has vertices with coordinates $(0,0,-1,-1)$, $(0,0,0,-1)$, $(0,0,0,0)$, $(1,0,0,0)$, and $(1,1,0,0)$, up to a permutation of the coordinates $x_3$ and $x_4$.
    \end{enumerate}
    Hence, the barycenters lie in the hyperplanes $\{(x_1,x_2,x_3,x_4)\mid 3x_1+3x_2-2x_32x_4=3\}$ and $\{(x_1,x_2,x_3,x_4)\mid 8x_1-2x_2-2x_3-2x_4=4\}$. The barycenters; $\frac{1}{5}(4,3,2,1)$, $\frac{1}{5}(4,3,1,2)$, and $\frac{1}{5}(3,2,1,-1)$ are not collinear.
    
    \item Triangle with vertices $(0,0,0,0)$,
    $(1,0,0,0)$, and $(1,1,1,0)$.
    The barycenters of the $4$-dimensional simplices containing this triangle are the following:
    \begin{enumerate}
        \item Assume $L_1$ is the first line in the path that defines the simplex $S$. Then, the simplex $S$ has vertices with coordinates $(0,0,0,0)$, $(1,0,0,0)$, $(1,1,0,0)$, $(1,1,1,0)$, and $(1,1,1,1)$, up to a permutation of the coordinates $x_2$ and $x_3$.
        \item 
        Assume $L_1$ is the second line in the path that defines the simplex $S$. Then, the simplex $S$ has vertices with coordinates $(0,0,0,-1)$, $(0,0,0,0)$, $(1,0,0,0)$, $(1,1,0,0)$, $(1,1,1,0)$, up to a permutation of the coordinates $x_2$ and $x_3$.
    \end{enumerate}
    Hence, the barycenters lie in the hyperplanes $\{(x_1,x_2,x_3,x_4)\mid x_2+x_3-x_4=\frac{4}{5}\}$ and $\{(x_1,x_2,x_3,x_4)\mid 2x_1-x_2-x_3=\frac{3}{5}\}$. The barycenters; $\frac{1}{5}(4,3,2,1)$, $\frac{1}{5}(4,2,3,1)$, and $\frac{1}{5}(3,2,1,-1)$ are not collinear.
    
    \item Triangle with vertices $(0,0,0,0)$, $(1,0,0,0)$, and $(1,1,1,1)$.
    The simplices that contain this triangle have  vertices with coordinates $(0,0,0,0)$, $(1,0,0,0)$, $(1,1,0,0)$, $(1,1,1,0)$, and $(1,1,1,1)$, up to a permutation of the coordinates $(x_2,x_3,x_4)$. Thus, the barycenters are contained in the hyperplanes $\{(x_1,x_2,x_3,x_4)\mid x_1+x_2+x_3+x_4=2\}$ and $\{5x_1=4\}$, but are not collinear since they contain the points $\frac{1}{5}(4,3,2,1)$, $\frac{1}{5}(4,2,3,1)$, and $\frac{1}{5}(4,1,2,3)$.
    
    \item Triangle with vertices $(0,0,0,0)$, $(1,1,0,0)$, and $(1,1,1,1)$.
    The simplices that contain this triangle have vertices with coordinates $(0,0,0,0)$, $(1,0,0,0)$, $(1,1,0,0)$, $(1,1,1,0)$, and $(1,1,1,1)$ up to a permutation of the coordinates $x_1$ with $x_2$ and $x_3$ with $x_4$ .
    Hence, the barycenters are contained in the hyperplanes $\{(x_1,x_2,x_3,x_4)\mid x_1+x_2+x_3+x_4=2\}$ and $\{(x_1,x_2,x_3,x_4)\mid 5x_1+5x_2=7\}$, but are not collinear since they contain the points $\frac{1}{5}(4,3,2,1)$, $\frac{1}{5}(3,4,2,1)$, and $\frac{1}{5}(4,3,1,2)$.     
   \end{enumerate}
   
   For a 1-dimensional simplex $\Delta$, we need to prove that the points 
   $B_{\Delta_i}$ lie in the same 3-dimensional space and are not coplanar, i.e. they do not lie in exactly one hyperplane.
   Since the Freudenthal decomposition is invariant under translation and permutation of coordinates, we only need to check for $\Delta$ being one of the following lines
   \begin{enumerate}
    \item The line $L_1$. We describe the simplices $S$ containing $L_1$.
    \begin{enumerate}
        \item Assume $L_1$ is the first line in the path that defines the simplex $S$. Then, the simplex has vertices with coordinates $(0,0,0,0)$, $(1,0,0,0)$, $(1,1,0,0)$, $(1,1,1,0)$, and $(1,1,1,1)$, up to a permutation of the coordinates $(x_2,x_3,x_4)$.
        \item Assume $L_1$ is the second line in the path that defines the simplex $S$. Then, the simplex has vertices with coordinates $(0,0,0,-1)$, $(0,0,0,0)$, $(1,0,0,0)$, $(1,1,0,0)$, and $(1,1,1,0)$, up to a permutation of the coordinates $(x_2,x_3,x_4)$.
        \item Assume $L_1$ is the third line in the path that defines the simplex $S$. Then, the simplex has vertices with coordinates $(0,0,-1,-1)$, $(0,0,0,-1)$, $(0,0,0,0)$, $(1,0,0,0)$, and $(1,1,0,0)$, up to a permutation of the coordinates $(x_2,x_3,x_4)$.
        \item Assume $L_1$ is the last line in the path that defines the simplex $S$. Then, the simplex has vertices with coordinates $(0,-1,-1,-1)$, $(0,0,-1,-1)$, $(0,0,0,-1)$, $(0,0,0,0)$, and $(1,0,0,0)$ up to a permutation of the coordinates $(x_2,x_3,x_4)$.
    \end{enumerate}
    
    In each case, all the barycenters lie only in the hyperplane $\{(x_1,x_2,x_3,x_4)\mid 4x_1-x_2-x_3-x_4=2\}$.
    
    \item The line $L_2$ between the origin and  $(1,1,0,0)$.
    Now we will enumerate the simplices $S$ containing $L_2$.
    \begin{enumerate}
        \item  Assume $L_2$ connects the first and third vertices in the path that defines the simplex $S$. Then, the simplex has vertices with coordinates $(0,0,0,0)$, $(1,0,0,0)$, $(1,1,0,0)$, $(1,1,1,0)$, and $(1,1,1,1)$, up to a permutation of the coordinates $x_1$ with $x_2$ and the coordinates $x_3$ with $x_4$.  
        \item Assume $L_2$ connects the second and fourth vertices in the path that defines the simplex $S$. Then, the simplex has vertices with coordinates $(0,0,0,-1)$, $(0,0,0,0)$, $(1,0,0,0)$, $(1,1,0,0)$, and $(1,1,1,0)$, up to a permutation of the coordinates $x_1$ with $x_2$ and the coordinates $x_3$ with $x_4$.  
    
        \item Assume $L_2$ connects the third and fifth vertices in the path that defines the simplex $S$. Then, the simplex has vertices with coordinates $(0,0,-1,-1)$, $(0,0,0,-1)$, $(0,0,0,0)$, $(1,0,0,0)$, and $(1,1,0,0)$ up to a permutation of the coordinates $x_1$ with $x_2$ and the coordinates $x_3$ with $x_4$.
    \end{enumerate}
    
    In each case, all the barycenters lie in the hyperplane $\{(x_1,x_2,x_3,x_4)\mid 3x_1+3x_2-2x_3-2x_4=3\}$
    
    \item The line $L_3$ between the origin and  $(1,1,1,0)$
    We describe the simplices $S$ containing $L_3$.
    \begin{enumerate}
        \item Assume $L_3$ connects the first and fourth  vertices in the path that defines the simplex $S$. Then, the simplex has vertices with coordinates $(0,0,0,0)$, $(1,0,0,0)$, $(1,1,0,0)$, $(1,1,1,0)$, and $(1,1,1,1)$, up to a permutation of the coordinates $x_1,x_2$ and $x_3$.  
    
        \item Assume $L_3$ connects the second and fifth vertices in the path that defines the simplex $S$. Then, the simplex has vertices with coordinates $(0,0,0,-1)$, $(0,0,0,0)$, $(1,0,0,0)$, $(1,1,0,0)$, and $(1,1,1,0)$ up to a permutation of the coordinates $x_1,x_2$ and $x_3$.
    \end{enumerate}
    
    The only hyperplane containing the barycenters of these simplices is $\{(x_1,x_2,x_3,x_4)\mid 2x_1+2x_2+2x_3-3x_4=3\}$.
    
    \item The line $L_4$ between the origin and  $(1,1,1,1)$
    All the simplices containing $L_4$ must have coordinates $(0,0,0,0)$, $(1,0,0,0)$, $(1,1,0,0)$, $(1,1,1,0)$, and $(1,1,1,1)$, up to a permutation of the coordinates. Therefore, the barycenters lie in the hyperplane $\{(x_1,x_2,x_3,x_4)\mid x_1+x_2+x_3+x_4=2\}$
    and this is the only hyperplane that contains the barycenters.
   \end{enumerate}
   
    For $\Delta$ a point, we need to prove that the points $B_{\Delta_i}$ do not lie in the same 3-dimensional space.
    Without loss of generality, we can assume that $\Delta$ is the origin. 
    The 4-dimensional simplices in the Freudenthal decomposition of the cube with diagonal formed by the points $(0,0,0,0)$ and $(1,1,1,1)$ have barycenters whose coordinates are any permutation of $\{\frac{1}{5},\frac{2}{5},\frac{3}{5},\frac{4}{5}\}$. The only hyperplane containing these points is $\{(x_1,x_2,x_3,x_4)\mid x_1+x_2+x_3+x_4=2\}$. However, the barycenter of the simplex with vertices
    $(-1,-1,-1,-1)$, $(0,-1,-1,-1)$, $(0,0,-1,-1)$, $(0,0,0,-1)$, and $(0,0,0,0)$ is not contained in the hyperplane  $\{(x_1,x_2,x_3,x_4)\mid x_1+x_2+x_3+x_4=2\}$. Thus, $\Delta^\vee$ is 4-dimensional.

\end{proof}

\begin{definition}
{\em 
Let $\mathcal{F}$ be the F-decomposition of $\rr^4$.
By Lemma~\ref{lem:dual-decomp}, 
the set $\{ \Delta^\vee \mid \Delta \in \mathcal{F} \}$
gives a convex polyhedral decomposition of $\rr^4$.
This decomposition will be called
the {\em Freudenthal dual decomposition of $\rr^4$}
or simply the {\em F-dual decomposition}.
Note that both the F-decomposition
and the F-dual decomposition
can be regarded as complexes. For the $F$-decomposition, we can use the lattice $\zz^4$, while for the $F-$dual decomposition, we can use the lattice $(\frac{1}{5}\zz)^4$.
}
\end{definition} 

\subsection{Subcomplexes of the F-decomposition}
\label{subsec:subcomp-F-decomp}
In this subsection, we prove the following proposition regarding subcomplexes of the F-decomposition.

\begin{definition}
{\em 
Let $M$ be a subcomplex of the Freudenthal decomposition in $\rr^4$.
We define the {\em dual} $M^\vee$ of $M$ in the following way:
For each vertex $v\in M$ 
the dual polyhedron $v^\vee$ is in $M^\vee$. 
We write $\partial M^\vee$ for the boundary of the dual.
}
\end{definition}

\begin{proposition}\label{prop:smoothing}
Let $M$ be a subcomplex of the F-dual decomposition in $\rr^4$.
Let $M_0$ be the boundary of the dual of $M$.
Then, there exists a $3$-dimensional smooth polyhedral complex $M'$ that is homotopy
equivalent to $M_0$.
\end{proposition}

\begin{proof}
We will prove the statement in several steps.
First, we will study the polyhedral decomposition of $M_0$.
Then, we will proceed by performing blow-ups
on $M_0$ in order to obtain a smooth polyhedral complex.
Recall from Remark~\ref{rem:faces-blow-up} that blow-ups may not exist in the general setting.
Hence, we will need to carefully analyze the nerves at each simplex to argue that the blow-up does exist.\\

\textit{Step 1:} In this step, we study the polyhedra in $M_0$ and its nerves. We show the following statements:
\begin{enumerate}
    \item [(i)] The maximal polyhedra of $M_0$ are truncated octahedrons and hexagonal prisms.
    \item[(ii)] The nerves of the vertices of $M_0$ are $3$-dimensional simplices
    or triangular prisms.
    \item[(iii)] The nerves of the edges of $M_0$ are triangles or quadrilaterals.\\
\end{enumerate}

The maximal polyhedra in $M_0$ are 3-dimensional polyhedra of the dual Freudenthal decomposition of $\rr^4$. Thus, they are truncated octahedra or hexagonal prisms. This shows (i).\\

   A 3-dimensional polyhedron of the dual Freudenthal decomposition is in $M_0$ if and only if it is a face of a 4-dimensional polyhedron not in $M^\vee$ and the face of a 4-dimensional polyhedron in $M^\vee$.
   Therefore, a 3-dimensional polyhedron is in $M_0$ if its dual edge has one vertex in $M$ and one vertex not in $M$.
   For a vertex $v$ in $M_0$, its dual in the Freudenthal decomposition is a 4-dimensional simplex $\Delta$. The maximal polyhedra of $M_0$ containing $v$ correspond to the edges of $\Delta$ that have one vertex in $M$ and the other vertex not in $M$. Two of these maximal polyhedra intersect in a $1$ or $2$-dimensional cell if the corresponding edges share a $3$ or  $2$-dimensional cell in $\Delta$, respectively. The vertices in the nerve of $v$ are connected by an edge if the corresponding maximal polyhedra of $M_0$ share a $2$-dimensional cell. This is the same as asking for the corresponding edges in the Freudenthal decomposition to be in the same $2$-dimensional cell.
   Since $\Delta$ has 5 vertices there are combinatorially 3 possibilities:
   \begin{enumerate}
       \item The simplex $\Delta$ contains 0 or 5 vertices of $M$. Then no maximal polyhedron of $M_0$ contains $p$. So, the point $v$ is not in $M_0$.
       \item The simplex $\Delta$ contains 1 or 4 vertices of $M$. Then exactly 4 edges of $\Delta$ have one vertex in $M$ and one vertex outside of $M$. Since all these edges in $\Delta$ share a common vertex, each pair is contained in a common triangle. Therefore, in the nerve of $v$ in $M_0$ all the points are joined by an edge, hence it is a 3-dimensional simplex. 
       \item The simplex $\Delta$ contains 2 or 3 points of $M$. Then exactly 6 edges of $\Delta$ have one vertex in 
       $M$ and one vertex outside of $M$. Call $A,B$, $1,2,3$ the points of $\Delta$. Where $A$ and $B$ are both inside or outside of $M$ and the same holds for $1,2$, and $3$. The edge formed by points $\{1,A\}$ shares a triangle only with $\{1,B\}$, $\{2,A\}$ and $\{3,A\}$. We can proceed similarly for each other of the 6 edges. Hence, the nerve of $v$ in $M_0$ is a triangular prism. 
   \end{enumerate}
This finishes the proof of (ii).\\

   For an edge $E$ in $M_0$ its dual in the Freudenthal decomposition is a 3-dimensional simplex $\Delta$. The maximal polyhedra of $M_0$ containing $E$ correspond to the edges of $\Delta$ that have one vertex in $M$ and the other vertex not in $M$. Two of these maximal polyhedra intersect in a $2$-dimensional cell if the corresponding edges share a  $2$-dimensional cell in $\Delta$. Hence, the vertices in the nerve of $E$ are connected by an edge if the corresponding edges in the Freudenthal decomposition are in the same $2$-dimensional cell.

    Since $\Delta$ has 4 vertices there are combinatorially 3 possibilities:

    \begin{enumerate}
        \item The simplex $\Delta$ contains 0 or 4 vertices of $M$. Then, no maximal polyhedron of $M_0$ contains $v$. So, the vertex $v$ is not in $M_0$.
        \item The simplex $\Delta$ contains 1 or 3 vertices of $M$. Hence, exactly 3 edges of $\Delta$ have one vertex in $M$ and one vertex outside of $M$. Therefore, the nerve of $E$ in $M_0$ is a 2-dimensional simplex.
        \item
        The simplex $\Delta$ contains 2 vertices of $M$. Hence, exactly 4 edges of $\Delta$ have one vertex in $M$ and one vertex outside of $M$. Therefore, the nerve of $E$ in $M_0$ is a quadrilateral.
    \end{enumerate}
This finishes the proof of (iii).\\

\textit{Step 2:} In this step, we show that the neighborhood of any non-smooth vertices of $M_0$ admits an embedding in $\qq^3$ as a strongly polytopal fan.\\

A vertex $v$ in $M_0$ corresponds to a simplex $\Delta$ in the Freudenthal decomposition that has $1,2,3$ or $4$ vertices in $M$. If it corresponds to a non-smooth vertex, then it must have $2$ or $3$ vertices in $M$, without loss of generality, we will assume that two vertices are in $M$.
    Up to a permutation of coordinates, we can assume that the simplex has vertices $(0,0,0,0)$, $(1,0,0,0)$, $(1,1,0,0)$, $(1,1,1,0)$, and $(1,1,1,1)$.

    The vertex $v$ in $M_0$, then has coordinates $\frac{1}{5}(4,3,2,1)$.
    The vertices of $M_0$ that have edges connecting to $v$ will correspond to the barycenters of $4$-dimensional simplices sharing a $3$-dimensional simplex with $\Delta$. Hence, they will be one of the following: 
\[
\frac{1}{5}(3,2,1,-1), \frac{1}{5}(4,3,1,2),  \frac{1}{5}(4,2,3,1),  \frac{1}{5}(3,4,2,1), \text{ and }\frac{1}{5}(6,4,3,2).
\]
Therefore, after translating to the origin the lattice generators of the edges in $M_0$ with vertex $v$ will be:
\[ 
A:=\frac{1}{5}(-1,-1,-1,-2),  B:=\frac{1}{5}(0,0,-1,1), C:=\frac{1}{5}(0,-1,1,0),  D:=\frac{1}{5}(-1,1,0,0), \text{ and } E:=\frac{1}{5}(2,1,1,1).
\]
As all the cones are smooth, defining linear maps for each of the cones to a common $\zz^4$ is the same as defining an
element in $\qq^3$ for each lattice ray generator in $\zz^4$. The edges of $\Delta$ that have one vertex in $M$ and one vertex not in $M$ correspond to the $3$-dimensional polytopes in $M_0$. Up to symmetry with respect to the origin, there are $6$ possibilities for the 2 vertices in $M$.
    \begin{enumerate}
        \item The vertices are $(0,0,0,0)$ and $(1,0,0,0)$.
        Hence, the cones in $M_0$,  will have generators: $BDE$, $BCE$, $BCD$, $ADE$, $ACE$, and $ACD$. In this case, we define the map as mapping the vertices $A$, $B$, $C$, $D$, and $E$ to $(0,0,1)$, $(0,0,-1)$, $(1,0,0)$, $(0,1,0)$, and $(-1,-1,0)$, respectively.
        \item The vertices are $(0,0,0,0)$ and $(1,1,0,0)$.
        So, the cones in $M_0$ have generators: $CDE$, $BCE$, $BCD$, $ADE$, $ABE$. and $ABD$. We can define the map as sending the vertices $A$, $B$, $C$, $D$, and $E$ to $(0,0,1)$, $(1,0,0)$, $(0,0,-1)$, $(0,1,0)$, and $(-1,-1,0)$, respectively.
        \item The vertices are $(0,0,0,0)$ and $(1,1,1,0)$.
        The cones in $M_0$ have generators: $CDE$, $BDE$, $BCD$, $ACE$, $ABE$ and $ABC$. We consider the map sending the vertices $A$, $B$, $C$, $D$, and $E$ to $(0,0,1)$, $(1,0,0)$, $(0,1,0)$, $(0,0,-1)$, and $(-1,-1,0)$, respectively.     
        \item The vertices are $(0,0,0,0)$ and $(1,1,1,1)$.
        Hence, the cones in $M_0$,  will have generators: $CDE$, $BDE$, $BCE$, $ACD$, $ABD$, and $ABC$. 
        We define the map as sending the vertices $A$, $B$, $C$, $D$ and $E$ to $(0,0,1)$, $(1,0,0)$, $(0,1,0)$, $(-1,-1,0)$, and $(0,0,-1)$, respectively.
        \item The vertices are $(1,0,0,0)$ and $(1,1,0,0)$.
        The cones in $M_0$ have generators: $CDE$, $ACE$, $ACD$, $BDE$, $ABE$, and $ABD$. 
        We consider the morphism that maps the vertices $A$, $B$, $C$, $D$, and $E$ to $(1,0,0)$, $(0,0,1)$, $(0,0,-1)$, $(0,1,0)$, and $(-1,-1,0)$, respectively.
        \item The vertices are $(1,0,0,0)$ and $(1,1,1,0)$.
        Hence, the cones in $M_0$ will have generators: $DCE$, $ADE$, $ACD$, $BCE$, $ABE$ and $ABC$. 
        We define the map as sending the vertices $A$, $B$, $C$, $D$, and $E$ to $(1,0,0)$, $(0,0,1)$, $(0,1,0)$, $(0,0,-1)$, and $(-1,-1,0)$, respectively.
\end{enumerate}
In any case, we have lattice maps that send each of the cones in $\mathcal{P}$ with vertex $v$ to a common fan in $\zz^3$. As the vertices $(0,0,1)$, $(0,0,-1)$, $(1,0,0)$, $(0,1,0)$, and $(-1,-1,0)$ form a convex bipyramid, the fan in $\zz^3$ is strongly polytopal.\\

\textit{Step 3:} In this step, we produce a polyhedral complex $M_1$ obtained from $M_0$ by a sequence of blow-ups at points.\\

By Lemma~\ref{lem:blow-up-dim-3} and Step 2, we can perform blow-ups at each non-smooth point of $2M$. 
The construction of the blow-up in the proof of Lemma~\ref{lem:blow-up-dim-3} depends on the choice of:
an embedding into $\mathbb{Q}^3$, a piecewise linear function $\mathbb{Q}^3 \rightarrow \mathbb{Q}^4$, and a hyperplane $H$ in $\mathbb{Q}^4$.
By the second step, 
we can choose the same embedding, piecewise linear function, and hyperplane
for every two vertices in $M_0$ with isomorphic neighborhoods.
By doing so, we produce a polyhedral complex $M_1$ where all the non-smooth edges are disjoint.\\

\textit{Step 4:} In this step, we blow-up at the edges of $M_1$ to obtain a smooth complex.
\\

To do so, we show that the blow-up of $M_1$ at any edge with nerve a quadrilateral exists.

In $M_0$ the only nerves of edges that are not 2-dimensional simplices are quadrilaterals. The complex $M_1$ is obtained by performing blow-ups at vertices of $M_0$ and all the nerves at 2-dimensional cells are simplices. Therefore, by Lemma~\ref{lem:nerve-blowup}, all the new edges in $M_1$ have simplices as nerves. 
Hence, by Step 1, the only edges in $M_1$, without simplices as nerves were edges in $M_0$ without simplices as nerves. Such edge had non-simplicial nerves in both vertices, hence there were blow-ups performed in $M_0$ at both of these vertices.
    
Let $L$ be one such edge, with vertices $A$ and $B$. In $M_1$ this edge is contained in four 3-dimensional simplices. The points $A$ and $B$ are contained in these four 3-dimensional simplices and also in polyhedrons $P_A$ and $P_B$, respectively.
The polyhedrons $P_A$ and $P_B$ are the result of the blow-ups at points $A'$ and $B'$ in $M_0$.

The cone at $P_A$ with vertex $A$ and  the cone at $P_B$ with vertex $B$ are isomorphic by the choice of blow-ups in Step 2.
We take a hyperplane $H_A$ in $P_A$ separating $A$ and the other vertices in $P_A$. We define a hyperplane $H_B$ in $P_B$ by the same equations that define  $H_A$ in $P_A$.
The intersection points of $H_A \cap P_A$ and $H_B \cap P_B$ define four points in each of the four 3-dimensional polyhedra that contain $L$. In each of these 3-dimensional polyhedra, the 2 points coming from $R_A$ are joined to the 2 points coming from $R_B$ with lines parallel to $L$. Therefore, these four points define a parallelogram. 
Therefore, we define the blow-up at $L$ by cutting with the hyperplanes defined by the parallelograms.
The edge $L$ is replaced with the quadrilateral prism with faces given by the parallelograms in the 3-dimensional polyhedra containing $L$ and the quadrilaterals in $H_A$ and $H_B$. This is the prism over the quadrilateral $H_A\cap P_A \cong H_B \cap P_B$.     

In $M_1$ the only strata with non-simplicial nerves are disjoint edges and the points in these edges.
Hence, after blowing-up these edges, by Lemma~\ref{lem:nerve-blowup}, all the nerves are simplicial. Therefore, after performing the aforementioned blow-up at each non-smooth edge, we end up with $M_2$ a smooth polyhedral complex.\\

\textit{Step 5:} In this step, we define $M':=M_2$ and finish the proof.\\

Observe that $M'$ is a smooth polyhedral complex of dimension $3$ by construction.
Furthermore, since blow-ups are simply homotopy equivalences, 
we conclude that $M'$ is homotopic to $M_0$.
This finishes the proof.
\end{proof} 

\subsection{Free groups in dimension 4}
\label{subsec:free-in-dim-4}
In this subsection, we show that every free group appears as the fundamental group
of a log canonical $4$-dimensional singularity.

\begin{proof}[Proof of Theorem~\ref{introthm:free-in-4}]
In view of Theorem~\ref{introthm:from-sm-poly-comp-to-lc-sing}
it suffices to prove the existence of a smooth $3$-dimensional polyhedral complex $\mathcal{P}_r$
for which $\pi_1(\mathcal{P}_r)=\pi_1(M)$.
As $M$ is compact and smooth in $\rr^4$, there exists large enough $N$, such that the union of the lattice cubes of size $\frac{1}{N}$ that intersect $M$ gives a set homotopic to a tubular neighborhood of $M$.
Equivalently dilating $M$ by a factor of $N$ and taking all the unit cubes intersecting $M_{N}$ gives us a subcomplex of the $F$-decomposition of $\rr^4$ homotopic to a tubular neighborhood of $M$. Call $M'$ this subcomplex of the $F$-decomposition.
The boundary of the dual complex of $M'$ in the F-dual decomposition of $\mathbb{R}^4$ is two disjoint complexes, homotopic to $M$. Let $M_0$ be one of these subcomplexes.

By Proposition~\ref{prop:smoothing}
there exists a smooth $3$-dimensional
polyhedral complex $\mathcal{P}_r$
for which $\pi_1(\mathcal{P}_r)=\pi_1(M)$.
This finishes the proof.
\end{proof}

\begin{proof}[Proof of Theorem~\ref{introthm:snc-CY-dim-3}]
This follows from the proof of Theorem~\ref{introthm:free-in-4}.
\end{proof}

\begin{proof}[Proof of Theorem~\ref{introthm:3-manifold}]
The 3-manifold $M_0$ in the proof of
Theorem~\ref{introthm:free-in-4}
is homotopic to $\#^r(S^2\times S^1)$.
Then, the statement follows from 
the proof of Proposition~\ref{prop:smoothing}.
\end{proof}

\section{Examples and questions} 
\label{sec:e-and-q}

In this section, we collect some examples and questions for further research.
The following example shows that fundamental groups of lc singularities can be infinite. 
We also describe an expectation about the fundamental group of 
lc cone singularities.

\begin{example}\label{ex:cone-over-smooth-CY}
{\rm 
The simplest way to construct singularities, 
both algebraically and topologically, 
is by taking cones over complete geometric objects. 
It is well-known that cones over smooth Fano varieties with respect to the polarization $-K_X$ lead to 
isolated klt singularities.
These singularities have finite fundamental groups~\cite[Theorem 1]{Bra21}. 
Analogously, the cone over a smooth Calabi--Yau variety, with respect to any polarization, gives a log canonical singularity.

For instance, we can consider an Abelian variety $A$ of dimension $n$.
Let $P_A$ be a very ample divisor on $A$.
We can consider the affine cone
\[
C(A,P_A):={\rm Spec}\left( 
\bigoplus_{m\geq 0} 
H^0(A,\mathcal{O}_A(mP_A))
\right).
\]
The singularity 
$(C(A,P_A);c)$ is log canonical, where
$c$ is the vertex point (see, e.g.,~\cite{Kol13}).
For the fundamental group of the cone, 
there is an exact sequence 
\[
1\rightarrow 
\zz 
\rightarrow 
\pi_1^{\rm loc}(C(A,P_A);c)
\rightarrow 
\zz^{2n}
\rightarrow 
1,
\]
where $\zz$ is generated by the loop $\gamma_E$ 
around the exceptional divisor of the blow-up of $c\in C(A,P_A)$.
Thus, the image of $\zz$ in $\pi_1^{\rm loc}(C(A,P_A);c)$ defines a central element.
Indeed, we can see that this element is central as the normal bundle of $E$ on $Y$ trivializes on some open subset of $E$.
This implies that 
$\pi_1^{\rm loc}(C(A,P_A;c))$ is a nilpotent group of nilpotency order $2$ with a nilpotent basis of rank $2n+1$.
The fundamental group
of a smooth Calabi--Yau variety
of dimension $n$ 
is virtually nilpotent with 
a nilpotent basis of length at most $2n$ (see, e.g.,~\cite[Corollary 2]{KW11}).
Thus, the process of taking cones over smooth Calabi--Yau varieties 
should lead to log canonical singularities
whose local fundamental groups
are virtually nilpotent with a nilpotent basis of length at most $2n+1$.
The case of cones over Abelian varieties should be the worse one, in terms of the length of the nilpotent basis.

We remark that even for log Calabi--Yau pairs of dimension $n$ the regional fundamental group
is expected to be virtually nilpotent
with a nilpotent basis of rank at most $2n$.
In summary, the local fundamental groups 
of log canonical cone singularities
can be infinite (unlike log terminal singularities). 
However, the fundamental groups
of these examples are still somehow controlled as nilpotent groups. 
Thus, these fundamental groups are still close to abelian groups.
}
\end{example} 

The following example shows that fundamental groups
of log canonical singularity
may be infinite and not nilpotent.

\begin{example}\label{ex:cone-over-surf-group}
{\em 
Let $S$ be a Riemann surface of genus $g$.
In Theorem~\ref{introthm:surf-group-in-3}, 
we show that there exists 
an isolated lc $3$-fold singularity
$(X_S;x)\simeq \pi_1(S)$. 
If $g\geq 2$, then the center of $\pi_1(X_S;x)$ is trivial.
Thus, the central sequence for
$\pi_1(X_S;x)$ terminates with the trivial subgroup. 
This shows that the expectation in Example~\ref{ex:cone-over-smooth-CY}
does not hold if the dual complex of the singularity is higher-dimensional, i.e., if the regularity is positive.
In these examples, the dlt modification 
$(Y_S,E_S)$ of $(X_S;x)$ satisfies
that $\mathcal{D}(E_S)$ is homotopic
to the Riemann surface $S$.
Moreover, we have that $\pi_1(E_S)\simeq \pi_1(S)$.
}
\end{example}

The following example shows that the fundamental group of an lc $3$-fold singularity can be interesting even 
if the associated dual complex is a $2$-sphere.

\begin{example}\label{ex:spherical-dc}
{\em 
Let $E$ be the torus invariant boundary of $\pp^{n+1}$.
Then, $E$ is a simple normal crossing Calabi--Yau variety, i.e., 
it has snc singularities
and $K_E\sim 0$.
Furthermore, $E$ is obtained by gluing $n+1$ copies of $\pp^n$ forming an $(n+1)$-simplex.
Thus, the dual complex $\mathcal{D}(E)$ is an $n$-sphere.
Consider the polarization $\mathcal{L}_{E,m}$
that restricts to the line bundle
$\mathcal{O}_{\pp^n}(mH)$ on each irreducible component $E_i$ isomorphic to $\pp^n$.
By Subsection~\ref{subsec:from-poly-to-toric} and Subsection~\ref{subsec:from-toric-to-lc}, we know that there exists an lc  singularity $(X_{n,m};x)$ of dimension $(n+1)$
with a dlt modification $(Y,E)$ and $E$ having normal bundle $\mathcal{L}_{E,m}^\vee$.
This construction depends on the choice of certain very ample divisors. In this case, it suffices to take $m> n+1$.
Let $L_i\subset E_i$ be a general line,
Hence, we have that 
$E_j\cdot L_i=1$ for $j\neq i$
and $E_i\cdot L_i=m-n-1$.
Let $\gamma_i$ be the loop around $E_i$. 
Then, all the loops $\gamma_i$ commute.
Furthermore, we have the relations
\[
\gamma_i^{m-n-1}=\prod_{j\neq i} \gamma_j.
\]
From these relations, one can deduce that 
$\pi_1(X_{n,m};x)$ is a finite abelian group
of rank at most $n$. 

For instance, we can consider $n=3$
and choose the polarization 
$\mathcal{L}_{E,5}$
on the $2$-dimensional snc Calabi--Yau variety $E$.
Then, we obtain the generator's relations 
\[
\gamma_1^{-2}\gamma_2\gamma_3\gamma_4,
\gamma_1\gamma_2^{-2}\gamma_3\gamma_4,
\gamma_1\gamma_2\gamma_3^{-2}\gamma_4,\text{ and }
\gamma_1\gamma_2\gamma_3\gamma_4^{-2}.
\]
We conclude that
\[
\pi_1(X_{2,5};x) \simeq (\zz/3\zz)^3.
\]
This example is interesting as it is closely related to toric singularities.
The regional fundamental group of an $n$-dimensional toric pair is a finite abelian group
of rank at most $n$.
In this example, we have an $n$-dimensional isolated lc singularity
with spherical dual complex
and fundamental group
a finite abelian group of rank $n$.
}
\end{example}

The following example shows that all the possible ranks in Theorem~\ref{introthm:2-dim-lcy} can happen.

\begin{example}\label{ex:possible-ranks}
{\em 
Consider the following log Calabi--Yau surfaces:
\begin{itemize}
    \item A smooth K3 surface,
    \item the pair $(\pp^2,L+C)$,
    where $L$ is a line and $C$ a transversal conic, 
    \item the pair $(\pp^2,L_1+L_2+L_3)$,
    where the $L_i$'s are transversal lines,
    \item the product $E\times (\pp^1,\{0\}+\{\infty\})$, where $E$ is an elliptic curve, and 
    \item an abelian surface.
\end{itemize}
Then, the rank of $\pi_1^{\rm reg}(X,B)$ in the previous examples are $\{0,1,2,3,4\}$, respectively.
}
\end{example}

We finish this section with a couple of questions. 
In Table~\ref{table1}, we have a complete description of all the possible isomorphism classes of regional fundamental groups of lc surface singularities.
Theorem~\ref{introthm:surf-group-in-3}
and Theorem~\ref{introthm:not-free-in-3}
give some new examples 
and restrictions for the fundamental groups in dimension $3$.
However, the full picture in dimension $3$ is still not clear.
An answer to the following question would enhance our understanding of the fundamental groups in dimension $3$.

\begin{question}\label{quest:3-dim-extension}
Let $G$ be a finite cyclic extension of a surface group.
Does there exist a $3$-fold lc singularity $(X;x)$ for which $\pi_1^{\rm loc}(X;x)\simeq G$?
\end{question}

The methods described in this paper can be exploited to produce 
lc singularities with smooth dual complexes. These examples are very interesting. However, we lack the machinery to produce singular dual complexes yet. 
To do so, there is a natural thing to try:
study smooth polyhedral complexes with finite actions
and try to realize the polyhedral quotient as a dual complex.

\begin{question}
Is it possible to perform an equivariant version of polyhedral complexes
to construct lc singularities
with singular dual complexes? 
\end{question}

Still, this construction would only lead to geometric orbifolds as dual complexes. 
One would need to develop slightly different machinery to obtain arbitrary orbifolds. 
Constructing examples of $4$-fold lc singularities
with singular dual complexes is harder.
However, it is fairly easy to construct log Calabi--Yau $4$-folds with singular dual complex.
This can happen as quotients
$(T,B_T)/G$ of a $4$-dimensional toric Calabi--Yau pair $(T,B_T)$ that admits a finite subgroup $G<{\rm Aut}(T,B_T)$.

Our main theorem points in the direction 
that any finitely presented fundamental group could appear as the local fundamental group
of an lc singularity.
We do not know yet the existence of a finitely presented group
that does not appear
as the fundamental group of a $4$-dimensional lc singularity.

\begin{question}
Does there exist a finitely presented group that it is not the fundamental group
of a $4$-dimensional lc singularity?
\end{question}

We expect the answer to the previous question to be yes.
However, at the same time, we expect that every finitely presented group appears in dimension $5$. This still leaves open the question about fundamental groups of rational log canonical singularities.
These singularities are expected to behave much more like a klt singularity.
This leads to the following questions.

\begin{question}
Is there any restriction for the fundamental group of a rational lc singularity? 
\end{question}

As we discussed several times throughout the article, these fundamental groups are closely related to the fundamental group of the underlying dual complex.
Thus, to obtain interesting examples for the previous question one needs to consider $5$-dimensional singularities.
Indeed, the fundamental groups of homology spheres of dimension $3$ are more restrictive.
The machinery introduced in this article should allow us to tackle these questions.

\appendix

\section{Fundamental groups of lc surface singularities}
\label{sec:fun-lc-surf}

The following tables summarize the possible fundamental groups of log canonical singularities of surfaces. It follows from Proposition~\ref{prop:caso-eliptic}, Proposition~\ref{prop:caso-cycle-rational}, the proof of Proposition~\ref{prop:B_s=0}, and the proof of Proposition~\ref{prop:caso-chain-rational}.
In the first column, we describe the exceptional divisor of the minimal resolution.
In the second column, we write down the coefficients of the standard approximation of the boundary divisor and describe the intersection of its strict transform with $E$.
In the third column, we describe the isomorphism class of the regional fundamental group.\\

\begin{center}
\captionof{table}{Fundamental groups of lc surface singularities}
\label{table1}
\begin{longtable}[H]{ | m{11em} | m{14em} | c | } 
\hline
Exceptional divisor $(E)$ & Boundary $(B_s)$ & $\pi_1^{\rm reg}(X,B;x)$ \\ 
 \hline \hline 
 
Elliptic curve & $B_s=0$  & \ $\zz \rtimes \zz^2$\\

\hline
 
Cycle of rational curves & $B_s=0$ & $\zz^2 \rtimes \zz$ \\
 
\hline
 
A rational curve intersected by $4$ other $(-2)$-curves & $B_s=0$ &  $\langle a,b,c \mid a^{2}b^{-2},a^{2}c^{-2}, a^{2}(a^{2m-1}b^{-1}c^{-1})^{-2}\rangle$ \\ 
\hline 
A rational curve intersected by 3 chains of rational curves & $B_s=0$  & $\langle a,b,c \mid a^{A}b^{-B},a^{A}c^{-C},a^{A'}b^{B'}c^{C'} \rangle$ \\ 
\hline
Chain of rational curves & $B_s=0$  & $\zz /n\zz$ \\
\hline
 
A chain of rational curves intersected by $2$ other $(-2)$ curves in one end & $B_s=0$ & $\langle a,b \mid a^2b^{-2},a^{A}(ab)^{B} \rangle$ \\

\hline 
 
A chain of rational curves intersected by $2$ other $(-2)$ curves in each end & $B_s=0$ & 

{$\Bigl\langle 
       \begin{array}{l|}
 a,b,c                                          
        \end{array}
        \begin{array}{c}
             a^2b^{-2},a^{A}(ab)^{B}c^{-2},  \\
             c^{2}(a^{A'}(ab)^{B'}c^{C'})^{-2}  
        \end{array}
     \Bigr\rangle$}
 \\

\hline 
 
$E=0$ & $B_s=\frac{1}{2}B_1+\frac{1}{2}B_2$ & $\zz / 2 \zz \times \zz / 2\zz $ \\
\hline

$E=0$ & $B_s=\frac{m_1-1}{m_1}B_1$ &  $\zz / m_1 \zz$\\
\hline 

A rational curve intersected by $3$ other $(-2)$ curves & $B_s=\frac{1}{2}B_1$ & $\langle a,b,c \mid a^2b^{-2},a^2c^{-2},(a^{1-2m}bc)^2\rangle$\\

\hline

A rational curve intersected by $2$ other $(-2)$ curves & $B_s=\frac{1}{2}B_1+\frac{1}{2}B_2$ &  $\langle a,b,c \mid a^2b^{-2},c^2,[a^2,c],(a^{1-2m}bc)^2\rangle$ \\

\hline

A rational curve intersected by $1$ other $(-2)$ curves & $B_s=\frac{1}{2}B_1+\frac{1}{2}B_2+ \frac{1}{2}B_3$ & $\langle a,b,c \mid b^2,c^2,[a^2,b],[a^2,c],(a^{1-2m}bc)^2\rangle$ \\

\hline

Rational curve & $B_s=\frac{m_1-1}{m_1}B_1+\frac{m_2-1}{m_2}B_2+\frac{m_3-1}{m_3}B_3$ & 

{$\Bigl\langle 
       \begin{array}{l|}
 a,b,x                                          
        \end{array}
        \begin{array}{c}
            a^{m_1},b^{m_2},[a,x],[b,x], \\(b^{-1}a^{-1}x^{m})^{m_3}
            
        \end{array}
     \Bigr\rangle$}

\\

\hline 

  Chain of rational curves & $B_s=\frac{m_1-1}{m_1}B_1$, not intersecting $E$ in an end curve & 

{$\Bigl\langle 
       \begin{array}{l|}
 a,b,c,x                                          
        \end{array}
        \begin{array}{c}
             a^{m_1},[a,x],xb^{-B},\\xc^{-C},ab^{B'}c^{C'}x^{-m}
        \end{array}
     \Bigr\rangle$}
\\

\hline

Chain of rational curves, where the last curve is a $(-2)$-curve & $B_s=\frac{1}{2}B_1$, intersecting $E$ in the second to last curve & $\langle a,b,x \mid a^2, xb^{-2},[a,x],(ab)^{A}x^{B}\rangle$\\

\hline

Chain of rational curves, where each end curve is a $(-2)$-curve  & $B_s=\frac{1}{2}B_1+\frac{1}{2}B_2$, $B_1$ intersecting $E$ in the second curve and $B_2$ intersecting $E$ in the second to last curve & 

{$\Bigl\langle 
       \begin{array}{l}
 a,b,c                                          
        \end{array}
        \begin{array}{|c}
            b^2,[a^2,b],a^{2A}(ab)^{B}c^{-2} ,\\ (a^{2A'}(ab)^{B'}c^{-2C'+1})^2
        \end{array}
     \Bigr\rangle$}\\

\hline

Chain of rational curves, where the last curve is a $(-2)$-curve & $B_s=\frac{1}{2}B_1+\frac{1}{2}B_2+\frac{1}{2}B_3$, 

$B_1$ intersecting $E$ in the second to last curve and $B_2$, $B_3$ intersecting the first curve. &

{$\Biggl\langle 
       \begin{array}{c}
 a,b,\\ c                                          
        \end{array}
        \begin{array}{|c}
             b^2, [a^2,b],c^2,[a^{2A}(ab)^{B},c],\\ (a^{2A'}(ab)^{B'}(a^{2A}(ab)^{B})^{-C}c)^2
        \end{array}
     \Biggr\rangle$}\\

\hline

Rational curve & $B_s=\frac{1}{2}B_1+\frac{1}{2}B_2+ \frac{1}{2}B_3+ \frac{1}{2}B_4$ &

{$\Bigl\langle 
       \begin{array}{l|}
 a,b,c,x                                          
        \end{array}
        \begin{array}{c}
         [a,x],[b,x], [c,x] \\
 a^2,b^2,c^2, (abcx^{m})^{2}
        \end{array}
     \Bigr\rangle$}\\

\hline

Chain of rational curves  & $B_s=\frac{1}{2}B_1+\frac{1}{2}B_2+\frac{1}{2}B_3+\frac{1}{2}B_4$

$B_1$ and $B_2$ intersecting $E$ in an end curve and $B_3$ and $B_4$ intersecting $E$ in the other end curve. &

{$\Biggl\langle 
       \begin{array}{c|}
 a,b,\\ c,x                                          
        \end{array}
        \begin{array}{c}
            [c,(ab)^{A}x^{B}], [b,x], \\a^{2}, b^{2},c^2,[a,x],\\ ((ab)^{A'}x^{B'}c)^{2}\\ 
        \end{array}
     \Biggr\rangle$}\\

\hline

Chain of rational curves  & $B=\frac{m_1-1}{m_1}B_1$, intersecting $E$ in an end curve & $\langle 
  a,x \mid  a^{m_1},[a,x],(a)^{A}x^{B}
 \rangle$\\

\hline

Chain of rational curves  & $B_s=\frac{m_1-1}{m_1}B_1+\frac{m_2-1}{m_2}B_2$, each $B_i$ intersecting $E$ in a different end curve & $\langle 
  a,x \mid  a^{m_1},[a,x], ((a)^{A}x^{B})^{m_2}
 \rangle$ \\
\hline

Chain of rational curves & $B_s=\frac{m_1-1}{m_1}B_1+\frac{1}{2}B_2+\frac{1}{2}B_3$

$B_1$ intersecting $E$ in an end curve and $B_2$ and $B_3$ intersecting $E$ in the other end curve & $\langle a,b,x \mid a^2,[a,x],b^2,[b,x], ((ab)^{A}x^{B})^{m_1}\rangle$ \\

\hline
\end{longtable}
\end{center}

\bibliographystyle{habbvr}
\bibliography{bib}

\begin{thebibliography}{10}
\expandafter\ifx\csname url\endcsname\relax
  \def\url#1{\texttt{#1}}\fi
\expandafter\ifx\csname doi\endcsname\relax
  \def\doi#1{\burlalt{doi:#1}{http://dx.doi.org/#1}}\fi
\expandafter\ifx\csname urlprefix\endcsname\relax\def\urlprefix{URL }\fi
\expandafter\ifx\csname href\endcsname\relax
  \def\href#1#2{#2}\fi
\expandafter\ifx\csname burlalt\endcsname\relax
  \def\burlalt#1#2{\href{#2}{#1}}\fi

\bibitem{Ale93}
V.~Alexeev.
\newblock Two two-dimensional terminations.
\newblock {\em Duke Math. J.}, 69(3):527--545, 1993.
\newblock \doi{10.1215/S0012-7094-93-06922-0}.

\bibitem{Bra21}
L.~Braun.
\newblock The local fundamental group of a {Kawamata} log terminal singularity
  is finite.
\newblock {\em Invent. Math.}, 2021.
\newblock \doi{10.1007/s00222-021-01062-0}.

\bibitem{BFMS20}
L.~Braun, S.~Filipazzi, J.~Moraga, and R.~Svaldi.
\newblock The {Jordan} property for local fundamental groups, 2020,
  \burlalt{arXiv:2006.01253}{http://arxiv.org/abs/arXiv:2006.01253}.

\bibitem{CT88}
D.~Cooper and W.~P. Thurston.
\newblock Triangulating {$3$}-manifolds using {$5$} vertex link types.
\newblock {\em Topology}, 27(1):23--25, 1988.
\newblock \doi{10.1016/0040-9383(88)90004-3}.

\bibitem{CLS11}
D.~A. Cox, J.~B. Little, and H.~K. Schenck.
\newblock {\em Toric varieties}, volume 124 of {\em Graduate Studies in
  Mathematics}.
\newblock American Mathematical Society, Providence, RI, 2011.
\newblock \doi{10.1090/gsm/124}.

\bibitem{Dur83}
A.~H. Durfee.
\newblock Neighborhoods of algebraic sets.
\newblock {\em Trans. Amer. Math. Soc.}, 276(2):517--530, 1983.
\newblock \doi{10.2307/1999065}.

\bibitem{FFMP22}
F.~Figueroa, S.~Filipazzi, J.~Peng, and J.~Moraga.
\newblock Complements and coregularity of fano varieties, 2022,
  \burlalt{arXiv:2211.09187}{http://arxiv.org/abs/arXiv:2211.09187}.

\bibitem{FMP22}
F.~Figueroa, J.~Moraga, and J.~Peng.
\newblock Log canonical thresholds and coregularity, 2022,
  \burlalt{arXiv:2204.05408}{http://arxiv.org/abs/arXiv:2204.05408}.

\bibitem{FMM22}
S.~Filipazzi, M.~Mauri, and J.~Moraga.
\newblock Index of coregularity zero {C}alabi--{Y}au pairs, 2022,
  \burlalt{arXiv:2209.02925}{http://arxiv.org/abs/arXiv:2209.02925}.

\bibitem{GKP16}
D.~Greb, S.~Kebekus, and T.~Peternell.
\newblock \'{E}tale fundamental groups of {K}awamata log terminal spaces, flat
  sheaves, and quotients of abelian varieties.
\newblock {\em Duke Math. J.}, 165(10):1965--2004, 2016.
\newblock \doi{10.1215/00127094-3450859}.

\bibitem{GHS16}
M.~Gross, P.~Hacking, and B.~Siebert.
\newblock Theta functions on varieties with effective anti-canonical class,
  2019, \burlalt{arXiv:1601.07081}{http://arxiv.org/abs/arXiv:1601.07081}.

\bibitem{Gro68}
A.~Grothendieck.
\newblock {\em Cohomologie locale des faisceaux coh\'{e}rents et
  th\'{e}or\`emes de {L}efschetz locaux et globaux {$(SGA$} {$2)$}}.
\newblock North-Holland Publishing Co., Amsterdam; Masson \& Cie, Editeur,
  Paris, 1968.
\newblock Augment\'{e} d'un expos\'{e} par Mich\`ele Raynaud, S\'{e}minaire de
  G\'{e}om\'{e}trie Alg\'{e}brique du Bois-Marie, 1962, Advanced Studies in
  Pure Mathematics, Vol. 2.

\bibitem{Har77}
R.~Hartshorne.
\newblock {\em Algebraic geometry}.
\newblock Graduate Texts in Mathematics, No. 52. Springer-Verlag, New
  York-Heidelberg, 1977.

\bibitem{KK14}
M.~Kapovich and J.~Koll\'{a}r.
\newblock Fundamental groups of links of isolated singularities.
\newblock {\em J. Amer. Math. Soc.}, 27(4):929--952, 2014.
\newblock \doi{10.1090/S0894-0347-2014-00807-9}.

\bibitem{KW11}
V.~Kapovitch and B.~Wilking.
\newblock Structure of fundamental groups of manifolds with {R}icci curvature
  bounded below, 2011,
  \burlalt{arXiv:1105.5955}{http://arxiv.org/abs/arXiv:1105.5955}.

\bibitem{KMS60}
A.~Karrass, W.~Magnus, and D.~Solitar.
\newblock Elements of finite order in groups with a single defining relation.
\newblock {\em Comm. Pure Appl. Math.}, 13:57--66, 1960.
\newblock \doi{10.1002/cpa.3160130107}.

\bibitem{Kol13b}
J.~Koll\'{a}r.
\newblock Links of complex analytic singularities.
\newblock In {\em Surveys in differential geometry. {G}eometry and topology},
  volume~18 of {\em Surv. Differ. Geom.}, pages 157--193. Int. Press,
  Somerville, MA, 2013.
\newblock \doi{10.4310/SDG.2013.v18.n1.a4}.

\bibitem{Kol13}
J.~Koll\'{a}r.
\newblock {\em Singularities of the minimal model program}, volume 200 of {\em
  Cambridge Tracts in Mathematics}.
\newblock Cambridge University Press, Cambridge, 2013.
\newblock \doi{10.1017/CBO9781139547895}.
\newblock With a collaboration of S\'{a}ndor Kov\'{a}cs.

\bibitem{KK10}
J.~Koll\'{a}r and S.~J. Kov\'{a}cs.
\newblock Log canonical singularities are {D}u {B}ois.
\newblock {\em J. Amer. Math. Soc.}, 23(3):791--813, 2010.
\newblock \doi{10.1090/S0894-0347-10-00663-6}.

\bibitem{Kol12}
J.~Kollár.
\newblock New examples of terminal and log canonical singularities, 2011,
  \burlalt{arXiv:1107.2864}{http://arxiv.org/abs/arXiv:1107.2864}.

\bibitem{Mil68}
J.~Milnor.
\newblock {\em Singular points of complex hypersurfaces}.
\newblock Annals of Mathematics Studies, No. 61. Princeton University Press,
  Princeton, N.J.; University of Tokyo Press, Tokyo, 1968.

\bibitem{Mor20b}
J.~Moraga.
\newblock Fano type surfaces with large cyclic automorphisms, 2020,
  \burlalt{arXiv:2001.03797}{http://arxiv.org/abs/arXiv:2001.03797}.

\bibitem{Mor20c}
J.~Moraga.
\newblock Kawamata log terminal singularities of full rank, 2021,
  \burlalt{arXiv:2007.10322}{http://arxiv.org/abs/arXiv:2007.10322}.

\bibitem{Mor21}
J.~Moraga.
\newblock On a toroidalization for klt singularities, 2021,
  \burlalt{arXiv:2106.15019}{http://arxiv.org/abs/arXiv:2106.15019}.

\bibitem{Mor22}
J.~Moraga.
\newblock Coregularity of {F}ano varieties, 2022,
  \burlalt{arXiv:2206.10834}{http://arxiv.org/abs/arXiv:2206.10834}.

\bibitem{MS21}
J.~Moraga and R.~Svaldi.
\newblock A geometric characterization of toric singularities, 2021,
  \burlalt{arXiv:2108.01717}{http://arxiv.org/abs/arXiv:2108.01717}.

\bibitem{Muk88}
S.~Mukai.
\newblock Finite groups of automorphisms of {$K3$} surfaces and the {M}athieu
  group.
\newblock {\em Invent. Math.}, 94(1):183--221, 1988.
\newblock \doi{10.1007/BF01394352}.

\bibitem{Mum61}
D.~Mumford.
\newblock The topology of normal singularities of an algebraic surface and a
  criterion for simplicity.
\newblock {\em Inst. Hautes \'{E}tudes Sci. Publ. Math.}, 9(9):5--22, 1961.
\newblock \urlprefix\url{http://www.numdam.org/item?id=PMIHES_1961__9__5_0}.

\bibitem{NW03}
W.~D. Neumann and J.~Wahl.
\newblock Universal abelian covers of quotient-cusps.
\newblock {\em Math. Ann.}, 326(1):75--93, 2003.
\newblock \doi{10.1007/s00208-002-0405-6}.

\bibitem{New68}
B.~B. Newman.
\newblock Some results on one-relator groups.
\newblock {\em Bull. Amer. Math. Soc.}, 74:568--571, 1968.
\newblock \doi{10.1090/S0002-9904-1968-12012-9}.

\bibitem{Nor83}
M.~V. Nori.
\newblock Zariski's conjecture and related problems.
\newblock {\em Ann. Sci. \'{E}cole Norm. Sup. (4)}, 16(2):305--344, 1983.
\newblock
  \urlprefix\url{http://www.numdam.org/item?id=ASENS_1983_4_16_2_305_0}.

\bibitem{Sch05}
K.~Schwede.
\newblock Gluing schemes and a scheme without closed points.
\newblock In {\em Recent progress in arithmetic and algebraic geometry}, volume
  386 of {\em Contemp. Math.}, pages 157--172. Amer. Math. Soc., Providence,
  RI, 2005.
\newblock \doi{10.1090/conm/386/07222}.

\bibitem{SZ01}
I.~Shimada and D.-Q. Zhang.
\newblock Classification of extremal elliptic {$K3$} surfaces and fundamental
  groups of open {$K3$} surfaces.
\newblock {\em Nagoya Math. J.}, 161:23--54, 2001.
\newblock \doi{10.1017/S002776300002211X}.

\bibitem{Sho93}
V.~V. Shokurov.
\newblock Semi-stable {$3$}-fold flips.
\newblock {\em Izv. Ross. Akad. Nauk Ser. Mat.}, 57(2):165--222, 1993.
\newblock \doi{10.1070/IM1994v042n02ABEH001541}.

\bibitem{TX17}
Z.~Tian and C.~Xu.
\newblock Finiteness of fundamental groups.
\newblock {\em Compos. Math.}, 153(2):257--273, 2017.
\newblock \doi{10.1112/S0010437X16007867}.

\bibitem{Xu14}
C.~Xu.
\newblock Finiteness of algebraic fundamental groups.
\newblock {\em Compos. Math.}, 150(3):409--414, 2014.
\newblock \doi{10.1112/S0010437X13007562}.

\end{thebibliography}
\end{document}